\documentclass[12pt]{elsarticle}
\usepackage{eurosym}
\usepackage{amsfonts}
\usepackage{amsmath}
\usepackage{mathtools}
\usepackage{subcaption}
\usepackage{xcolor}
\usepackage{array}
\usepackage{booktabs}
\usepackage{tikz}
\usepackage{pgfplots}
\usepackage{longtable}
\usepackage{soul}
\usepackage{lineno}
\usepackage[section]{placeins}

\usepgfplotslibrary{external} 
    \makeatletter
    \def\ps@pprintTitle{%
       \let\@oddhead\@empty
       \let\@evenhead\@empty
       \def\@oddfoot{\reset@font\hfil\thepage\hfil}
       \let\@evenfoot\@oddfoot
    }
    \makeatother

\newtheorem{theorem}{Theorem}

\newtheorem{definition}[theorem]{Definition}

\newtheorem{proposition}[theorem]{Proposition}
\newtheorem{remark}[theorem]{Remark}

\newenvironment{proof}[1][Proof]{\noindent\textbf{#1.} }{\ \rule{0.5em}{0.5em}}

\newcolumntype{M}[1]{>{\centering\arraybackslash}m{#1}}

\input{tcilatex}
\begin{document}

\begin{frontmatter}%

\title{A general linear method approach to the 
design and optimization of efficient, accurate, 
and easily implemented time-stepping methods in CFD}%

\author[ornl]{Victor DeCaria\fnref{utb}\corref{cor1}}
\ead{decariavp@ornl.gov}
\author[umass]{Sigal Gottlieb}
\ead{sgottlieb@umassd.edu}
\author[ornl]{Zachary J. Grant\fnref{utb}}
\ead{grantzj@ornl.gov}
\author[pitt]{William J. Layton}
\ead{wjl@pitt.edu}

\address[ornl]{Oak Ridge National Laboratory, Computational and Applied Mathematics Group}
\address[umass]{Mathematics Department, University of Massachusetts Dartmouth}
\address[pitt]{Department of Mathematics, University of Pittsburgh}

\fntext[utb]{This manuscript has been authored by UT-Battelle, LLC, under contract DE-AC05-00OR22725 with the US Department of Energy (DOE). The US government retains and the publisher, by accepting the article for publication, acknowledges that the US government retains a nonexclusive, paid-up, irrevocable, worldwide license to publish or reproduce the published form of this manuscript, or allow others to do so, for US government purposes. DOE will provide public access to these results of federally sponsored research in accordance with the DOE Public Access Plan (http://energy.gov/downloads/doe-public-access-plan).}
\cortext[cor1]{Corresponding author}

\address{}%

\begin{abstract}%

In simulations of fluid motion time accuracy has proven to be elusive. 
We seek highly accurate methods with strong enough stability properties to 
deal with the richness of scales of many flows. 
These methods must also be easy to implement within
current complex, possibly legacy codes. 
Herein we develop, analyze and test new time stepping methods
addressing these two issues with the goal of \textit{accelerating the
development of time accurate methods addressing the needs of applications}.
The new methods are created by introducing  inexpensive pre-filtering and post-filtering steps to popular 
methods which have been  implemented and tested within existing codes.  
We show that pre-filtering and post-filtering a multistep or multi-stage method results in new 
methods which have both multiple steps and stages: these are general linear methods (GLMs).
We utilize the well studied properties of GLMs to understand the accuracy and stability 
of filtered method, and to design optimal new filters for popular time-stepping methods. 
We present several new embedded families of high accuracy methods with low cognitive
complexity and excellent stability properties. Numerical tests of the
methods are presented, including ones finding failure points of some
methods. Among the new methods presented is a 
novel pair of alternating filters for the Implicit Euler method
which induces a third order, A-stable, error inhibiting scheme which is shown to be
particularly effective.

\end{abstract}%

\end{frontmatter}%

\section{Introduction}

\label{Intro}

\label{introduction} There is significant  cognitive complexity required for understanding, implementing
and validating new methods in complex, possibly legacy, codes. This results in  a  
need for improved methods that can be easily implemented through small
modifications of  simpler and well-tested codes.
Herein, we develop high order timestepping methods with favorable stability
properties that can be implemented by adding a minimal code modification of
a few, O(2), extra lines to simple methods often used in legacy codes. 
The newly developed methods often do not require additional function evaluations
or  extra storage, as the variables are simply over-written. 
Some of these methods provide an embedded error estimator, have
natural extensions to variable timesteps and arise from a process, Section %
\ref{sec:optimization}, \ that is amenable to optimization with respect to
applications driven design criteria.

To motivate the rest of the paper and provide useful methods, we now give three
(constant $\Delta t$) examples of the new methods as time discretizations of
the incompressible Navier–Stokes equations (NSE) (the application used to test the methods
in Section \ref{sec:NumericalTests}). The NSE are 
\begin{equation}
u_{t}+u\cdot \nabla u-\nu \Delta u+\nabla p=f\text{ and }\nabla \cdot u=0.
\label{eq:NSE}
\end{equation}%
Here, $u$ is the velocity, $p$ is the pressure, $\nu $ is the kinematic
viscosity, and $f=f(x,t)$ is an external non-autonomous body force. The two
equations describe the conservation of momentum and mass, respectively. We
set $f=0$ for simplicity in these examples.

\textbf{Example 1:} Beginning with the usual fully implicit Euler (IE)
method, let $\Delta t$ denote the timestep and superscript the timestep
number. Suppressing the spatial discretization, the IE method for the
Navier-Stokes equations (NSE) is: 
\begin{subequations}
\label{eq:BEstep}
\begin{equation}
\frac{u^{n+1}-u^{n}}{\Delta t}+u^{n+1}\cdot \nabla u^{n+1}-\nu \Delta
u^{n+1}+\nabla p^{n+1}=0\text{ and }\nabla \cdot u^{n+1}=0.
\end{equation}%
Adding two extra lines of code yields the following new method: 
\end{subequations}
\begin{equation*}
\begin{array}{lc}
\text{1) Pre-filter } & \tilde{u}^{n}=u^{n}-\frac{1}{2}\left(
u^{n}-2u^{n-1}+u^{n-2}\right) , \vspace{0.15in}\\ 
\text{2) IE Solve:} & \left\{ 
\begin{array}{c}
\frac{1}{\Delta t}u^{n+1}+u^{n+1}\cdot \nabla u^{n+1}-\nu \Delta
u^{n+1}+\nabla p^{n+1}=\frac{1}{\Delta t}\tilde{u}^{n}\text{ } \\ 
\text{ }\nabla \cdot u^{n+1}=0%
\end{array}%
\right. \vspace*{.2in}  \\ 
\vspace*{.1in}
\text{3) Post-filter } & u^{n+1}_{3rd} = u^{n+1}-\frac{5}{11}\left(
u^{n+1}-3u^{n}+3u^{n-1}-u^{n-2}\right) ,%
\end{array}%
\end{equation*}%
The method arising by stopping after Step 2 is second order accurate and
L-stable, Section \ref{sec:bepre}, and is referred to herein as IE-Pre-2.
The method after Step 3 is third order accurate and  $A(\alpha )$  stable
with $\alpha \simeq 71^{\circ }$, Section \ref{sec:beprepost}, and is
referred to as IE-Pre-Post-3. Thus the difference between the Steps 2 and 3
approximation is an estimator of the local truncation error\footnote{%
The pair of approximations allows implementation as a variable order method,
not developed herein. }. We stress that the pre and post-filters are
themselves comprehensible (not exotic) operations and reduce 
the discrete fluctuation of the numerical solution.

Steps 2 and 3 are time filters designed to have no effect on (some
realization of) smooth solution scales and damp (some realization of)
fluctuating scales. For example, for the prefilter $\tilde{u}^{n} = u^{n}-%
\frac{1}{2}\left( u^{n}-2u^{n-1}+u^{n-2}\right) $, if $u^{n}=a+bt^{n}$ then
the filter does not alter $u^{n}.$ Let $\kappa =u^{n}-2u^{n-1}+u^{n-2}$
denote the discrete curvature, then the filter also has the effect of
halving the discrete curvature: $\kappa _{post}=\frac{1}{2}\kappa_{pre}$. 
The post-filter is a higher order realization of this
process. The coefficient values, 1/2 and 5/11, are derived by applying the
order conditions in the general linear method induced by the pre and post
filtered method, Section 3.

\textbf{Example 2:} One equivalent realization\footnote{%
The next steps would be reorganized for a different implementation of the
method.} of the usual implicit midpoint (1-1 Pad\'{e}, Crank-Nicolson,
one-leg trapezoidal) method is: given $u^{n}$ 
\begin{subequations}
\label{eq:MPstep}
\begin{gather}
\left\{ 
\begin{array}{c}
\frac{1}{\frac{1}{2}\Delta t}u^{n+\frac{1}{2}}+u^{n+\frac{1}{2}}\cdot \nabla
u^{n+\frac{1}{2}}-\nu \Delta u^{n+\frac{1}{2}}+\nabla p^{n+\frac{1}{2}}=%
\frac{1}{\frac{1}{2}\Delta t}u^{n}\text{ ,} \\ 
\text{ }\nabla \cdot u^{n+\frac{1}{2}}=0,%
\end{array}%
\right.  
\vspace*{.1in} \\
\vspace*{.1in}
u^{n+1}=2u^{n+\frac{1}{2}}-u^{n}.
\end{gather}%
We abbreviate this as MP. The pre- and post-filtered methods in Section \ref{sec:imp_midpoint} 
require introducing temporary variables $\tilde{u}^{n},%
\tilde{u}^{n+1}$ (that can be overwritten each time step) and produces an
embedded triplet of approximations of 2nd, 3rd and 4th order $u_{2\text{nd}%
}^{n+1},u_{3\text{rd}}^{n+1},u_{4\text{th}}^{n+1}$. Naturally one must be
selected to be $u^{n+1}$. The method is 
\end{subequations}
\begin{equation*}
\begin{array}{lc}
\text{1) Pre-filter} & \tilde{u}^{n}=u^{n}+\frac{5}{6}\left( u^{n}-\frac{3}{2%
}u^{n-1}+\frac{3}{5}u^{n-2}-\frac{1}{10}u^{n-3}\right) 
\vspace{.15in} \\ 
\text{2) IM Solve} & \left\{ 
\begin{array}{c}
\frac{1}{\frac{1}{2}\Delta t}u^{n+\frac{1}{2}}+u^{n+\frac{1}{2}}\cdot \nabla
u^{n+\frac{1}{2}}-\nu \Delta u^{n+\frac{1}{2}}+\nabla p^{n+\frac{1}{2}}=%
\frac{1}{\frac{1}{2}\Delta t} \tilde{u}^{n},\text{ } \\ 
\text{ }\nabla \cdot u^{n+\frac{1}{2}}=0, \\ 
u^{n+1}=2u^{n+\frac{1}{2}}-\tilde{u}^{n}\text{ }%
\end{array}%
\right.
\end{array}
 \vspace{2mm} 
\end{equation*}%

\begin{equation*}
\begin{array}{lc}
\text{3) Post-filters} & u_{3\text{rd}}^{n+1}=\frac{1}{2}\left( u^{n+1}+%
\tilde{u}^{n}\right) \vspace{2mm} \\ 
& u_{2\text{nd}}^{n+1}=\frac{12}{11}u_{3\text{rd}}^{n+1}-\frac{7}{22}u^{n}+%
\frac{9}{22}u^{n-1}-\frac{5}{22}u^{n-2}+\frac{1}{22}u^{n-3},\vspace{2mm} \\ 
& u_{4\text{th}}^{n+1}=\frac{24}{25}u_{3\text{rd}}^{n+1}+\frac{4}{25}u^{n}-%
\frac{6}{25}u^{n-1}+\frac{4}{25}u^{n-2}-\frac{1}{25}u^{n-3}.%
\end{array}%
\end{equation*}%
The 2nd order approximation, MP-Pre-Post-2, is $A$-stable, and has the same
linear stability region as MP. The 3rd order approximation, MP-Pre-Post-3,
is $A(\alpha )$\ with $\alpha =79.4^{\circ }$\ while the fourth order
approximation, MP-Pre-Post-4, is $A(\alpha )$\ with $\alpha =70.64^{\circ }$%
, Section \ref{sec:imp_midpoint}. Their differences provide an estimator,
which makes MP-Pre-Post-2/3/4 potentially useful as the basis for a variable stepsize variable order (VSVO)
method. Other examples are provided building on BDF2 (Section \ref{sec:bdf2}%
) and Runge-Kutta methods (Section \ref{sec:rk22}). The above is the natural
implementation without requiring additional function evaluations.

\textbf{Example 3:} Looking at the pre- and post-filtering process as a general linear method (GLM)
also opens the door to creating methods that have the error inhibiting properties described in \cite{EISpp20}.
An example of this is the EIS method IE-EIS-3 which we will describe in \ref{sec:EISmethod}:
\begin{equation*}
\begin{array}{lc}
\text{1) Pre-filter} &  \tilde{u}^n =   \frac{23}{5} u^{n-\frac{1}{3}}  - 3  u^n - \frac{9}{5} \tilde{u}^{n-1} + \frac{6}{5} \hat{u}^{n-1}
\vspace*{.1in} \\
\text{2) IE Solve:} & \left\{ 
\begin{array}{c}
\frac{1}{\Delta t}u^{n+\frac{2}{3}}+u^{n+ \frac{2}{3} }\cdot \nabla u^{n+ \frac{2}{3}}-\nu \Delta
u^{n+\frac{2}{3}}+\nabla p^{n+\frac{2}{3}}=\frac{1}{\Delta t}\tilde{u}^{n}\text{ } \\ 
\text{ }\nabla \cdot u^{n+\frac{2}{3}}=0%
\end{array}%
\right. \vspace*{.1in}  \\ 
\text{3) Filter} & \hat{u^n} = \frac{5}{12} u^{n} -  \frac{1}{12} u^{n+\frac{2}{3}} -  \frac{5}{12}  \hat{u}^{n-1} +  \frac{13}{12} \tilde{u}^n
 \vspace*{.1in}  \\
\text{4) IE Solve:} & \left\{ 
\begin{array}{c}
\frac{1}{\Delta t}u^{n+1}+u^{n+1}\cdot \nabla u^{n+1}-\nu \Delta
u^{n+1}+\nabla p^{n+1}=\frac{1}{\Delta t}\hat{u}^{n}\text{ } \\ 
\text{ }\nabla \cdot u^{n+1}=0%
\end{array}%
\right.   \\ 
\end{array}
\end{equation*}
This method is third order and A-stable, and has the cost of two implicit Euler evaluations.

\textbf{\ Method design and analysis.} In the methods' derivation and
implementation, extra variables are introduced. In their analysis and design the
extra variables are condensed to obtain an equivalent single method. With
pre and post filtering, this equivalent method can be both multi-step and
multi-stage. Thus its optimization and analysis must be through the theory
of general linear methods, Section 2. Section 3 shows how this theory can be
used to design filters so that the induced method satisfies desired
optimality conditions. In Section \ref{sec:GLMdesign} we will show how to
rewrite time-filtering methods as general linear methods, and give some
examples of such time-filtered methods as GLMs. In Section \ref{optimization}
we will show how this formulation can be used to develop an optimization
code that is a powerful tool that allows us to find pre- and post- processed
methods with advantageous stability properties. We will demonstrate the
utility of this approach showing some methods that resulted from this
optimization code (Section \ref{newmethods}). \ Section 4 applies Section 2
and 3 to derive the  methods in the section above and several more based on design
criteria of stability, accuracy and an embedded algorithmic structure. These
new methods are embedded families of high order methods based on
combinations of pre- and post-filtering for the most commonly used methods
for time discretization of incompressible flows, including the fully
implicit Euler method, the midpoint rule, and the BDF2 method. (We stress
however that the theory can be applied to a wide variety of other
applications driven design criteria.) The tools developed in Sections 2 and
3 thus show \textit{how to accelerate the development of time accurate
methods addressing the needs of applications}. Finally, we study the
performance of some of these methods in Section \ref{numerical}. The higher
Reynolds flow problems in Section \ref{sec:NumericalTests} were selected
because they are nonlinearity dominated and have solutions rich in scales, so accuracy is difficult and 
stability is essential. 

\subsection{Related work}

Time filters are widely used to stabilize leapfrog time discretizations of
atmosphere models, e.g., \cite{R69}, \cite{A72}, \cite{W09}. Our study of this
important work, in particular its stress on balancing computational, space
and cognitive complexities, was critical for the development path herein. In 
\cite{GL18} it was shown a well calibrated post-filter can increase accuracy
in the fully implicit method to second order, preserve A-stability,
anti-diffuse the approximation and yield an error estimator useful for time
adaptivity. The new methods and analysis in \cite{GL18} for $y^{\prime}=f(t,y)$
were extended to the Navier-Stokes equations in \cite{DLZ18}. In 
\cite{DGLL18} post-filters were studied starting with BDF3 (yielding an
embedded family of orders 2,3,4). It was also proven that post-filtering
alone has an accuracy barrier: improvement by at most one power of $\Delta t$
is possible from any sequence of linear post-filters. The idea of adding a
prefilter step herein is simple in principle (though technically intricate
in analysis and design, Sections 2,3) and overcomes this accuracy barrier.

There is a significant body of detailed and technically intricate stability
and convergence analysis for CFD problems of, mostly simpler (e.g., IE, IM,
BDF2) and mostly constant timestep methods, including \cite{GR79}, 
\cite{Crouzeix1976}, \cite{BDK82}, \cite{JMRT17}, \cite{E04a}, \cite{E04b}, \cite{J15}. 
Important work on adaptive timestepping for similar problems occurs
in \cite{KGGS10}, \cite{VV13}, \cite{HEPG15}, \cite{JR10}. The embedded
structure of the new method families herein suggest their further
development into adaptive methods.

\section{Examples: Time Filters induce General Linear Methods (GLMs)}

\label{Motivation}

Writing the method induced by adding pre- and post-filters in the form of a
GLM allows us to apply the order conditions and stability theory of GLMs to
optimize the method. To illustrate how a GLM is induced, for 
\begin{equation*}
u_{t}=F(u)
\end{equation*}%
we consider the pre- and post-filtered implicit Euler method (Example 1 in  Section
1) which can be written in the form
\begin{subequations}
\label{BEprepost3}
\begin{eqnarray}
y^{(1)} &= &u^{n}-\frac{1}{2}\left( u^{n}-2u^{n-1}+u^{n-2}\right) \\
y^{(2)} &=& y^{(1)} +\Delta  tF\left( y^{(2)} \right) \\
u^{n+1} &=& y^{(2)}  - \frac{5}{11}\left( y^{(2)} -3u^{n}+3u^{n-1}-u^{n-2}\right)
\end{eqnarray}

GLMs represent any combination of multistep and multistage methods. A GLM
with $s$ stages and $k$ steps is 
\end{subequations}
\begin{subequations}
\label{eq:mrktypeII}
\begin{align}
y^{(i)}& =\sum_{\ell =1}^{k}d_{i\ell }u^{n-k+\ell }+\Delta t\sum_{\ell
=1}^{k-1}\hat{a}_{i\ell }F(u^{n-k+\ell })+\Delta
t\sum_{j=1}^{s}a_{ij}F(y^{(j)})\;\;\;\;1\leq i\leq s \\
u^{n+1}& =\sum_{\ell =1}^{k}\theta _{\ell }u^{n-k+\ell }+\Delta t\sum_{\ell
=1}^{k-1}\hat{b}_{\ell }F(u^{n-k+\ell })+\Delta
t\sum_{j=1}^{s}b_{j}F(y^{(j)}),
\end{align}%
where the $u^{n-k+l}$ denote the steps, while the $y_{j}^{n}$ are
intermediate stages used to compute the next solution value $u^{n+1}$. We
will refer to these coefficients more compactly by the matrices $D,A,\hat{A}$
given by 
\end{subequations}
\begin{equation*}
D_{i\ell }=d_{i\ell },\;\;\;\hat{A}_{i\ell }=\hat{a}_{i\ell },\;\;\;%
\mbox{and}\;\;\;A_{ij}={a}_{ij},
\end{equation*}%
and the vectors $\Theta ,\mathbf{b},\mathbf{\hat{b}}$ that are given by 
\begin{equation*}
\Theta _{\ell }=\theta _{\ell },\;\;\;\mathbf{\hat{b}}_{\ell }=\hat{b}_{\ell
},\;\;\;\mbox{and}\;\;\;\mathbf{b}_{j}={b}_{j}.
\end{equation*}

The method \eqref{BEprepost3} is  in the form \eqref{eq:mrktypeII} with
\begin{equation*}
D=\left( -\frac{1}{2},1,\frac{1}{2}\right) ,\;\;\;\theta =\left( \frac{2}{11}%
,-\frac{9}{11},\frac{18}{11}\right) ,\;\;\;\hat{A}=\left( 0,0,0\right)
,\;\;\;\hat{b}=\left( 0,0,0\right) ,\;\;\;
\end{equation*}%
and $A=1$, $b_{1}=\frac{6}{11}.$
Note that we can write this method as a pair of {\bf embedded second and third order methods}. 
Only pre-filtering the implicit Euler method yields the 
$\mathcal{O}(\Delta t^{2})$\ \ accurate\footnote{%
All accuracy statements are proven by checking the order conditions for GLMs
presented in Section 3.} method 
\begin{eqnarray*}
y^{(1)} &=&-\frac{1}{2}u^{n-2}+u^{n-1}+\frac{1}{2}u^{n} \\
u^{n+1}_{2nd} &=&y^{(1)}+\Delta tF\left( u^{n+1}_{2nd}  \right) .
\end{eqnarray*}%
Adding a final postprocessing line 
\begin{eqnarray}
u^{n+1}_{3rd} &=&\frac{5}{11} u^{n-2}-\frac{15}{11} u^{n-1}+\frac{15}{11} u^{n}+\frac{6}{11} u^{n+1}_{2nd} 
\end{eqnarray}
makes the approximation $\mathcal{O}(\Delta t^{3})$. This embedded structure
means that the difference between the velocities $u^{n+1}_{3rd}$ and $u^{n+1}_{2nd}$ can
be used as an error estimator. 

\textbf{A 1-parameter family of pre- and post-filtered methods.} Adding a
parameter to the pre- and post- filters allows optimization with respect to
accuracy, stability and error constant criteria in Section 3. For the
implicit Euler method the parametric family is 
\begin{subequations}
\label{BEprepost2_family}
\begin{eqnarray}
y^{(1)} &=&u^{n}-d\left( u^{n}-u^{n-1}\right) \\
y^{(2)} &=&y^{(1)}+\Delta tF\left( y^{(2)}\right) \\
u^{n+1} &=&y^{(2)}-\frac{1}{3-2d}\left\{ (1-2 d )y^{(2)}-2(1-d)u^{n}+u^{n-1}\right\} .
\end{eqnarray}%
This method can be written as a GLM by 
\end{subequations}
\begin{eqnarray*}
y^{(1)} &=&du^{n-1}+(1-d)u^{n} \\
y^{(2)} &=&y^{(1)}+\Delta tF\left( y^{(2)}\right) \\
u^{n+1} &=&\frac{2d-1}{3-2d}u^{n-1}+\frac{4(1-d)}{3-2d}u^{n}+\frac{2}{3-2d}\Delta tF(y^{(2)}).
\end{eqnarray*}%
which is of GLM form with the choices $A=1$, $b_{1}=\frac{2}{3-2d}$
and 
\begin{equation*}
D=\left( 
\begin{array}{cc}
d& 1-d\\ 
d& 1- d %
\end{array}%
\right) ,\;\;\;\theta =\left( \frac{2d-1}{3-2d},\frac{4(1-d)}{3-2d}\right) ,
\;\;\;\hat{A}=\left( 0,0\right) ,\;\;\;\hat{b}=\left(0,0\right) .\;\;\;
\end{equation*}

Clearly, it is possible to develop easily implemented methods with many free
parameters by pre- and post-filtering widely used methods. Next, in Section %
\ref{sec:GLMdesign}, we show that the induced GLM form can simplify,
automate and accelerate filter, and thus method design.

\section{Optimizing Time Filters using the GLM framework}
\label{sec:GLMdesign}

Section 2 illustrates that adding two lines of time filter code to commonly
used linear multistep methods can increase accuracy without significant
additional computational work. To design the filter required, it is
necessary to write the filtered methods in the GLM form \eqref{eq:mrktypeII}. 
This section develops the GLM\ form of filtered general GLMs, shows how to
use this in a method design process and generalizes the filtering process to
also include, previously computed, trend values (i.e. values of $F(y^{(j)})$).

Subsection \ref{GLMasGLM} shows that pre- and
post-filtering \emph{any} GLM, whether a multistep, Runge--Kutta, or a
combination of these, induces another GLM whose properties depend on the
filter parameters in a precise way. Theorem \ref{filteredGLM} gives a
complete characterization of the coefficients of the filtered GLM that
results from pre- and/or post-filtering a core GLM method in terms of the
filter parameters and the coefficients of the core method. Subsection \ref%
{GLMbackground} shows how to determine the accuracy and stability properties
of the method so induced. The results in the first three subsections open
the possibility, developed in Subsection \ref{optimization}, of \textit{%
optimizing accuracy and stability properties of the filtered GLM over the
choices of the filter parameters}. The optimization algorithm in Subsection %
\ref{optimization} is then used to design methods with favorable stability
and accuracy properties. We will present the new methods found using this
optimization in Section \ref{newmethods}.

\subsection{Pre and post filtering a GLM}
\label{GLMasGLM}

In the section above, we showed how time-filtering a linear multistep method
can be written as a GLM. Filtering is also useful for multi-stage
(Runge--Kutta) methods and multistep multi-stage methods, GLMs. If the core
method is a linear multistep method as in \eqref{LMM} the section above then
we have one stage only and $k$ steps. If the core method is a Runge--Kutta
method then we have $k=1$ steps and $s$ stages. In more generality, we
consider here a core method with $s$ stages and $k$ steps, defined by the
coefficients $\check{d}_{i\ell },\hat{a}_{i\ell },a_{ij}$ for $i=2,...,s$, $%
j=2,...,s$ and $\ell =1,...,k$:

\begin{eqnarray}
y^{(1)} &=&u^{n}  \notag  \label{coreGLM} \\
y^{(i)} &=&\sum_{\ell =1}^{k}\check{d}_{i\ell }u^{n-k+\ell }+\Delta
t\sum_{\ell =1}^{k-1}\hat{a}_{i\ell }F(u^{n-k+\ell })+\Delta
t\sum_{j=1}^{s}a_{ij}F(y^{(j)}) \\
&&\hspace{1in}\mbox{for}\;\;\;2\leq i\leq s  \notag \\
u^{n+1} &=&y^{(s)}.  \notag
\end{eqnarray}%
We can write this in the form \eqref{eq:mrktypeII} where the final row
coefficients are the same as the prior row coefficients. In the\ last
theorem we limited the form of the post-filter to be \eqref{post} and not
include any function evaluations. However, the post filter does not have to
be limited to the form \eqref{post}, but in fact can have $\theta _{l},\hat{b%
}_{l},b_{j}$ as free parameters, chosen only to satisfy stability and
accuracy considerations. This does not impose much additional cost to the
method, as any function evaluations have already been computed by the final
stage.

The following theorem provides the relationship between the core method, the
filter parameters and the filtered method.

\begin{theorem}
\label{filteredGLM} If we filter a GLM of the form \eqref{coreGLM} we obtain
a GLM of the form \eqref{eq:mrktypeII} where the first stage is a pre-filter
is given by 
\begin{subequations}
\begin{equation}  \label{GLMpre}
y^{(1)}=\sum_{\ell=1}^{k} d_{1\ell}u^{n-k+\ell}
\end{equation}
and the final stage is a post filter given by 
\begin{equation}  \label{GLMpost}
u^{n+1} = \sum_{\ell=1}^{k} \theta_\ell u^{n-k+\ell} + \Delta
t\sum_{\ell=1}^{k-1} \hat{b}_{\ell} F(u^{n-k+\ell}) + \Delta t\sum_{j=1}^s
b_j F(y^{(j)}).
\end{equation}
The coefficients $d_{1\ell}$ of \eqref{GLMpre} are the pre-filter
coefficients, and the coefficients $\theta_\ell, \hat{b}_{\ell}, b_j $ in %
\eqref{GLMpost} are the post-filter coefficients. These do not depend on the
coefficients of the core method \eqref{coreGLM} and can be chosen freely,
subject only to order and stability constraints.

The middle stages $2\leq i\leq s$ of the filtered method have the form 
\end{subequations}
\begin{equation*}
y^{(i)} = \sum_{\ell=1}^{k} {d}_{i\ell} u^{n-k+\ell}+ \Delta t
\sum_{\ell=1}^{k-1} \hat{a}_{i\ell} F(u^{n-k+\ell})+\Delta t
\sum_{j=1}^{s}a_{ij} F(y^{(j)}) ,
\end{equation*}
where coefficients $a_{ij}$ and $\hat{a}_{il}$ are not impacted by the
filtering, and the coefficients ${d}_{i\ell}$ of the filtered methods are
related to the coefficients $\check{d}_{i\ell}$ of the core method %
\eqref{coreGLM} by 
\begin{subequations}
\begin{align}  \label{PPcoef}
d_{i\ell}& =\check{d}_{i1}d_{1\ell}+\check{d}_{i\ell}\;\;\;\mbox{for}\;\;
\ell=1,k-1 \\
d_{ik}& =\check{d}_{i1}d_{1k}.
\end{align}
\end{subequations}
\end{theorem}

\begin{proof}
Clearly, the pre-filter coefficients $d_{1\ell }$ in \eqref{GLMpre} can be
freely chosen, subject only to accuracy and stability considerations. We
then use $y^{(1)}$ instead of $u^{n}$ in all the middle stages of %
\eqref{coreGLM}: 
\begin{eqnarray*}
y^{(i)} &=&\tilde{d}_{i1}y^{(1)}+\sum_{l=1}^{k-1}\tilde{d}_{i\ell
}u^{n-k+\ell }+\Delta t\sum_{\ell =1}^{k-1}\hat{a}_{i\ell }F(u^{n-k+\ell
})+\Delta t\sum_{j=1}^{s}a_{ij}F(y^{(j)}) \\
&=&\tilde{d}_{i1}\sum_{\ell =1}^{k}d_{1\ell }u^{n-k+\ell }+\sum_{\ell
=1}^{k-1}\tilde{d}_{i\ell }u^{n-k+\ell }+ \\
&&\text{ \ \ \ \ }+\Delta t\sum_{\ell =1}^{k-1}\hat{a}_{i\ell }F(u^{n-k+\ell
})+\Delta t\sum_{j=1}^{s}a_{ij}F(y^{(j)}) \\
&=&\sum_{\ell =1}^{k}{d}_{i\ell }u^{n-k+\ell }+\Delta t\sum_{\ell =1}^{k-1}%
\hat{a}_{i\ell }F(u^{n-k+\ell })+\Delta t\sum_{j=1}^{s}a_{ij}F(y^{(j)}),
\end{eqnarray*}%
where the coefficients of the pre-processed scheme are related to the
original coefficients by: 
\begin{equation*}
d_{i\ell }=\tilde{d}_{i1}d_{1\ell }+\tilde{d}_{i\ell }\;\;\;\mbox{for}%
\;\;\ell =1,k-1,\;\;\;\mbox{and}\;\;\;d_{ik}=\tilde{d}_{i1}d_{1k}
\end{equation*}%
To post-process the method, we simply modify the final line \eqref{GLMpost}
by allowing the postprocessing coefficients $\theta _{\ell },\hat{b}_{\ell
},b_{j}$ to be chosen freely, subject only to order conditions and stability
considerations.
\end{proof}

\begin{remark}
By placing additional constraints on the post-filter coefficients, methods
may be derived where the function evaluations in \eqref{GLMpost} are
replaced by a linear combination of previously computed $u^{n-k+l}$ and
stages $y^{(j)}$. This may be desirable depending on the problem, computer
architecture, and availability of function evaluations from a possibly black
box solver.
\end{remark}

\subsubsection{Filtering a linear multistep method}
Consider a $k$-step linear multistep method: 
\begin{equation}
u^{n+1}=\sum_{\ell =1}^{k}\alpha _{\ell }u^{n-k+\ell }+\Delta t\sum_{\ell=1}^{k+1}\beta _{\ell }F(u^{n-k+\ell }).  \label{LMM}
\end{equation}%
We pre-filter (Eqn. \eqref{pre}) and post-filter (Eqn. \eqref{post} ) as
in Section \ref{Motivation}: 
\begin{subequations}
\label{filtered_multistep}
\begin{eqnarray}
\hat{u}^{n} &=&u^{n}-\frac{\alpha }{2}\sum_{\ell =1}^{k}\hat{d}%
_{l}u^{n-k+\ell }  \label{pre} \\
\hat{u}^{n+1} &=&\alpha _{k}\hat{u}^{n}+\sum_{\ell =1}^{k-1}\alpha
_{l}u^{n-k+\ell }+\Delta t\sum_{\ell =1}^{k-1}\beta _{\ell }F(u^{n-m+\ell })
\label{core} \\
&+&\Delta t\beta _{k}F(\hat{u}^{n})+\Delta t\beta _{k+1}F(\hat{u}^{n+1}) 
\notag \\
{u}^{n+1} &=&\hat{u}^{n+1}-\frac{\omega }{2}\sum_{\ell =1}^{k}\hat{q}_{\ell
}u^{n-k+\ell }.  \label{post}
\end{eqnarray}
\end{subequations}
Following Theorem \ref{filteredGLM}, the coefficients of the filtered methods
satisfy:
\begin{equation*}
d_{1k}=1-\frac{\alpha }{2}\hat{d}_{k}, \
d_{2k}=\alpha _{k}d_{1k} \,
 \; \; \mbox{and}\;\;\;
d_{1\ell }=-\frac{\alpha }{2}\hat{d}_{\ell },
d_{2\ell }=\alpha _{k}d_{1\ell}+\alpha _{l}\
\;\;\;\mbox{for $\ell <k$}.
\end{equation*}%
\begin{equation*}
\hat{a}_{2\ell }=\beta _{l}\;\;\mbox{for $\ell \leq k-1$}, \; \; \mbox{and} \; \;
a_{21}=\beta _{k}, \; a_{22}=\beta _{k+1}.
\end{equation*}%

Generally, the post-filtering coefficients $\theta _{\ell }$ and $\hat{b}_{\ell }$ for $\ell \leq k$, and $%
b_{j}$ for $j=1,2$ can be chosen freely subject only to order and stability
considerations. However, if we wish to limit our post-filtering to the form %
\eqref{post}, which does not include function evaluations, then we have 
\begin{equation}
\theta _{\ell }=d_{2l}-\frac{\omega }{2}\hat{q}_{\ell },\;\;\;\hat{b}_{l}=%
\hat{a}_{2\ell },\;\;\;\mbox{and}\;\;\;b_{j}=a_{2j}.
\end{equation}

Notice that in the pre-filter \eqref{pre}, for consistency $\sum \hat{d}_{l}u^{n-k+l}$ 
should annihilate polynomials up to a certain (non-zero) degree so that this quantity is related to a
discrete derivative. In particular, for $u^{n}\equiv 1$ this implies $%
\sum_{l=1}^{k}\hat{d}_{l}=0$. In general, the quantity \ $\sum_{l=1}^{k}\hat{%
d}_{l}u^{n-k+l}$ represents a discrete fluctuation. Pre-filtering acts to
reduce the discrete fluctuation of the numerical solution.

\begin{proposition}
If 
\begin{equation*}
0<\alpha \hat{d}_{k}<2,
\end{equation*}%
then the filter 
\begin{equation*}
{\LARGE \ } u_{pre}^{n}=u^{n}-\frac{\alpha }{2}\sum_{\ell=1}^{k}\hat{d}_{l}u^{n-k+l}
\end{equation*}%
strictly reduces without changing sign the discrete fluctuation $\sum \hat{d}_{\ell}u^{n-k+l}$%
\begin{equation*}
\left( \hat{d}_{k}u_{pre}^{n}+\sum_{l=1}^{k-1}\hat{d}_{l}u^{n-k+\ell}\right)
=\left( 1-\frac{\alpha \hat{d}_{k}}{2}\right) \left( \sum_{\ell=1}^{k}\hat{d}_{\ell}u^{n-k+\ell}\right)
\end{equation*}
\end{proposition}

\begin{proof}
Multiply by $\hat{d}_{k}$ to obtain 
\begin{equation*}
\hat{d}_{k}u_{pre}^{n}=\hat{d}_{k}u^{n}-\frac{\alpha \hat{d}_{k}}{2}%
\sum_{\ell =1}^{k}\hat{d}_{\ell }u^{n-k+\ell }.
\end{equation*}%
Add $\sum_{\ell =1}^{k-1}\hat{d}_{\ell }u^{n-k+\ell }$ to both sides: 
\begin{eqnarray*}
\left( \hat{d}_{k}u_{pre}^{n}+\sum_{\ell =1}^{k-1}\hat{d}_{\ell }u^{n-k+\ell
}\right)  &=&\left( \hat{d}_{k}u^{n}+\sum_{\ell =1}^{k-1}\hat{d}_{\ell}u^{n-k+\ell }\right)  
- \frac{\alpha \hat{d}_{k}}{2}\left( \sum_{\ell =1}^{k}\hat{d}_{\ell }u^{n-k+\ell }\right)  \\
&=&\left( 1-\frac{\alpha \hat{d}_{k}}{2}\right) \left( \sum_{\ell =1}^{k}%
\hat{d}_{\ell }u^{n-k+\ell }\right) 
\end{eqnarray*}%
which establishes the result.
\end{proof}

\subsection{Background on General Linear Methods}

\label{GLMbackground}

To analyze the order and stability of \eqref{eq:mrktypeII} we convert to the
compact form 
\begin{subequations}
\label{eq:msrktypeI}
\begin{eqnarray}
y^{(i)} &=&u^{n-k+i}\quad \quad \mbox{for i =1 ..., k-1} \\
y^{(i)} &=&\sum_{\ell =1}^{k}{d}_{i\ell }u^{n-k+\ell }+\Delta t\sum_{j=1}^{k+s}%
\tilde{A}_{ij}F(y^{(j)})\quad \mbox{for i =k ..., k+s} \\
u^{n+1} &=&\sum_{\ell =1}^{k}\;\theta _{\ell }u^{n-k+\ell }+\Delta
t\sum_{j=1}^{k+s}\;\tilde{b}_{j}F(y^{(j)})
\end{eqnarray}%
where 
\end{subequations}
\begin{equation*}
\tilde{\mathbf{A}}=\left[ 
\begin{array}{cc}
\multicolumn{2}{c}{\mathbf{0}_{(k-1)\times (s+k-1)}} \\ 
\hat{A} & A \\ 
& 
\end{array}%
\right] ,\;\;\;\;\tilde{\mathbf{D}}=\left[ 
\begin{array}{c}
I_{(k-1)\times (k-1)} \; \; \mathbf{0}_{(k-1)\times (1)} \\ 
D \\ 
\end{array}%
\right] ,\;\;\;\;\mathbf{\tilde{b}}=\left[ \mathbf{\hat{b}}\;\;\;\;\;\mathbf{%
b}\right]
\end{equation*}%
are matrices of dimension $(s+k-1)\times (s+k-1)$, $(s+k-1)\times k$, and $%
1\times (s+k-1)$, respectively. This compact form is obtained by defining
the first few stages to be older steps. The consistency and stability
properties of GLMs are delineated in \cite%
{Butcher1966,ButcherGLM,HairerWanner1987}, reviewed next.

\begin{definition}
\label{OrderConditions} Let $\mathbf{e}$ denote a column vector of $1$'s,
and define the vectors 
\begin{equation*}
\ell =-\left[ k-1,k-2,\cdots ,1,0\right] ^{T}\;\;\;\;\mbox{and}\;\;\;\mathbf{%
c}=\mathbf{\tilde{A}}e_{s+k-1}+\mathbf{\tilde{D}}\ell .
\end{equation*}%
Below, let terms of the form $\mathbf{c}^{p},\ell ^{p}$ mean that each
element is raised to the power $p$, and terms of the form $\mathbf{c}\cdot 
\mathbf{\tilde{A}}\mathbf{c}$ mean element-wise multiplication of the vector 
$\mathbf{c}$ with the vector $\mathbf{\tilde{A}}\mathbf{c}$.

A $k$-step $s$-stage method of the form \eqref{eq:msrktypeI} is \textbf{%
consistent of order $p$} if the order conditions are satisfied 
\begin{equation*}
\tau _{q_{i}}=0,\;\;\;\forall q_{i}\;\;\mbox{such that}\;\;q\leq p.
\end{equation*}%
The values $\tau _{q_{i}}$ are given in \ref{app:OC}.
\end{definition}

Order conditions for GLMs are similar to those of Runge-Kutta methods, but
take into account how previous time levels provide additional information in
the form of the elementary differentials.

\begin{definition}
\label{GLMstability} A General Linear Method \eqref{eq:msrktypeI} has an 
  \textbf{evolution operator} 
\begin{equation*}
M(z)=\left[ 
\begin{array}{cc}
\mathbf{0} & I^{k-1} \\ 
\multicolumn{2}{c}{\Phi (z)}%
\end{array}%
\right]
\end{equation*}%
where  $\mathbf{0}$ is a vector of zeros of length $k-1$,and $I^{k-1}$ is a $%
(k-1)\times (k-1)$ identity matrix, and
\begin{equation*}
\Phi (z)=\Theta +z\mathbf{\tilde{b}}(I-z\mathbf{\tilde{A}})^{-1}\mathbf{\tilde{D}}.
\end{equation*}%
\end{definition}

\begin{definition} 
A GLM of the form \eqref{eq:msrktypeI} is linearly stable if and only if the roots
of $M$ satisfy the root condition: the eigenvalues $\lambda _{i}(M)$ are less
than or equal to one in magnitude, and that when a given eigenvalue has
magnitude one then it must have algebraic multiplicity equal to one.
\end{definition}

\subsection{Formulating the Optimization Problem\label{sec:optimization}}
\label{optimization} Building on the work of Ketcheson \cite{ketcheson},
this section formulates methods to optimize pre- and post-filters.
Optimization requires specifying the objective function to be optimized, the
inputs, the free parameters, and the equality and inequality constraints. We
begin with a core method with given coefficients, and aim to determine 
the coefficients of a pre-filter and post-filter so that the resulting order is of a specified value
$p$ and the $A(\alpha )$ stability region is maximized. 
The optimization problem is described in the following algorithm:\newline
\textbf{Optimization algorithm:} 

\begin{itemize}
\item \emph{The GLM} is defined by:
\begin{itemize}
\item \emph{Core method coefficients:}  The coefficients of the core method $\mathbf{A}, \mathbf{\hat{A}}, \mathbf{\tilde{D}}$.
\item  \emph{Free variables:}  The {\bf pre- and post-filter coefficients}, which are the coefficients in
$\mathbf{\theta}, \mathbf{b}, \mathbf{\hat{b}}$ and the first row of $\mathbf{D}$.
\item \emph{Computed coefficients:}  The coefficients in all but the first row of $\mathbf{D}$ are 
defined by  \eqref{PPcoef}, which depends on the core coefficients and the free variables. 
\end{itemize}
\item \emph{Select the free variables to maximize $R$ subject to conditions:}
\begin{enumerate}
\item \emph{(Inequality constraints:) The eigenvalues $\lambda _{j}$ of $%
M(z) $ defined in Definition \ref{GLMstability} satisfy 
\begin{equation*}
max_{j}\left\vert \lambda _{j}\right\vert <1
\end{equation*}%
for all $z=\left\vert z\right\vert \exp (i\theta )$ in the wedge defined by 
\begin{equation*}
0\leq |z|\leq \infty ,\;\;\;\mbox{and}\;\;\;\theta \in \left( \frac{\mu^{2}+1}{2\mu ^{2}+1}\pi ,\pi \right) \;\;\;\mbox{for}
\;\;\;0\leq \mu \leq r.
\end{equation*}%
(We enforce this condition for $0\leq |z|\leq 10^{4}$ in the optimization,
but then verify the results for larger values of z.) }

\item \emph{(Equality constraints:)} The order conditions in Definition \ref{OrderConditions} are satisfied to order $p$. 
\end{enumerate}
\end{itemize}
The $\alpha $ that corresponds to the value of $r$ resulting from this
optimization algorithm is given by $\alpha =180\left( 1-\frac{r^{2}+1}{2r^{2}+1}\right) $.

\bigskip

If the needs of an application requires optimization of a different type of
stability region (e.g. imaginary axis stability or real axis stability), we
replace the definition of $A(\alpha )$ with a different type of stability
region, such as:

\begin{enumerate}
\item \textbf{Imaginary axis stability:} For imaginary axis stability, we
require that the eigenvalues $\lambda_j$ of $M(z) $ satisfy $max_j \left|
\lambda_j \right| < 1 $ for all $z=i\nu$ where $0 \leq \nu \leq r$.

\item \textbf{Negative real axis stability:} For negative real axis
stability, we require that the eigenvalues $\lambda_j$ of $M(z) $ satisfy $%
max_j \left| \lambda_j \right| < 1 $ for all $z= - \nu$ where $0 \leq \nu
\leq r$.
\end{enumerate}

Of course, it is possible to optimize with respect to other properties as
well simply by adding these to the objective function or the constraints.

\section{Some new methods induced by time-filters}

\label{sec:NewMethods}

\label{newmethods} In the following sections we present some of the methods
we discussed above and some new methods that we obtained by optimization.

\subsection{Core method: Implicit Euler (IE)}
Consider the implicit Euler method as the starting point or core method: 
\begin{equation*}
u^{n+1}=u^{n}+\Delta t F(u^{n+1}).
\end{equation*}

\subsubsection{Second order method based on implicit Euler with 2-point
filters (IE-Filt($d$))\label{sec:dopt}}

The simplest possible pre-filter to the implicit Euler method is a 2 point
filter of the form: 
\begin{equation*}
y^{(1)} = u^{n}-d\left( u^{n}-u^{n-1}\right) \text{ for some }0<d<1%
\text{.}
\end{equation*}
If a 3-point post-filter is also added the resulting 1-parameter family of
second order methods, given in Eqn. \eqref{BEprepost2_family},
can be written as a GLM of the form: 
\begin{eqnarray*}
y^{(1)} &=&du^{n-1}+(1-d)u^{n} \\
y^{(2)} &=&y^{(1)}+\Delta t F\left( y^{(2)}\right) \\
u^{n+1} &=&\frac{2d-1}{3-2d}u^{n-1}+\frac{4(1-d)}{3-2d}u^{n} + \frac{2}{3-2d}
\Delta t F\left(y^{(2)}\right) .
\end{eqnarray*}
The case $d=0$ gives a method that is not pre-filtered at all, only
post-filtered. However, this does not impact the order of the scheme which
is $O(\Delta t^2)$. This pre-filter does not enhance the order of the
method, but it may serve to improve the error constants. 
The post-filter is designed to impact both the order enhancement and the stability
properties of the method.

When implementing this method, especially in a black-box setting, it is
easier to write it as 
\begin{eqnarray*}
y^{(1)} &=&du^{n-1}+(1-d)u^{n} \\
y^{(2)} &=&y^{(1)}+\Delta tF\left( y^{(2)}\right) \\
u^{n+1} &=& \frac{1}{3-2d} \left(2 y^{(2)} + 2(1-d) u^n - u^{n-1} \right)
\end{eqnarray*}

All values of $0 \leq d \leq 1$ give an A-stable second order
method. We also show that this method is energy-stable
for all values $0 \leq d \leq 1$ (see proof in Appendix).
 An interesting value is 
\begin{equation*}
d=\frac{3-\sqrt{3}}{3},
\end{equation*}
the resulting method is second order, but for linear problems we will see
third order convergence.

\subsubsection{Second order L-stable method based on implicit Euler with
3-point pre-filter (IE-Pre-2)\label{sec:bepre}}

We saw above that a 2-point pre-filter introduces a parameter that can then
be used to optimize some property of the scheme but does not increase order
of accuracy, while the post-filter allows the enhancement of accuracy. In
this section we show that 3-point pre-filter can be used to increase order
of accuracy while maintaining favorable stability properties.

Consider the 3-point pre-filter added to the core implicit Euler method: 
\begin{subequations}
\label{IE_pre3}
\begin{eqnarray}
y^{(1)} &=&-\frac{1}{2}u^{n-2}+u^{n-1}+\frac{1}{2}u^{n} \\
u^{n+1} &=& y^{(1)} +\Delta t F(u^{n+1}).
\end{eqnarray}%
This produces a second order approximation to the solution, and the method
is L-stable.

We can verify that this method is L-stable by analyzing the eigenvalues of
the incremental operator, $M(z)$, which advances the solutions to the next
time level i.e. 
\end{subequations}
\begin{equation*}
\begin{pmatrix}
u^{n-1} \\ 
u^{n} \\ 
u^{n+1}%
\end{pmatrix}%
=M(z)%
\begin{pmatrix}
u^{n-2} \\ 
u^{n-1} \\ 
u^{n}%
\end{pmatrix}%
\mbox{\quad where \quad}M(z)=%
\begin{bmatrix}
0 & 1 & 0 \\ 
0 & 0 & 1 \\ 
\frac{-1}{2(1-z)} & \frac{1}{1-z} & \frac{-1}{2(1-z)}%
\end{bmatrix}%
\end{equation*}%
When we take $\lim_{z\rightarrow -\infty }M(z)$, $M$ becomes upper
triangular so all its eigenvalues become zero. This shows that the second
order pre-filtered method \eqref{IE_pre3} is L-stable. The stability region
of this method is shown in Figure \ref{prepostIE} on the left.

\begin{figure}[tbp]
\includegraphics[scale=.35]{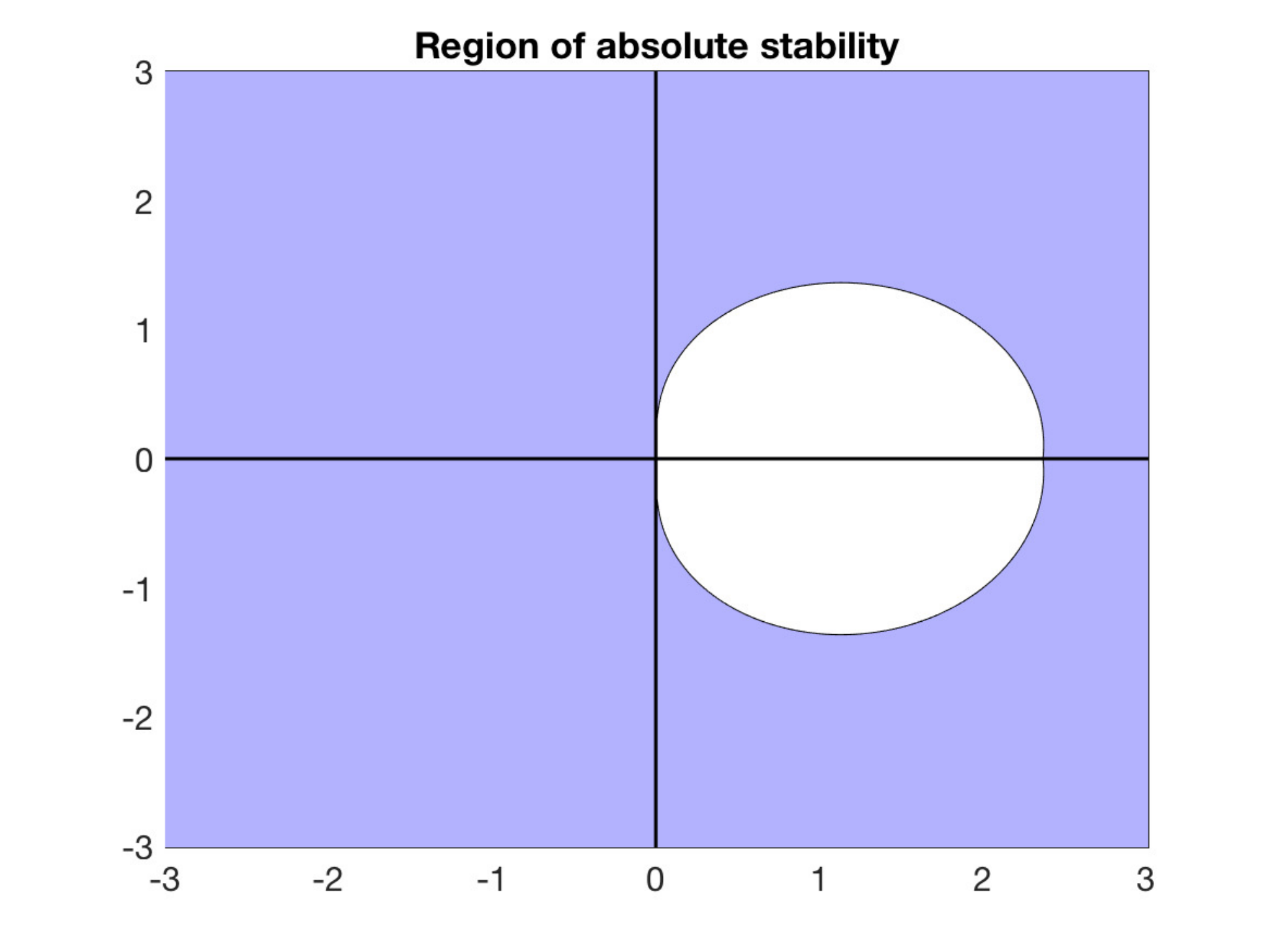}\;\; %
\includegraphics[scale=.35]{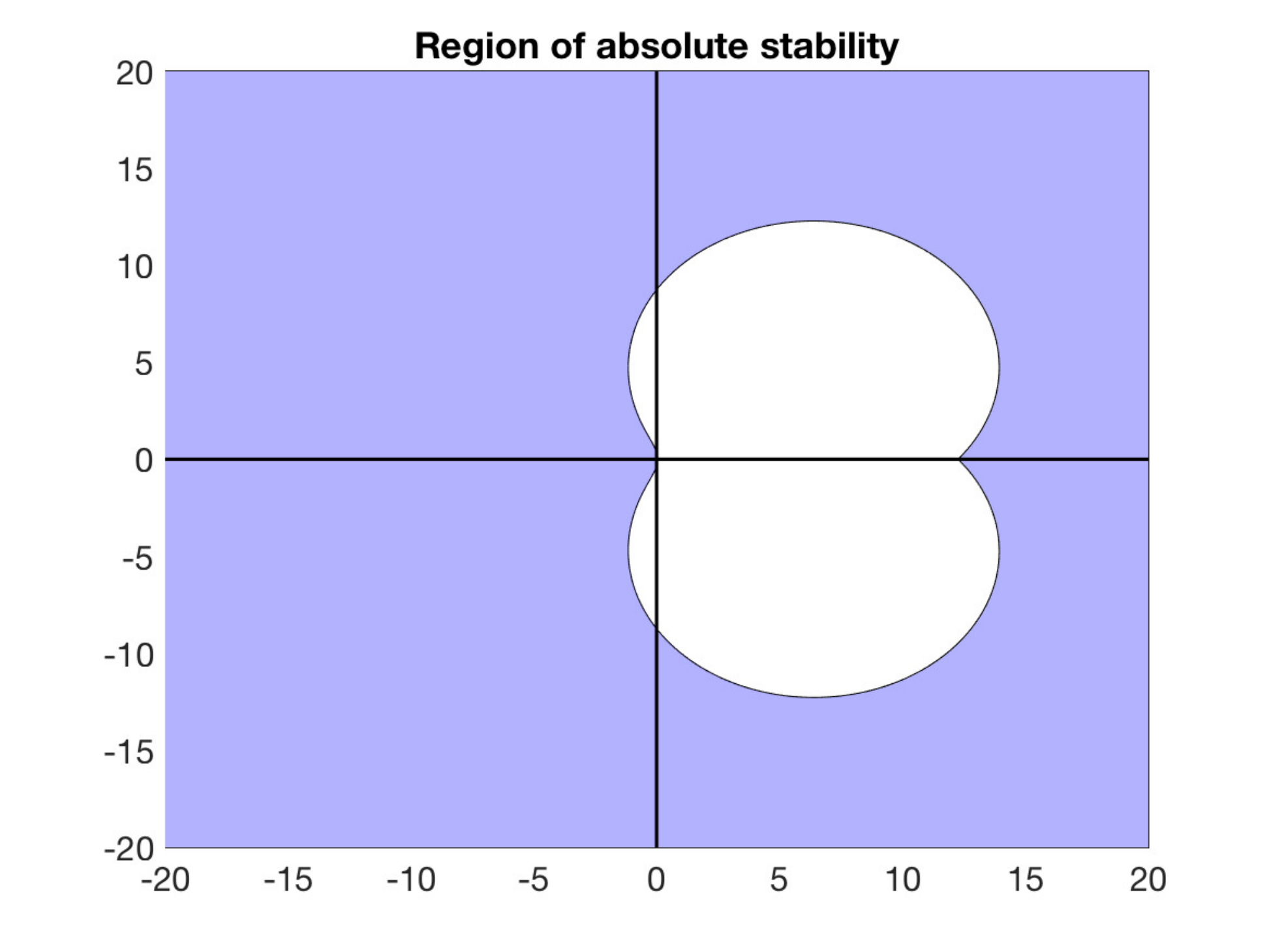}
\caption{Stability Region of the pre-filtered implicit Euler method IE-Pre-2
\eqref{IE_pre3} (left) which is L-stable, and of the pre- and post-filtered
implicit Euler IE-Pre-Pose-3  \eqref{IE_prepost3}(right) which is not A-stable, but is 
$A(\protect\alpha)$ stable with $\protect\alpha \approx 71.51$. }
\label{prepostIE}
\end{figure}

\subsubsection{Third order method using with 3-point pre- and post-filters
(IE-Pre-Post-3)\label{sec:beprepost}}

The method above has only a pre-filter. To raise the order of accuracy of
the scheme to third order, we can add a post-filter as well: 
\begin{subequations}
\label{IE_prepost3}
\begin{eqnarray}
y^{(1)} &=&-\frac{1}{2}u^{n-2}+u^{n-1}+\frac{1}{2}u^{n} \\
y^{(2)} &=&y^{(1)}+\Delta tF\left( y^{(2)}\right) \\
u^{n+1} &=&\frac{5}{11}u^{n-2}-\frac{15}{11}u^{n-1}+\frac{15}{11}u^{n}+\frac{%
6}{11}y^{(2)}
\end{eqnarray}%
While we have gained order of accuracy, we lost stability properties by
adding a post-filter. This third order method is not A-stable, but it has an
$A(\alpha )$ region of stability with $\alpha \approx 71.51$. The stability
region of this method is shown in Figure \ref{prepostIE} on the right.

\subsubsection{A-stable third order error inhibiting method (IE-EIS-3)}
\label{sec:EISmethod}

An A-stable third order method that is based on the implicit Euler method
can be obtained by using the error inhibiting approach presented in \cite{EISpp20}. 
In this formulation, we retain previous stages as well as previous
steps, to create a method of the form 
\end{subequations}
\begin{eqnarray*}
u^{n+\frac{2}{3}} &=&\frac{14}{5}u^{n-\frac{1}{3}}-\frac{9}{5}u^{n}+\frac{9}{%
5}\Delta tF(u^{n-\frac{1}{3}})-\frac{6}{5}\Delta tF(u^{n})+\Delta tF(u^{n+%
\frac{2}{3}}) \\
u^{n+1} &=&\frac{14}{5}u^{n-\frac{1}{3}}-\frac{9}{5}u^{n}+\frac{9}{5}\Delta
tF(u^{n-\frac{1}{3}}) \\
&&\text{ \ \ }-\frac{47}{60}\Delta tF(u^{n})-\frac{1}{12}\Delta tF(u^{n+%
\frac{2}{3}})+\Delta tF(u^{n+1}).
\end{eqnarray*}%
This method takes the time levels $u^{n-\frac{1}{3}}$ and $u^{n}$ and
advances them to $u^{n+\frac{2}{3}}$ and $u^{n+1}$. To reveal the dependence
on previous stages and previous steps, and the fact that the method is based
on the implicit Euler method, we re-write it in the form 
\begin{eqnarray*}
y_{n}^{(1)} &=&\frac{14}{5}y_{n-1}^{(2)}-\frac{9}{5}u^{n}+\frac{9}{5}\Delta
tF(y_{n-1}^{(2)})-\frac{6}{5}\Delta tF(u^{n}) \\
y_{n}^{(2)} &=&y_{n}^{(1)}+\Delta tF(y_{n}^{(2)}) \\
y_{n}^{(3)} &=&y_{n}^{(2)}+\frac{5}{12}\Delta tF(u^{n})-\frac{13}{12}\Delta
tF(y_{n}^{(2)}) \\
u^{n+1} &=&y_{n}^{(3)}+\Delta tF(u^{n+1}).
\end{eqnarray*}%
It is usually preferable to implement this in the form: 
\begin{eqnarray}
y_{n}^{(1)}\notag
&=&\frac{23}{5}y_{n-1}^{(2)}-3u^{n}-\frac{9}{5}y_{n-1}^{(1)}+\frac{6}{5}%
y_{n-1}^{(3)} \label{eqn:eis3-solve-1} \\
y_{n}^{(2)} &=&y_{n}^{(1)}+\Delta tF(y_{n}^{(2)}) \\\notag
y_{n}^{(3)} 
&=&\frac{5}{12}u^{n}-\frac{1}{12}y_{n}^{(2)}-\frac{5}{12}y_{n-1}^{(3)}
+\frac{13}{12}y_{n}^{(1)} \\
u^{n+1} &=&y_{n}^{(3)}+\Delta tF(u^{n+1}).\label{eqn:eis3-solve-2}
\end{eqnarray}%
This method satisfies the order conditions up to second order, but its
coefficients also satisfy the error inhibiting property in \cite{EISpp20} and so
the resulting numerical solution is third order. As mentioned above, it is
A-stable, and we notice that while $y^{(1)}$ and $y^{(3)}$ are linear
combinations of previous steps and stages, $y^{(2)}$ and $u^{n+1}$ are
simply implicit Euler computations, that can be computed in any legacy code
that is based on the implicit Euler method.

It is important to note that, as in Runge--Kutta methods,  if the problem is
non-autonomous we need to compute the function evaluations at the correct
time-levels. In this case the time-levels are $t_{n}+\frac{2}{3}\Delta t$
for $y_{n}^{(2)}$, and $t_{n}+\Delta t$ for $u^{n+1}$.

\subsection{Core method: Implicit Midpoint Rule (MP)\label{sec:imp_midpoint}}
Next, we will  consider the second order implicit midpoint rule
as our core method: 
\begin{equation*}
u^{n+1}=u^{n}+\Delta t\;F\left( \frac{1}{2}u^{n}+\frac{1}{2}u^{n+1}\right)
\end{equation*}%
which, as we saw in Section 2, can be written in the equivalent form 
\begin{eqnarray*}
y^{(1)} &=&u^{n}+\frac{1}{2}\Delta tF(y^{(1)}) \\
u^{n+1} &=&2y^{(1)}-u^{n}.
\end{eqnarray*}%
Using pre- and post-filters we can raise the order of this method. While the
resulting methods are not A-stable, they are A$(\alpha)$ stable for large
values of $\alpha$.

\subsubsection{Third order filtered implicit midpoint rule (MP-Pre-Post-3)}

We can filter the implicit midpoint method to obtain the following third
order method: 
\begin{subequations}
\label{Midpoint_prepost3}
\begin{eqnarray}
y^{(1)} &=& -\frac{1}{12}u^{n-3}+\frac{1}{2}u^{n-2}-\frac{5}{4}u^{n-1}+\frac{%
11}{6}u^{n} \\
y^{(2)} &=& y^{(1)} + \frac{1}{2} \Delta t F(y^{(2)} ) \\
y^{(3)} &=& 2 y^{(2)} - y^{(1)} \\
u^{n+1} &=& \frac{1}{2} y^{(1)} + \frac{1}{2} y^{(3)} \; .
\end{eqnarray}
Observe that the final step is simply $y^{(2)}$, so we can say 
\end{subequations}
\begin{eqnarray*}
y^{(1)}&=& -\frac{1}{12} u^{n-3}+\frac{1}{2}u^{n-2}-\frac{5}{4}u^{n-1}+\frac{%
11}{6}u^{n} \\
y^{(2)} &=& y^{(1)} + \frac{1}{2}\Delta t F(y^{(2)} ) \\
u^{n+1} & = & y^{(2)} \; .
\end{eqnarray*}
However, if one is working with a code that treats the implicit midpoint
rule as a black box and does not output the intermediate value in the core
method, it is more convenient to use \eqref{Midpoint_prepost3}. 

This method
is not A-stable, but it is $A(\alpha)$ stable with $\alpha = 79.4 $. 
The advantage of this method is that using the same pre-filter but a
different post-filter gives a fourth order method, as we see in the next
subsection. The two approaches form an embedded pair which is convenient for 
error estimation.

\subsubsection{Fourth order filtered implicit midpoint rule (MP-Pre-Post-4)}

If we use a method similar to \eqref{Midpoint_prepost3}, but with a
different post-filter 
\begin{subequations}
\label{Midpoint_prepost4}
\begin{eqnarray}
u^{n+1} &=&-\frac{1}{25}u^{n-3}+\frac{4}{25}u^{n-2}-\frac{6}{25}u^{n-1}+%
\frac{4}{25}u^{n}+\frac{24}{25}y^{(2)}\;,
\end{eqnarray}%
we obtain a fourth order method . If the implicit midpoint rule is coded as
a black box, we may prefer to write this in the form 
\end{subequations}
\begin{eqnarray*}
y^{(1)} &=&-\frac{1}{12}u^{n-3}+\frac{1}{2}u^{n-2}-\frac{5}{4}u^{n-1}+\frac{%
11}{6}u^{n} \\
y^{(2)} &=&y^{(1)}+\frac{1}{2}\Delta tF(y^{(2)}) \\
y^{(3)} &=&2y^{(2)}-y^{(1)} \\
u^{n+1} &=& -\frac{2}{25}u^{n-3}+\frac{2}{5}u^{n-2}-\frac{21}{25}u^{n-1}+\frac{%
26}{25}u^{n}+\frac{12}{25} y^{(3)}  \notag
\end{eqnarray*}

This fourth order method has $A(\alpha )$ stability region with $\alpha
\approx 70.64$.  By using the post-filter from the third order
method \eqref{Midpoint_prepost3} and comparing it with the result from this
fourth order method, we obtain an error estimator.

\begin{figure}[tbp]
\includegraphics[scale=.35]{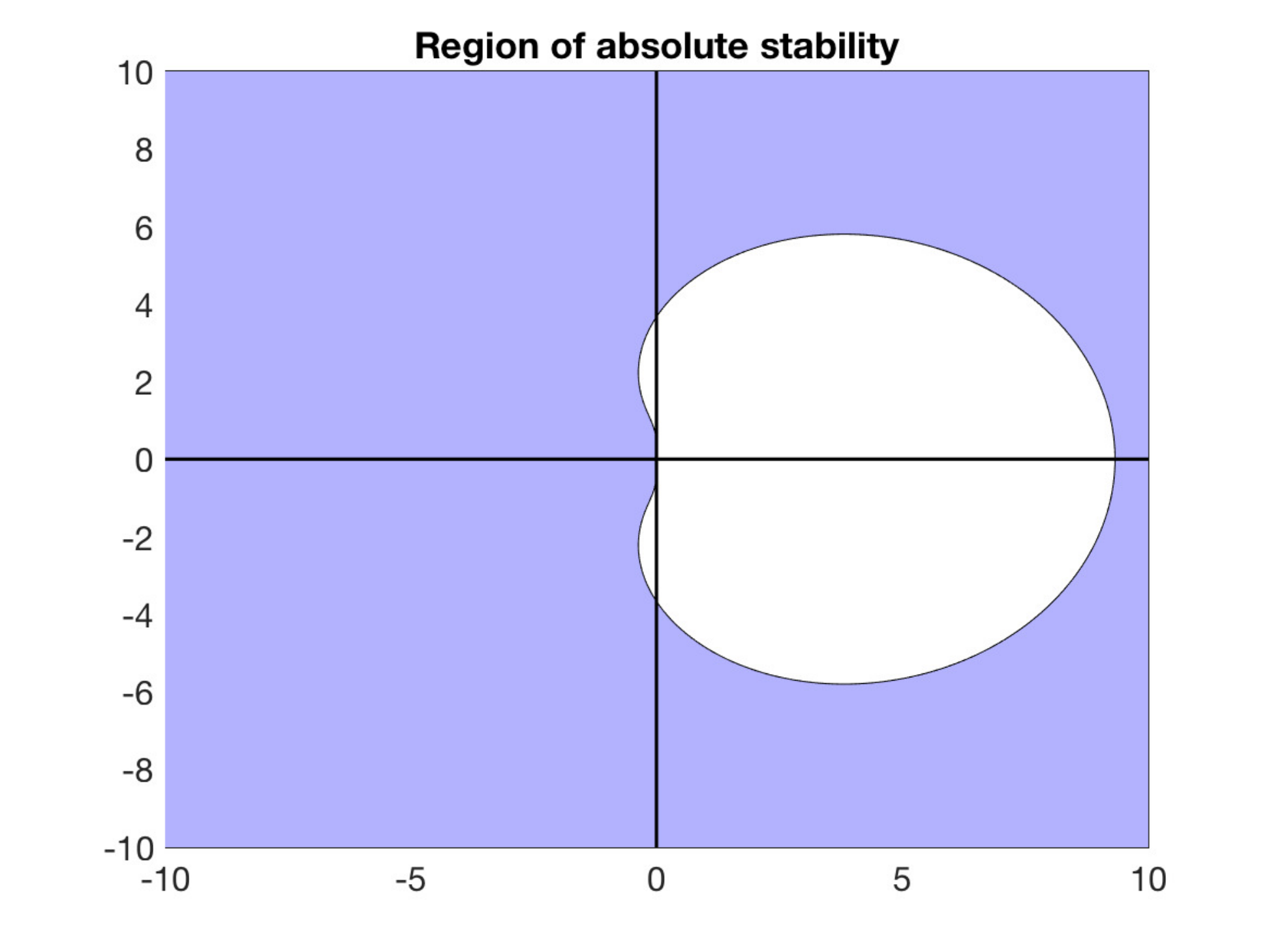}\;\; %
\includegraphics[scale=.35]{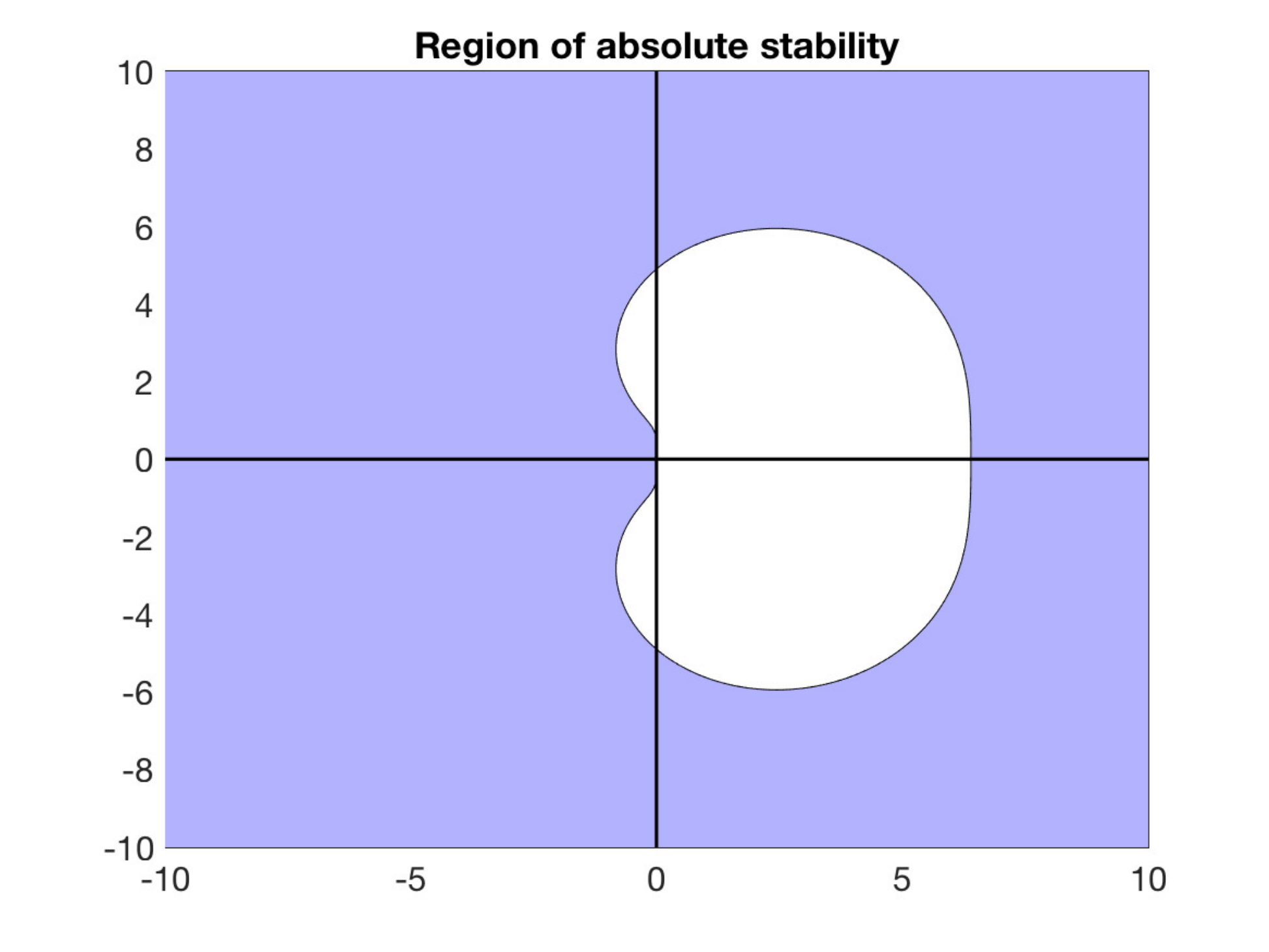}
\caption{Stability regions of third  order MP-Pre-Post-3 (left) and fourth order MP-Pre-Post-4 (right)
filtered implicit midpoint methods. }
\label{midpointprepost}
\end{figure}

\subsection{Filtered BDF2 \label{sec:bdf2}}

In this subsection, we begin with the second order backward differentiation
formula (BDF-2) scheme 
\begin{equation*}
u^{n+1} =-\frac{1}{3}u^{n-1} + \frac{4}{3}u^n + \frac{2}{3}\Delta t
F(u^{n+1}).
\end{equation*}
as the core method. We write this method in GLM form as 
\begin{eqnarray*}
y^{(1)} & =& u^n \\
y^{(2)} & = & -\frac{1}{3}u^{n-1} + \frac{4}{3}u^n + \frac{2}{3}\Delta t
F(y^{(2)}). \\
u^{n+1} & = & y^{(2)} .
\end{eqnarray*}

\subsubsection{Third order filtered method (BDF2-Post-3)}
We can post-filter the BDF2 method to obtain a third order method: 
\begin{eqnarray*}
y^{(1)} & =& u^n \\
y^{(2)} & = & -\frac{1}{3}u^{n-1} + \frac{4}{3}u^n + \frac{2}{3}\Delta t
F(y^{(2)}) \\
u^{n+1} & = & y^{(2)} - \frac{2}{11} \left( y^{(2)} - 3 u^n + 3 u^{n-1} -
u^{n-2} \right) \\
& = & \frac{9}{11} y^{(2)} + \frac{6}{11} u^n - \frac{6}{11} u^{n-1} + \frac{%
2}{11} u^{n-2}.
\end{eqnarray*}
This method has stability region $A(\alpha)$ with $\alpha=83.89$.

\subsubsection{A third order method with enhanced stability region (BDF2-Pre-Post-3)}

By adding a four-step pre- and post-filter, we can obtain a third order
method. The following third order method has four steps, three stages, and
stage order $q=2$. The method was created to optimize the value $\alpha$ in
the $A(\alpha)$ linear stability region. This method has value $\alpha=89.59$%
. 

\begin{figure}[tbp]
\begin{center}
\includegraphics[scale=.30]{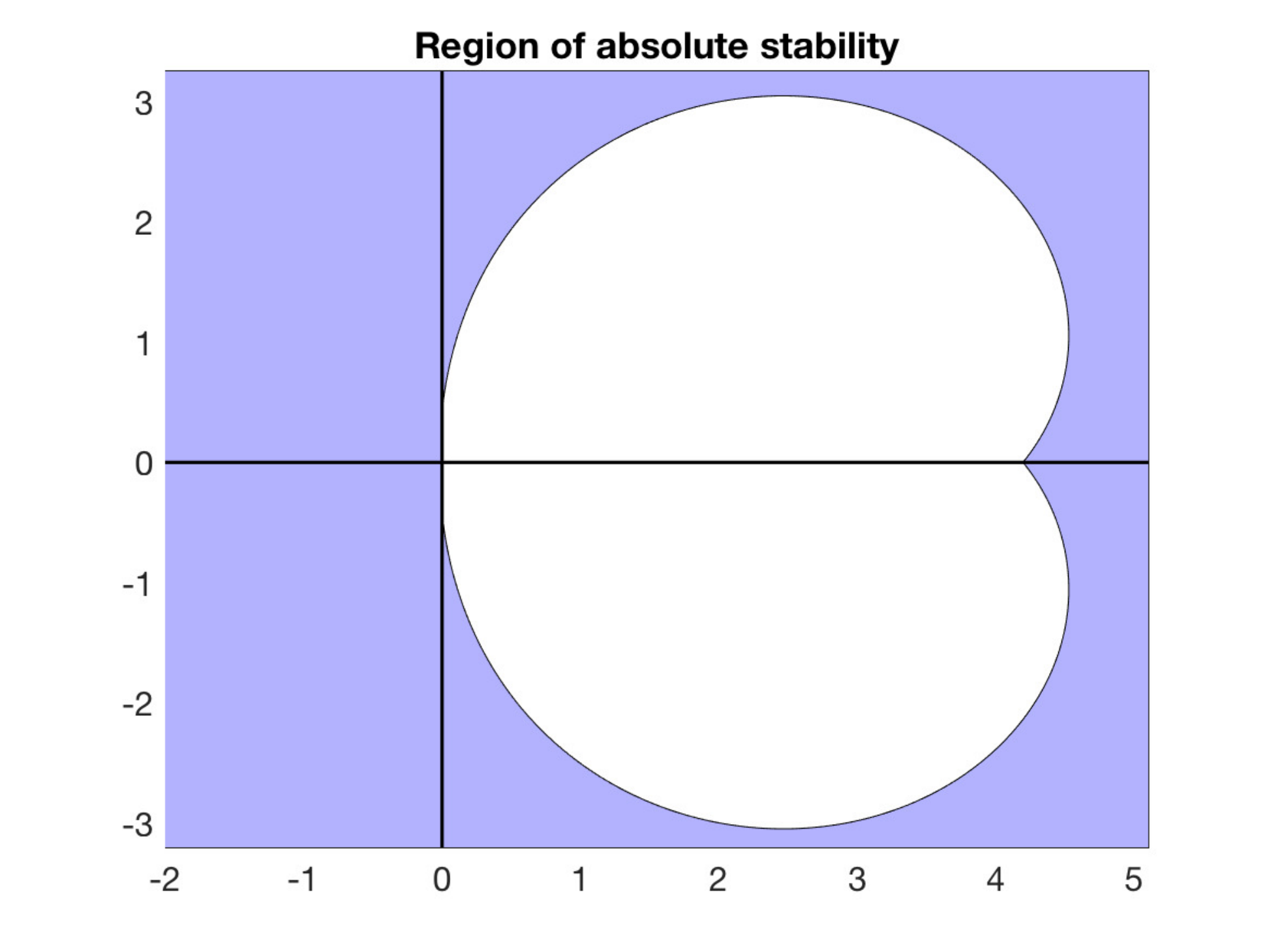} %
\includegraphics[scale=.30]{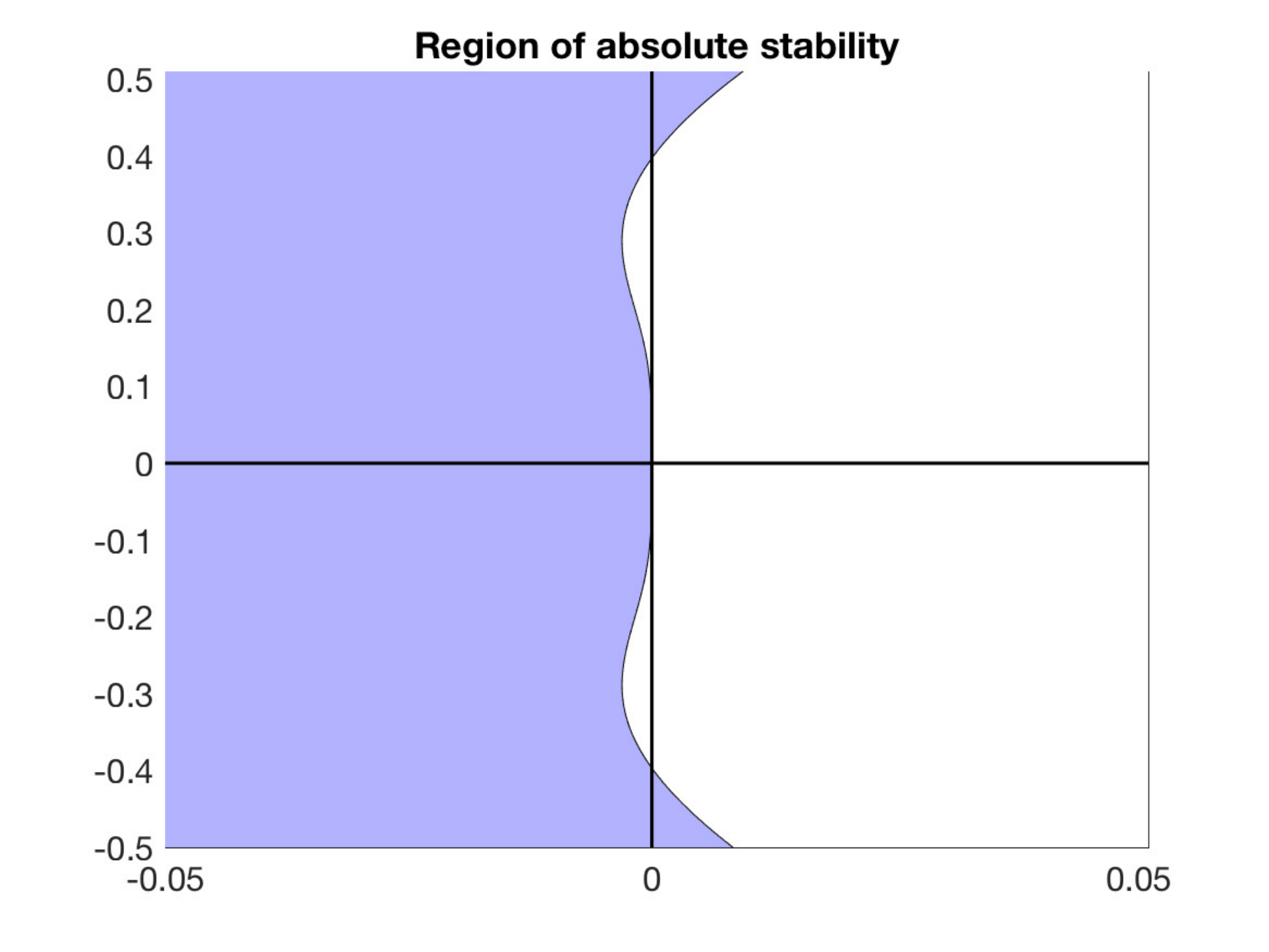}
\end{center}
\caption{Stability region for the BDF2-Pre-Post-3 \eqref{Method2}, with $k=4,s=3,q=2$. On the
left we see the stability region, the zoomed image on the right shows that
the method is not quite A-stable. In fact, we have $A(\protect\alpha)$
stability with $\protect\alpha \approx 89.59$. }
\label{Method2stability}
\end{figure}

\begin{subequations}
\begin{eqnarray}
y^{(1)} &=&d_{1}u^{n-3}+d_{2}u^{n-2}+d_{3}u^{n-1}+d_{4}u^{n},\;\;\;\;%
\mbox{(Pre-filter)}  \label{Method2} \\
y^{(2)} &=&-\frac{1}{3}u^{n-1}+\frac{4}{3}y^{(1)}+\frac{2}{3}\Delta
tF(y^{(2)}),. \\
u^{n+1} &=&\theta _{1}u^{n-3}+\theta _{2}u^{n-2}+\theta _{3}u^{n-1}+\theta
_{4}u^{n}+b\Delta tF(y^{(2)})\;\;\;\;\mbox{(Post-filter)} \\
&=&\theta _{1}u^{n-3}+\theta _{2}u^{n-2}+\theta _{3}u^{n-1}+\theta _{4}u^{n}+%
\frac{3b}{2}\Delta t\left( y^{(2)}+\frac{1}{3}u^{n-1}-\frac{4}{3}%
y^{(1)}\right)  \notag \\
&=&\theta _{1}u^{n-3}+\theta _{2}u^{n-2}+(\theta _{3}+\frac{1}{2}%
b)u^{n-1}+\theta _{4}u^{n}+\frac{3b}{2}\Delta t\left( y^{(2)}-\frac{4}{3}%
y^{(1)}\right),  \notag
\end{eqnarray}%
\end{subequations}
with coefficients
\begin{align*}
&d_{1}=2.670130894410204, &d_{2}=-3.311517498805319, \\
&d_{3}=-3.489799303077245, &d_{4}=5.131185907472361, \\
&\theta _{1}=0.370742163920604, &\theta _{2}=-0.631064728171402, \\
&\theta _{3}=-0.729528261935270,&\theta _{4}=1.989850826186068, \\
&b=0.120568773483737, &
\end{align*}
The linear stability region is presented in Figure \ref{Method2stability}.

\subsection{Filtered fully implicit Runge--Kutta (2,2) \label{sec:rk22}}

We emphasize that this approach works to pre- and post-filter all GLMs, not
just linear multistep methods. Consider the L-stable fully implicit Lobatto
IIIC scheme: 
\begin{eqnarray*}
y^{(1)} &=&u^{n}+\frac{1}{2}\Delta tF(y^{(1)})-\frac{1}{2}\Delta tF(y^{(2)})
\\
u^{n+1} &=&u^{n}+\frac{1}{2}\Delta tF(y^{(1)})+\frac{1}{2}\Delta tF(y^{(2)})
\end{eqnarray*}%
The 2-step time-filtered scheme, which we call RK22-Pre-Post-3,
 can be written as : 
\begin{eqnarray*}
\hat{u} &=&d_{1}u^{n-1}+d_{2}u^{n} \\
y^{(1)} &=&\hat{u}+\frac{1}{2}\Delta tF(y^{(1)})-\frac{1}{2}\Delta
tF(y^{(2)}) \\
y^{(2)} &=&\hat{u}+\frac{1}{2}\Delta tF(y^{(1)})+\frac{1}{2}\Delta
tF(y^{(2)}) \\
u^{n+1} &=&q_{1}u^{n-1}+q_{2}u^{n}+q_{3}y^{(1)}+q_{4}y^{(2)}\vspace*{-0.2in}
\end{eqnarray*}%
where 
\begin{align*}
&d_{1} =0.373461706729200,&d_{2}=0.626538293270800, \\
&q_{1} =-0.075425887737539,&q_{2}=0.551112405533260, \\
&q_{3} =-0.596071637983322,&q_{4}=1.120385120187601.
\end{align*}
This scheme is third order and is A-stable, but not L-stable.

\section{Numerical tests of the methods\label{sec:NumericalTests}}

\label{numerical} This section presents several numerical test and
comparisons of the timestepping methods applied to the Navier-Stokes
equations. The spacial terms are discretized by a standard (not upwind)
finite element method with inf-sup stable elements. Let $\mathcal{P}_k^d$ be
Lagrange finite elements with $d$ components with $d= 2$ or $3$, and degree $%
k$. We use Hood-Taylor elements described by $(\mathcal{P}_k^d, \mathcal{P}%
_{k-1}^1)$, which correspond to $\mathcal{P}_k^d$ vector elements for
velocity and $\mathcal{P}_{k-1}^1$ scalar elements for pressure. We use a
sufficiently fine meshes such that we expect the error to be dominated by
time discretization.

The fully discrete methods are based on a standard weak formulation for the
incompressible NSE. Let $\Omega $ be and open subset of $\mathbb{R}^{2}$,
and let $(\cdot ,\cdot )$ denote the $L^{2}(\Omega )$ inner product. In %
\eqref{eq:NSE}, test the momentum equation with a vector valued function $v$
which vanishes on the boundary and the mass equation with a scalar function $%
q$ with zero mean. After integrating by parts, the weak formulation of %
\eqref{eq:NSE} is 
\begin{gather*}
(u_{t},v)+\nu (\nabla u,\nabla v)+(u\cdot \nabla u+\frac{1}{2}(\nabla \cdot
u)u,v)-(p,\nabla \cdot v)=(f,v), \\
(\nabla \cdot u,q)=0.
\end{gather*}%
The nonlinearity has been explicitly skew-symmetrized (when boundary
conditions allow) in a standard way by adding $\frac{1}{2}(\nabla \cdot u)u$%
. 

While the methods we test are derived for autonomous ODEs, the ensuing tests involve non-autonomous sources and time dependent boundary conditions. There is also a question about the impact of the pre- and post-processors on the fluid pressure since it is an unknown which does not satisfy an evolution equation. These issues are addressed in  \ref{sec:non-aut}.

The first test in Section \ref{sec:test_analytic} is a convergence rate
verification against a closed form, exact solution. The second test, in
Section \ref{sec:test_benchmark}, is a benchmark test of flow through a
channel past a cylindrical obstacle for which there are published reference
values in \cite{J04}. The last test, in Section \ref{sec:test_offset}, is
for a quasi-periodic flow where phase accuracy is important.

\subsection{Convergence Benchmark Test\label{sec:test_analytic}}

For this test we solve the homogeneous NSE (so $f(x)=0$) with $d=2$ under $%
2\pi $ periodic boundary conditions with zero mean. Since the solution is
analytic, we used higher order Hood-Taylor, $(\mathcal{P}_{4}^{2},\mathcal{P}%
_{3}^{1})$ elements and 125 element edges per side of the periodic box,
resulting in 640,625 degrees of freedom. The boundary conditions and zero
mean condition are imposed strongly (as usual) on the FEM spaces. The
Taylor-Green exact solution used is%
\begin{eqnarray*}
u(x,y,t) &=&e^{-2\nu t}(\cos x\sin y,-\sin x\cos y)\text{ and} \\
\text{ }p(x,y,t) &=&-\frac{1}{4}e^{-4\nu t}(\cos 2x+\cos 2y)
\end{eqnarray*}%
The solutions are computed to a final time at $T_{f}=1$ for several
stepsizes starting from $\Delta t=0.2$, and then halving. Since the solution
decays exponentially, absolute errors have little meaning. Thus we compute
relative errors at final time $T_{f}=1$ 
\begin{equation}
\text{relative error}=\sqrt{\frac{\int_{\Omega }|u(T_{f})-u_{h}(T_{f})|^{2}dx%
}{\int_{\Omega }|u(T_{f})|^{2}dx}}.  \notag
\end{equation}

\begin{figure}[th]
\centering
\begin{subfigure}{0.49\linewidth}
   \centering
   \includegraphics[width = 1\linewidth]{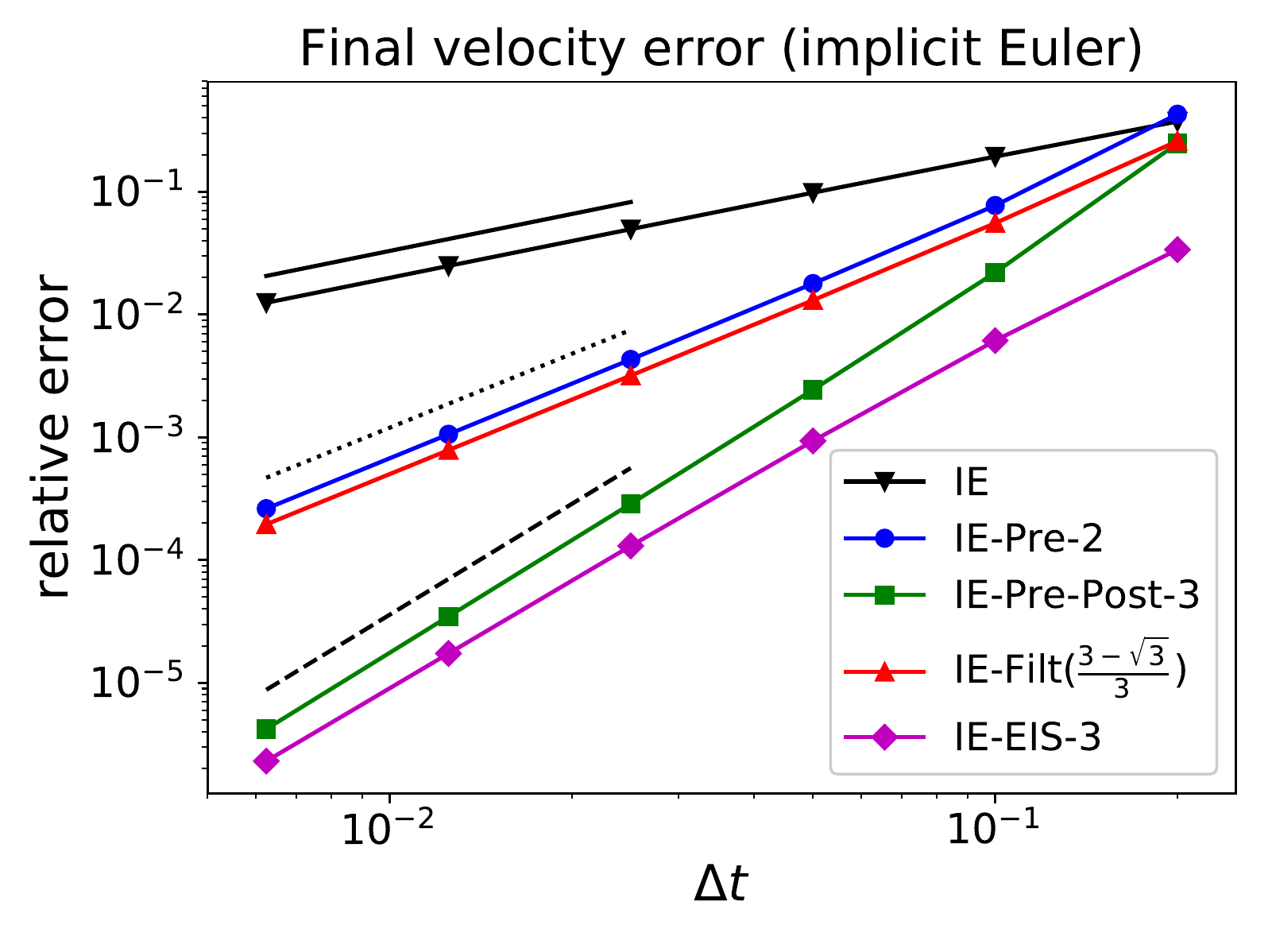}
\end{subfigure}
\begin{subfigure}{0.49\linewidth}
   \centering
   \includegraphics[width = 1\linewidth]{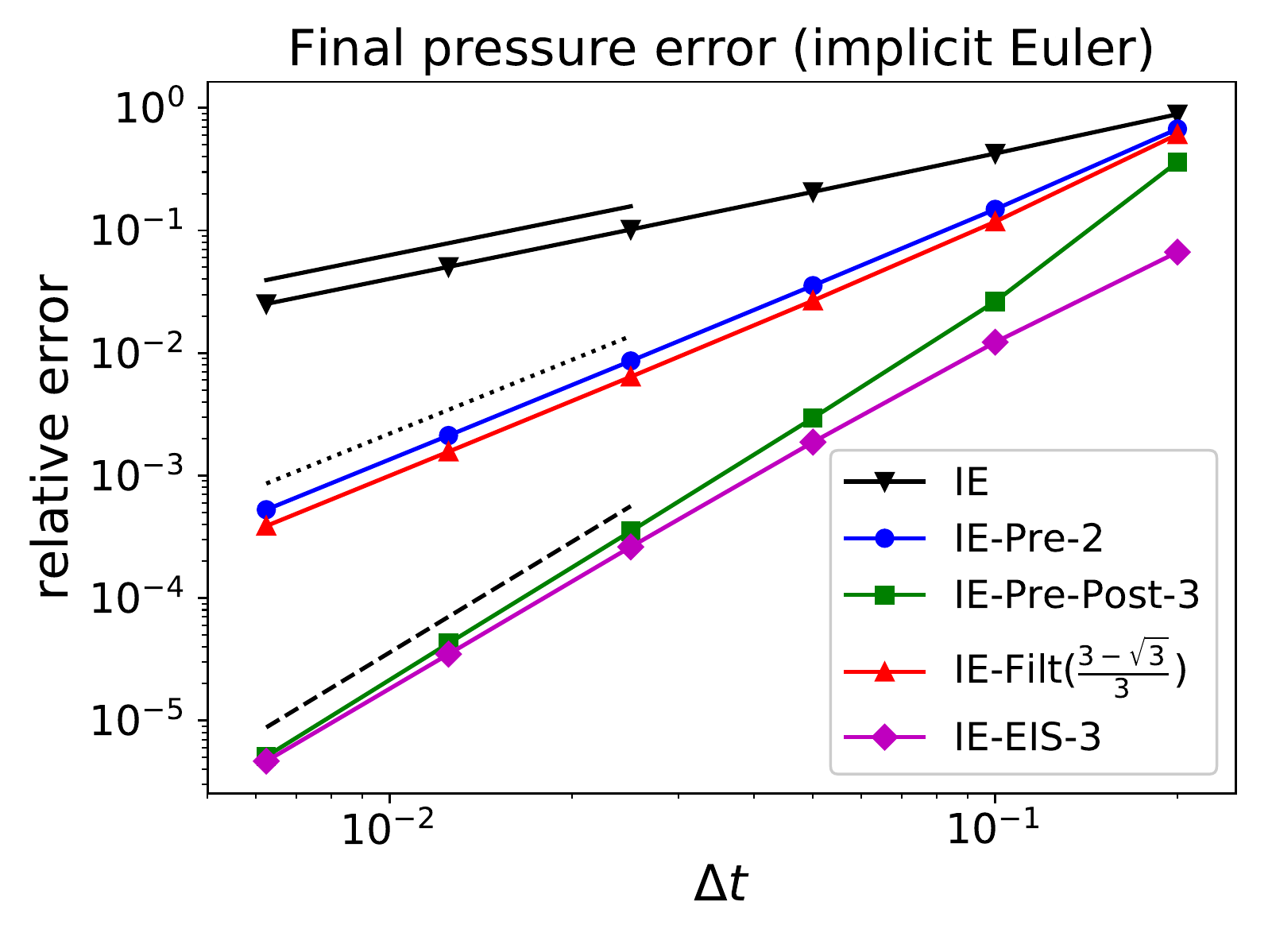}
\end{subfigure}
\par
\begin{subfigure}{0.49\linewidth}
   \centering
   \includegraphics[width = 1\linewidth]{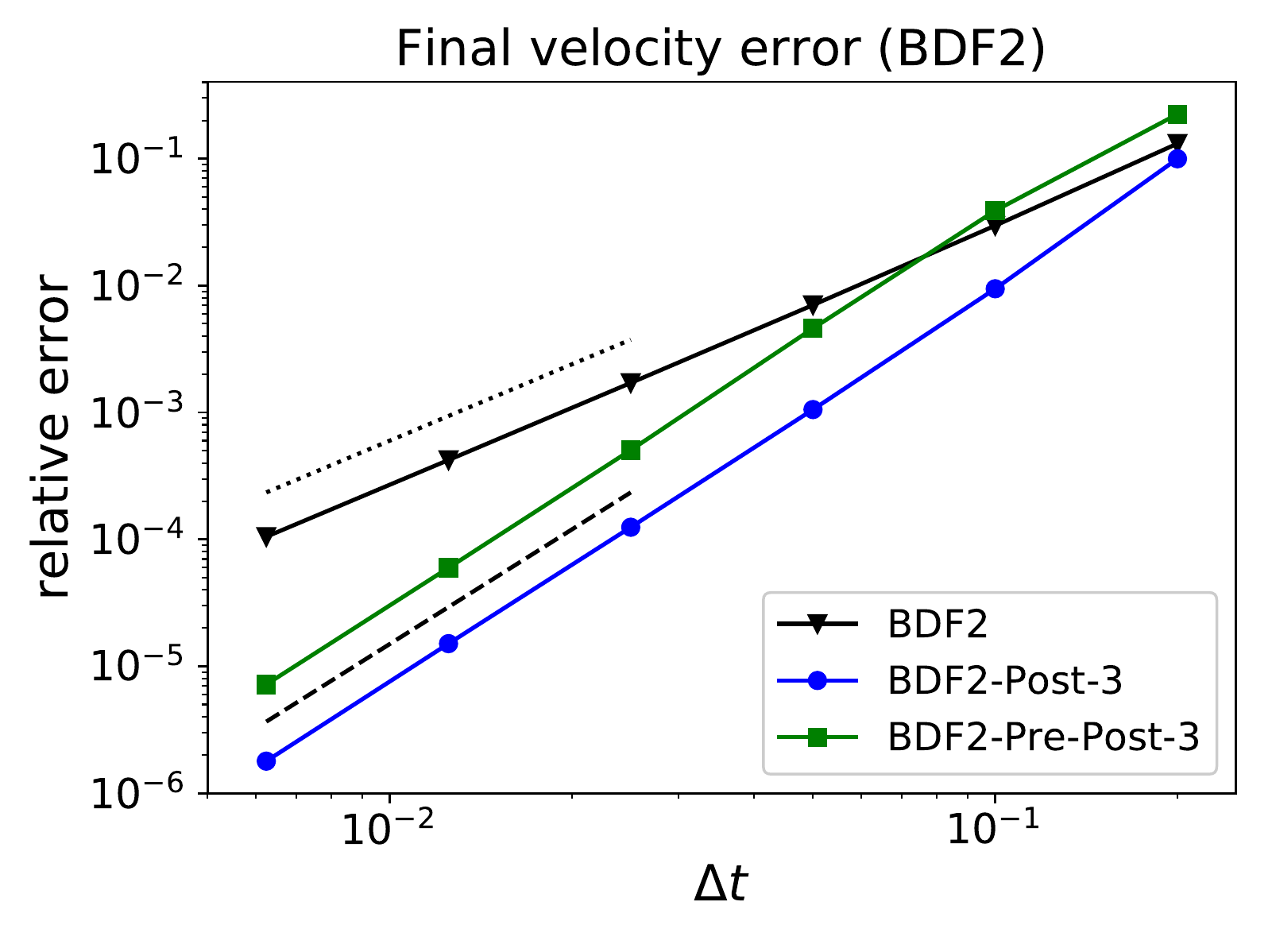}
\end{subfigure}
\begin{subfigure}{0.49\linewidth}
   \centering
   \includegraphics[width = 1\linewidth]{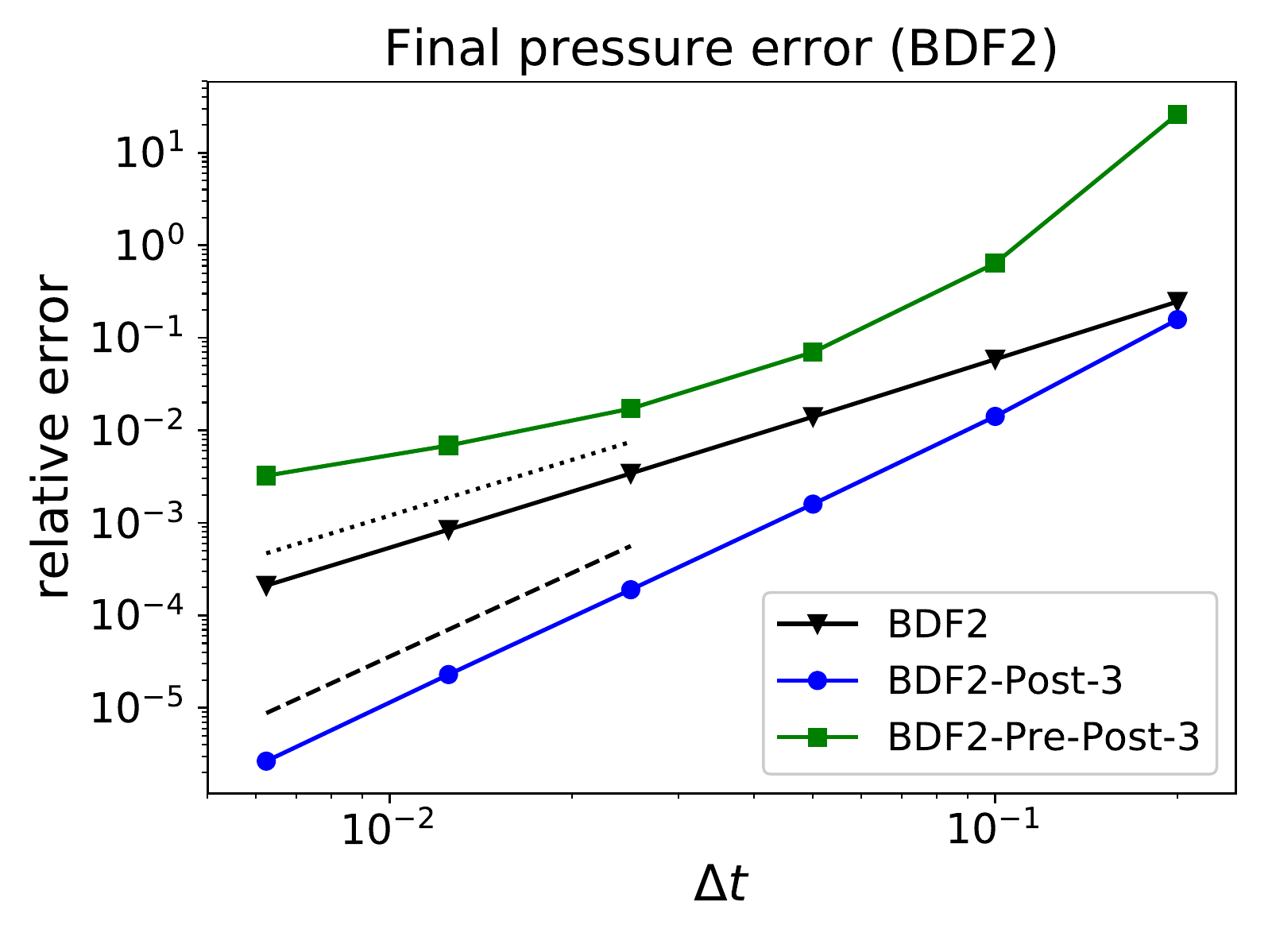}
\end{subfigure}
\par
\begin{subfigure}{0.49\linewidth}
   \centering
   \includegraphics[width = 1\linewidth]{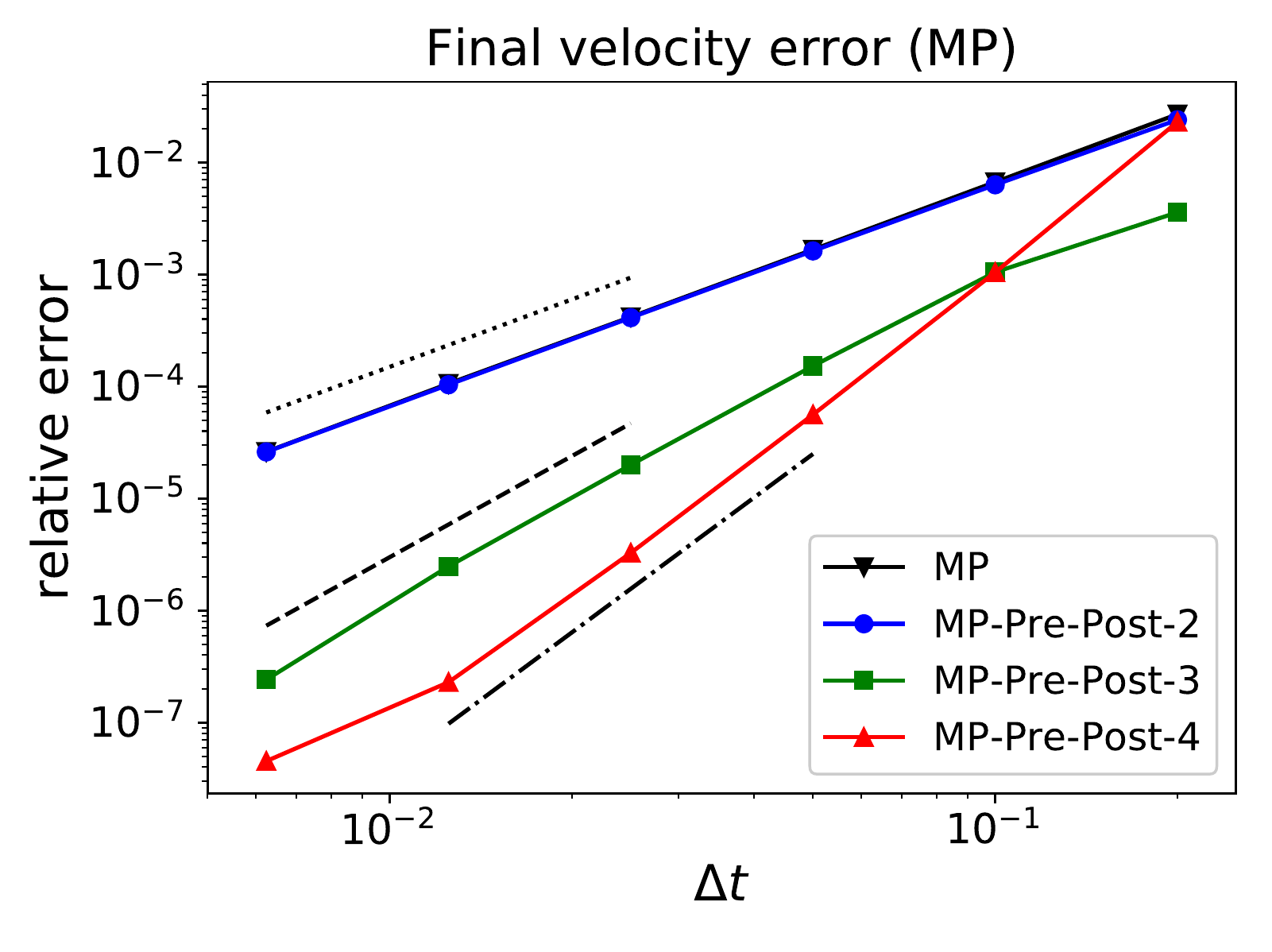}
\end{subfigure}
\begin{subfigure}{0.49\linewidth}
   \centering
   \includegraphics[width = 1\linewidth]{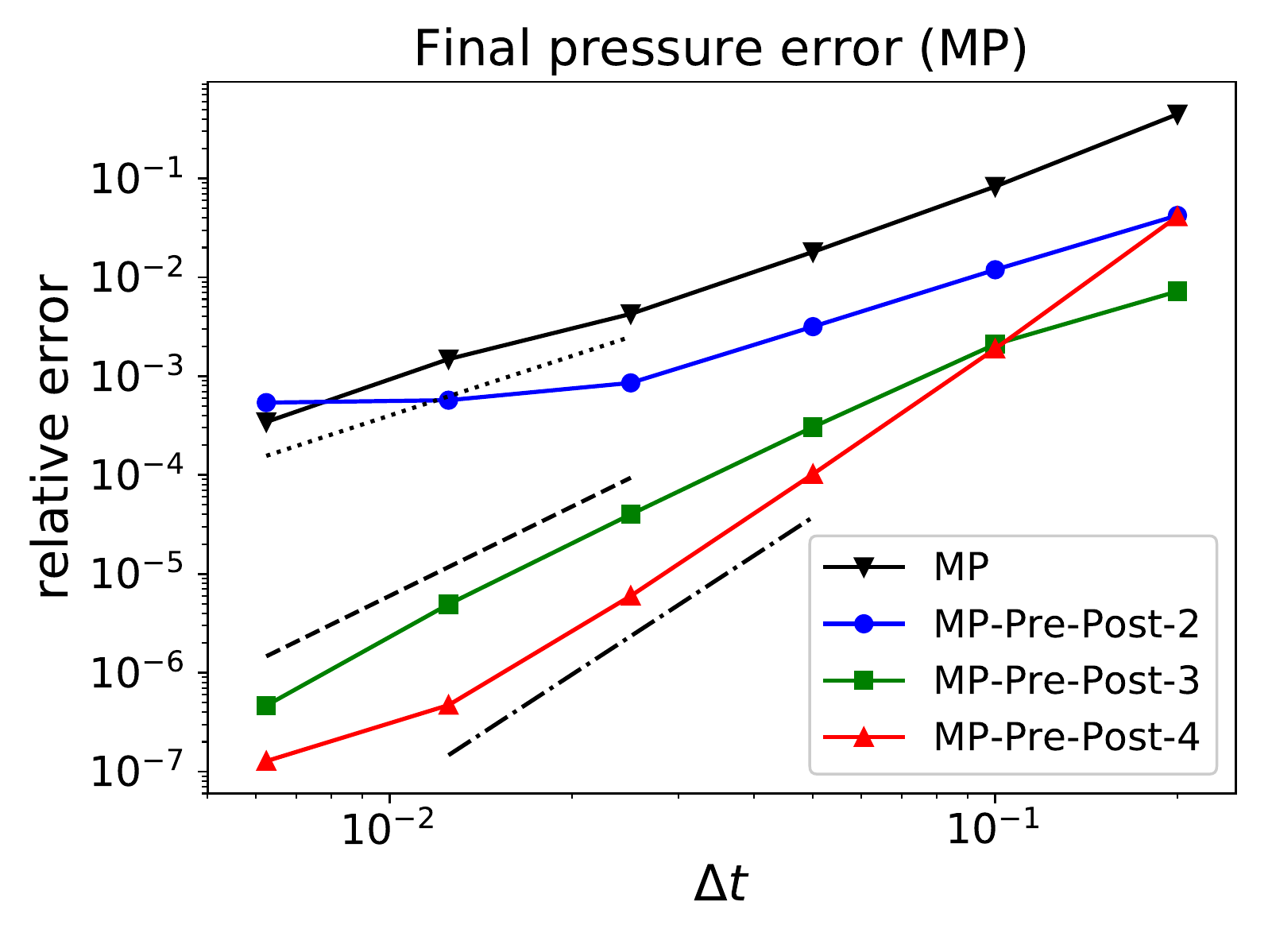}
\end{subfigure}
\caption{Convergence rates of velocity and pressure for different base
methods. The reference lines for the orders of the method are solid (first), dotted (second), dashed (third) and dashed-dotted (fourth). The velocities of the methods converge at the expected rates except. The errors in pressure for BDF2-Pre-Post-3 and MP-Pre-Post-2 either plateau or converge sub-optimally. The reason is unknown and is the subject of future research.}
\label{fig:errors}
\end{figure}

Comparing the methods, for implicit Euler the core method (upper row of
Figure 4)\textit{\ pre- and post-filtered implicit Euler and EIS3 are by far
the most accurate}. EIS3 requires 2 implicit Euler solves per step compared
to 1 solve/step for pre \& post-filtered IE. Both attain their the $\mathcal{O}%
(\Delta t^{3})$ rate of convergence, as predicted by the theory.
The middle row treats the midpoint rule plus filters. The result here
is entire consistent: \textit{higher accuracy (in the sense of consistency
error) produces a more accurate approximation}. We note that in the right
side figure the 4th order approximation hits an error plateau of $10^{-7}$
where the spacial errors are no longer negligible. In the bottom row the
second BDF2 filtered method performed far better, attaining its expected
rate of convergence. In all tests, \textit{run times depended on the number
of core method solves, independent of the number of filter steps}, as
expected.

\subsection{Benchmark test: Flow past a cylinder\label{sec:test_benchmark}}

This next test is a commonly used benchmark described in \cite{ST96}. Fluid
flows into a channel from the left and flows around slightly off center
cylindrical obstacle. The fluid starts at rest and the inflow velocity is
ramped up from zero. When the inflow velocity is high enough, vortices shed
off the obstacle (see Figure \ref{fig:flow-past-cylinder}). The data monitored are the lift and drag due to the
cylinder, and the pressure difference before and after the cylinder. The
geometry and flow profile is given by \cite{ST96}; we compare our results to
benchmark lift and drag values obtained in a DNS study from \cite{J04}.

\begin{figure}[th]
\centering
\begin{subfigure}{\linewidth}
   \centering
   \includegraphics[width = 1\linewidth]{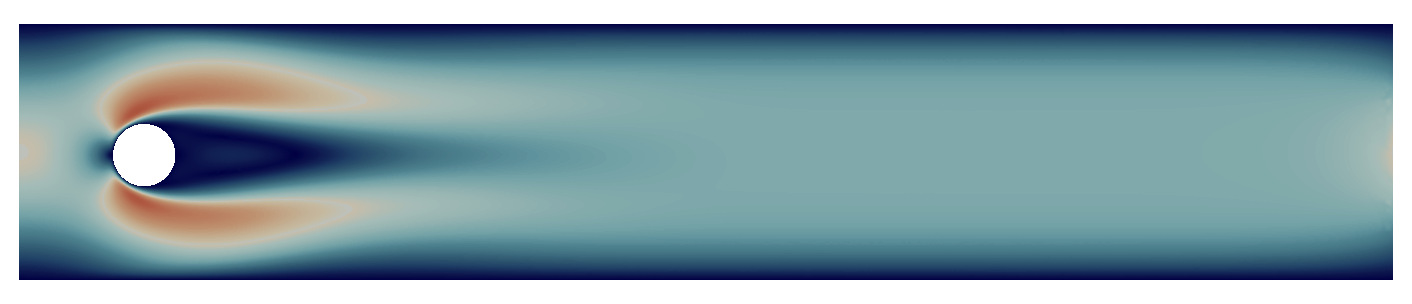}
\end{subfigure}
\begin{subfigure}{\linewidth}
   \centering
   \includegraphics[width = 1\linewidth]{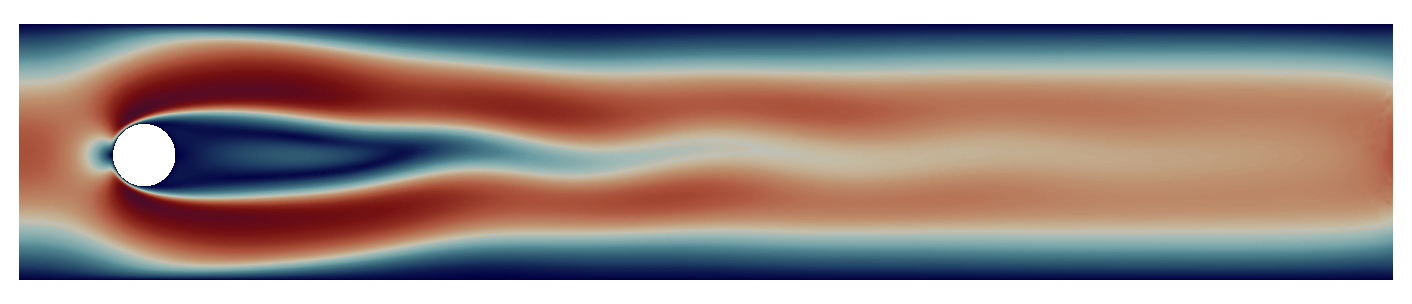}
\end{subfigure}
\begin{subfigure}{\linewidth}
   \centering
   \includegraphics[width = 1\linewidth]{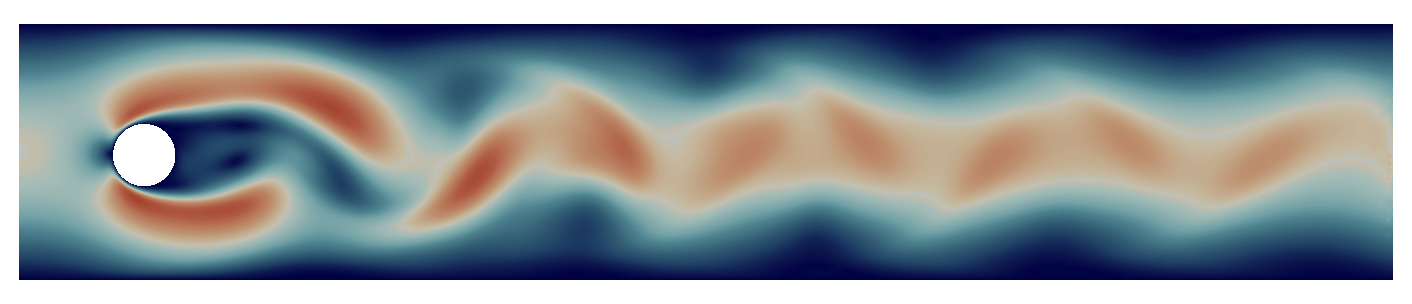}
\end{subfigure}
\begin{subfigure}{\linewidth}
   \centering
   \includegraphics[width = 1\linewidth]{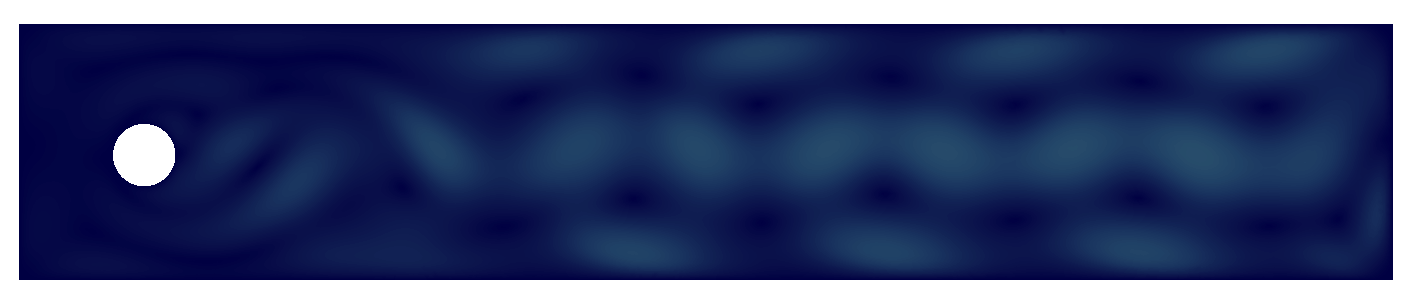}
\end{subfigure}
\begin{subfigure}{\linewidth}
   \centering
   \includegraphics[width = 1\linewidth]{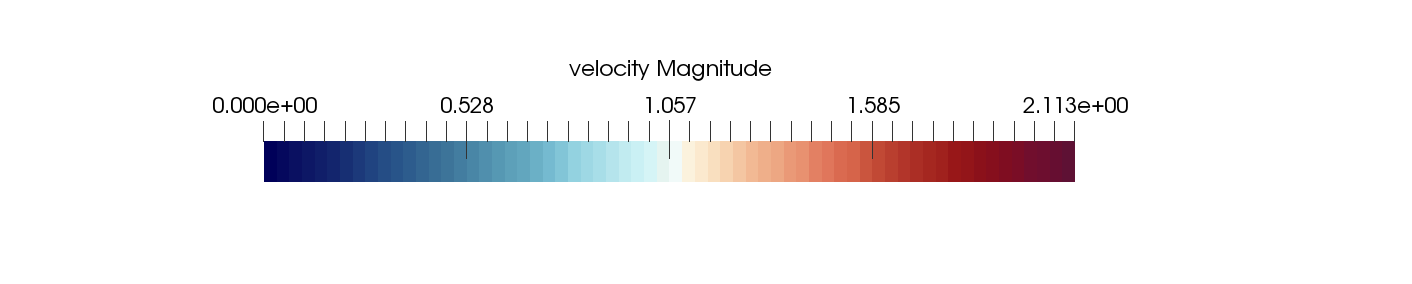}
\end{subfigure}
\caption{Snapshots of the flow past a cylinder solution, described in Section \ref{sec:test_benchmark}, at times (from top to bottom) $t = $2, 4, 6, and 8.\label{fig:flow-past-cylinder}}
\end{figure}

Time dependent boundary conditions present questions for both multi-step and
multi-stage methods. Here it also means that the numerical solution will
satisfy filtered BCs rather than their exact values since we applied the
filtering steps\ as written with no special treatment of the inflow. The
error committed at the boundaries is still consistent up to the order of the
method and no problem was observed. 

The flow configuration, \cite{ST96}, \cite{J04}, is as follows. The
kinematic viscosity $\nu =10^{-3}$, the final time is $T_f = 8$, and the
domain is

\begin{equation}
\Omega =\{(x,y)\,|\,0<x<2.2,\hspace{2mm}0<y<0.41\text{ and }%
(x-0.2)^{2}+(y-0.2)^{2}>0.05^{2}\}.  \notag
\end{equation}%
The external body force $f$ is set to zero. The inflow and outflow
velocities are parabolic: 
\begin{equation*}
u|_{(0,y)}=u|_{(2.2,y)}=0.41^{-2}(6y(0.41-y),0).
\end{equation*}%
We use $(\mathcal{P}_{2}^{2},\mathcal{P}_{1}^{1})$ elements with a well
resolved, static mesh with 479,026 degrees of freedom with 1000 edges on the
interior cylinder boundary. This is the same mesh used in \cite{DLZ18} and
was generated by adaptive refinement from solving the steady problem.

We measure the maximum drag $c_{d,\max }$, the time of maximum drag $%
t(c_{d,\max })$, maximum lift $c_{l,\max }$, time of maximum lift $%
t(c_{l,\max })$ and the pressure drop at the final time between the front
and back of the cylinder, $\Delta p(8)=p_{h}(0.15,0.2)-p_{h}(0.25,0.2)$. We
run the tests for the same $\Delta t^{\prime }s$ in \cite{J04}, which have a
largest value of $\Delta t=0.04$, and are successively halved until the
smallest value of $\Delta t=0.00125$. The stepsizes are doubled for IE-EIS-3
for a fair comparison since it requires two implicit Euler solves for one
timestep. If a simulation failed, the missing values are filled in with
dashes in the tables. Failure only happened for IE-Pre-Post-3{} and
MP-Pre-Post-4{} when the energy of the solution grew which was followed by
both Newton and fixed-point iterations failing in the nonlinear solve.

The results for the methods based on IE are shown in Table \ref{tab:ie},
results for methods based on BDF2 are shown in Table \ref{tab:bdf2}, and
results for methods based on MP are shown in Table \ref{tab:mp}. 

For the IE based methods, every method shows improvement over IE in the prediction of the maximum lift with the exception IE-Pre-Post-3 for the larger stepsizes. IE-Pre-Post-3 was the only IE based method to show instability, and the simulation failed to run to completion for larger stepsizes.  IE-EIS-3 showed superior accuracy at the smallest stepsize in the lift coefficient and pressure drop. 

For the BDF2 based methods, both BDF2-Post-3 and BDF2-Pre-Post-3 show a dramatic improvement in predicting the final pressure drop over BDF2. Interestingly, BDF2-Pre-Post-3 exhibits better convergence of pressure than was suggested by the test in Section \ref{sec:test_analytic}.

For the MP based methods, MP and MP-Pre-Post-2 yielded essentially identical results. The most noticable improvement over the base method is MP-Pre-Post-3's pressure drop which matches all four digits of the reference values for the smallest three $\Delta t$s. For all the methods, the maximum lift coefficient appears to be converging to a value slightly higher than the reference value. MP-Pre-Post-4 did not finish for $\Delta t$s higher than 0.0025.
\pagebreak
{\small
\begin{longtable}
{| p{1.5cm}p{1.5cm}p{2cm}p{1.5cm}p{2cm}p{1.5cm}|} \hline
\multicolumn{6}{|c|}{IE based methods.\label{tab:ie}}\\ 
$\Delta t$ & $t(c_{d,\max})$ & $c_{d,\max}$ & $t(c_{l,\max})$ & $c_{l,\max}$
& $\Delta p(8)$ \\
\hline
\multicolumn{6}{|c|}{Reference Values}\\
--- & 3.93625 & 2.950921575 & 5.693125 & 0.47795 & -0.1116 \\
\hline\multicolumn{6}{|c|}{IE}\\
0.005& 3.93& 2.950301672& 6.28& 0.17604& -0.1005 \\ 
0.0025& 3.9325& 2.950371110& 6.215& 0.30336& -0.1070 \\ 
0.00125& 3.93375& 2.950529384& 5.7175& 0.38229& -0.1114 \\ 
\hline\multicolumn{6}{|c|}{IE-Pre-2 (one IE solve and one filter per timestep)}\\
0.005& 3.935& 2.950802171& 5.72& 0.45978& -0.1111 \\ 
0.0025& 3.935& 2.950872330& 5.7& 0.47413& -0.1120 \\ 
0.00125& 3.93625& 2.950889791& 5.695& 0.47728& -0.1117 \\ 
\hline\multicolumn{6}{|c|}{IE-Filt($\frac{3-\sqrt{3}}{3}$) (one IE solve and two filters per timestep)}\\
0.005& 3.93& 2.950839424& 5.71& 0.46722& -0.1127 \\
0.0025& 3.9325& 2.950880844& 5.695& 0.47567& -0.1124 \\
0.00125& 3.935& 2.950891744& 5.6925& 0.47762& -0.1120 \\

\hline\multicolumn{6}{|c|}{IE-Pre-Post-3 (one IE solve and two filter per timestep)}\\
0.005& 7.825& 435.1275230& 7.84& 205.33324& -5.1966 \\ 
0.0025& 3.935& 2.950897874& 5.6925& 0.47895& -0.1116 \\ 
0.00125& 3.93625& 2.950895596& 5.6925& 0.47833& -0.1116 \\ 
\hline\multicolumn{6}{|c|}{IE-EIS-3 (two IE solves and two filters per timestep)}\\
0.01& 3.93667& 2.950816639& 5.69667& 0.46183& -0.1119 \\ 
0.005& 3.935& 2.950884378& 5.69333& 0.47608& -0.1117 \\ 
0.0025& 3.93667& 2.950893844& 5.6925& 0.47797& -0.1116 \\ \hline
\caption{For the smallest stepsize, the lift calculated by IE does not have any digits agreement for the smallest $\Delta t$, but the methods based on it have at least two digits of accuracy. IE-EIS-3 has the best agreement with the reference lift and pressure drop values for the smallest $\Delta t$. IE-Pre-Post-3 is unstable and/or did not finish for several stepsizes. } 
\label{tab:myfirstlongtable2}
\end{longtable}
}

{\small
\begin{longtable}{ |p{1.5cm}p{1.5cm}p{2cm}p{1.5cm}p{2cm}p{1.5cm}|} 
\hline\multicolumn{6}{|c|}{BDF2 based methods.}\\
$\Delta t$ & $t(c_{d,\max})$ & $c_{d,\max}$ & $t(c_{l,\max})$ & $c_{l,\max}$
& $\Delta p(8)$ \\
\hline
\multicolumn{6}{|c|}{Reference Values}\\
--- & 3.93625 & 2.950921575 & 5.693125 & 0.47795 & -0.1116 \\
\hline\multicolumn{6}{|c|}{BDF2}\\
0.02& 3.94& 2.950423752& 5.86& 0.34749& -0.1063 \\ 
0.005& 3.935& 2.950858401& 5.705& 0.47141& -0.1120 \\ 
0.00125& 3.93625& 2.950893074& 5.69375& 0.47787& -0.1117 \\ 
\hline\multicolumn{6}{|c|}{BDF2-Post-3 (one BDF2 solve and one filter per timestep)}\\
0.02& 3.94& 2.951551390& 5.7& 0.56819& -0.1119 \\ 
0.005& 3.935& 2.950905566& 5.695& 0.48028& -0.1116 \\ 
0.00125& 3.93625& 2.950895363& 5.6925& 0.47828& -0.1116 \\ 
\hline\multicolumn{6}{|c|}{BDF2-Pre-Post-3 (one BDF2 solve and two filters per timestep)}\\
0.02& 3.94& 2.950855922& 5.86& 0.44850& -0.1028 \\
0.005& 3.935& 2.950932030& 5.695& 0.48538& -0.1116 \\
0.00125& 3.93625& 2.950896341& 5.6925& 0.47845& -0.1116 \\ \hline
\caption{BDF2-Post-3 and BDF2-Pre-Post-3 have better agreement with the reference pressure drop than BDF2. BDF2-Post-3 tends to overestimate the lift coefficient for larger stepsizes.} 
\label{tab:bdf2}
\end{longtable}
}

{\small
\begin{longtable}{| p{1.5cm}p{1.5cm}p{2cm}p{1.5cm}p{2cm}p{1.5cm}|} \hline
\multicolumn{6}{|c|}{Methods based on the implicit midpoint rule}\\
$\Delta t$ & $t(c_{d,\max})$ & $c_{d,\max}$ & $t(c_{l,\max})$ & $c_{l,\max}$
& $\Delta p(8)$ \\
\hline
\multicolumn{6}{|c|}{Reference Values}\\
--- & 3.93625 & 2.950921575 & 5.693125 & 0.47795 & $-$0.1116 \\
\hline\multicolumn{6}{|c|}{MP}\\
0.005& 3.935& 2.950893884& 5.695& 0.47780& -0.1118 \\ 
0.0025& 3.9375& 2.950894329& 5.6925& 0.47809& -0.1117 \\ 
0.00125& 3.93625& 2.950895120& 5.6925& 0.47821& -0.1116 \\  
\hline\multicolumn{6}{|c|}{MP-Pre-Post-2 (one MP solve and two filters per timestep)}\\
0.005& 3.935& 2.950886785& 5.695& 0.47683& -0.1118 \\ 
0.0025& 3.935& 2.950892638& 5.6925& 0.47785& -0.1117 \\ 
0.00125& 3.93625& 2.950894681& 5.6925& 0.47815& -0.1116 \\ 
\hline\multicolumn{6}{|c|}{MP-Pre-Post-3 (one MP solve and two filters per timestep)}\\
0.005& 3.935& 2.950896750& 5.695& 0.47840& -0.1116 \\ 
0.0025& 3.935& 2.950894895& 5.6925& 0.47830& -0.1116 \\ 
0.00125& 3.93625& 2.950895215& 5.6925& 0.47825& -0.1116 \\ 
\hline\multicolumn{6}{|c|}{MP-Pre-Post-4 (one MP solve and two filters per timestep)}\\
0.005& ---& ---& ---& ---& --- \\ 
0.0025& 3.935& 2.950894654& 5.6925& 0.47824& -0.1116 \\ 
0.00125& 3.93625& 2.950895183& 5.6925& 0.47824& -0.1116 \\  \hline
\caption{MP-Pre-Post-3 has exact agreement with the reference pressure drop for the smallest three stepsizes. MP and MP-Pre-Post-2 give nearly identical values.  MP-Pre-Post-4 was unstable for $\Delta t$ larger than $0.0025$.} 
\label{tab:mp}
\end{longtable}
}

\subsection{Offset cylinder test\label{sec:test_offset}}

We now consider a body forced internal flow between two offset cylinders
inspired by e.g. \cite{EP00}. This flow is transitional between periodic and
turbulent. Similar tests have been used in, e.g., \cite{J15}. The domain is
a unit cylinder centered at the origin, minus a smaller cylinder of radius
0.1 centered at $(0,0.5)$. The flow is forced by

\begin{equation}
f(t,x,y) = 4\min(1,t)(1-x^2 - y^2) \left<-y , x \right>.  \notag
\end{equation}
The kinematic viscosity is set to $\nu = \frac{1}{150}$. The flow is started
at rest with an initial condition is $u \equiv 0$. The simulation is run
from $t \in [0, 30]$. We first obtain accurate reference data using the
implicit midpoint rule for several mesh refinements and step sizes, with a
fine mesh refinement of 255,870 degrees of freedom using $(\mathcal{P}_2^2, 
\mathcal{P}_1^1)$ elements, and a smallest stepsize of $\Delta t = 0.000625$%
. We capture the evolution of the kinetic energy, several snapshots, and
probe each component of velocity at several points.

For timesteps where the base methods fail to capture the correct dynamics
the filtered methods capture the reference solution better. Solution
snapshots for IE based methods are shown in Figure \ref{fig:ie_offset_cylinder}, and the snapshots for BDF2 and MP based methods are shown in Figure \ref{fig:bdf2_offset_cylinder}.

For a quantitative comparison, we measure the error of the $x$-component of
the velocity vector at the point $(0.5, 0.2)$, which is in the wake of the
obstacle. Since the solution is quasi-periodic, we normalize the error by
the difference between the maximum and minimum values taken on by the
reference solution. To calculate the error, let $u = u(t,x,y)$ be the
reference solution and $u_h(t,x,y)$ is the discrete solution, which has been
linearly interpolated in time. The absolute and relative error at time $t$ is defined as follows:
\begin{align*}
\text{absolute error}(t) &= \sup_{s \in [0,t]}|u(s,0.5,0.2) - u_h(s,0.5,0.2)|,
\\
\text{relative error}(t) &= \frac{\text{absolute error(t)}}{\sup_{s \in [0,T]}u(s,0.5,0.2) -
\inf_{s \in [0,T]} u(t,0.5,0.2)}.
\end{align*}
The reference $x$-component of the velocity is shown in Figure \ref%
{fig:x-component-reference}. The errors of the IE, BDF2, and MP based
methods are given in Figures \ref{fig:ie-x-component}, \ref%
{fig:bdf2-x-component}, and \ref{fig:mp-x-component}.

IE-Pre-Post-3, MP-Pre-Post-3 and MP-Pre-Post-4 were unstable for some or all
of the stepsizes shown. The solution exhibits the richness of scales typical
for higher Reynolds number flows. One consequence may be increased
importance of A-stability and instabilities seen in methods that are only $%
A(\alpha )$ stable. \emph{We have derived the methods herein by optimization of
accuracy. Deriving methods by optimizing stability instead can be done using
the tools herein.} Doing so is an important open question at higher Reynolds
numbers.

\begin{figure}[th]
\centering
\begin{subfigure}{0.49\linewidth}
   \centering
   \includegraphics[width = 1\linewidth]{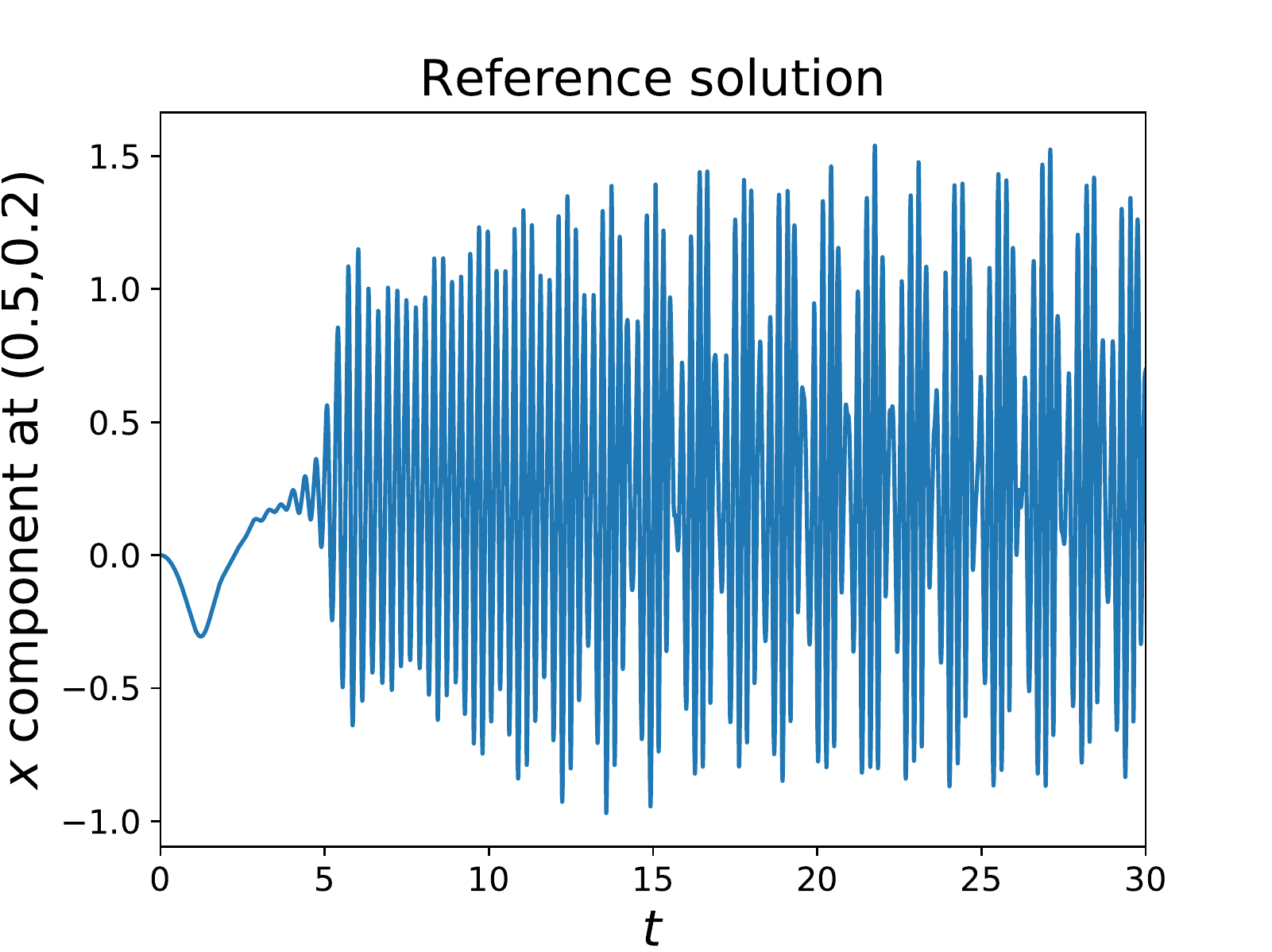}
\end{subfigure}
\caption{The $x$ component of the reference solution shows quasi-periodic
behavior.}
\label{fig:x-component-reference}
\end{figure}
\begin{figure}[th]
\begin{subfigure}{0.49\linewidth}
   \centering
   \includegraphics[width = 1\linewidth]{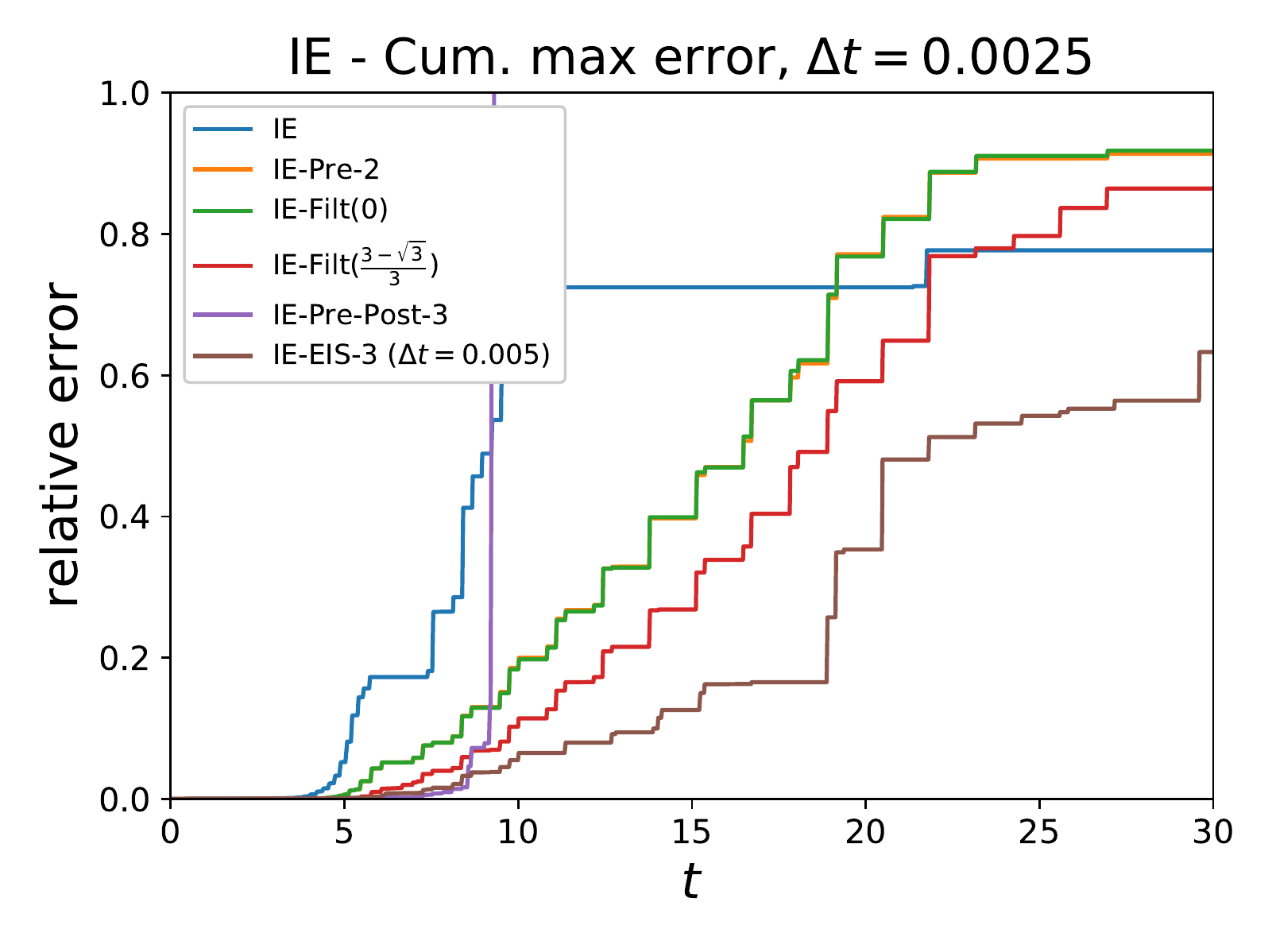}
\end{subfigure}
\begin{subfigure}{0.49\linewidth}
   \centering
   \includegraphics[width = 1\linewidth]{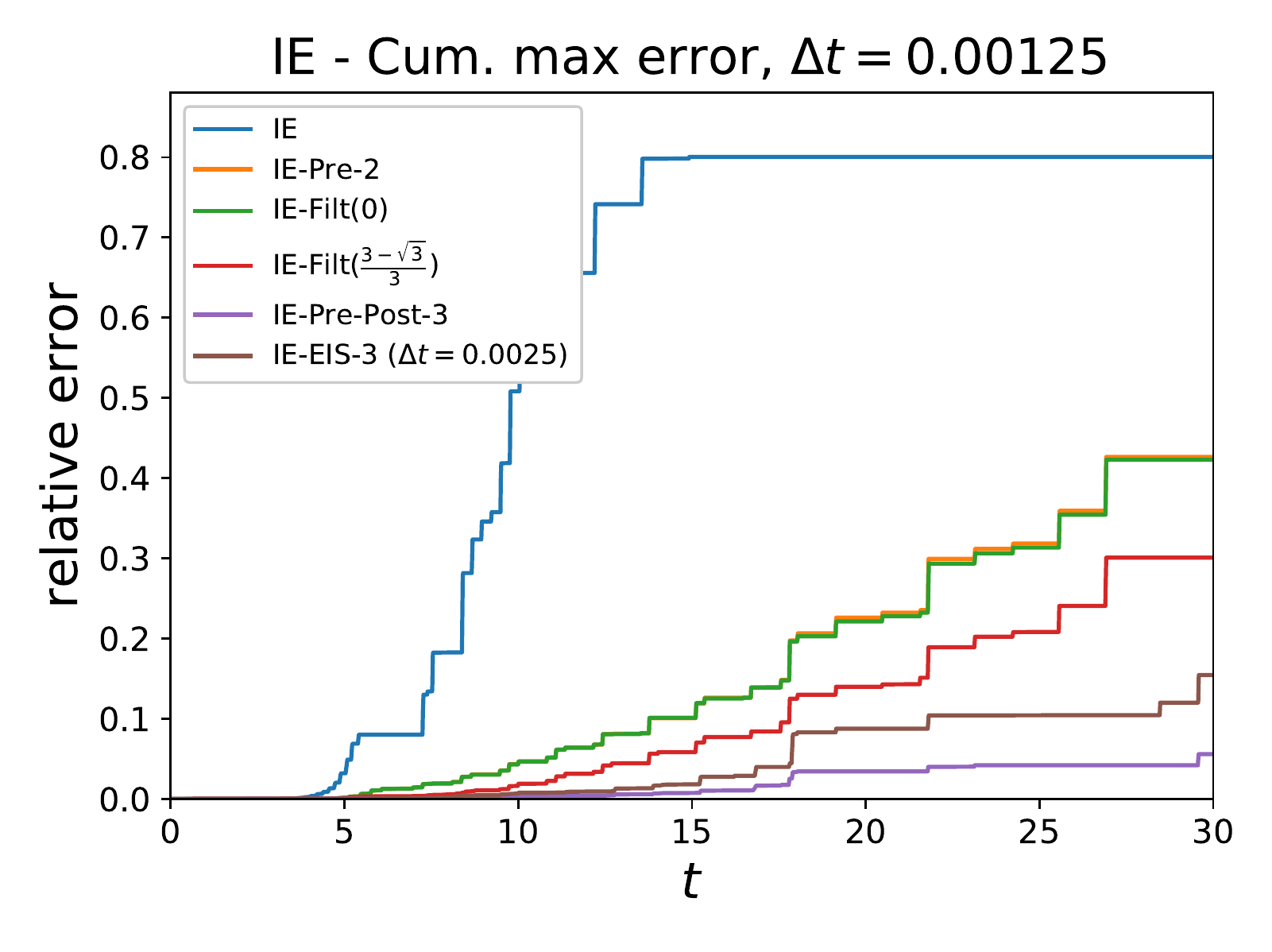}
\end{subfigure}
\caption{For $\Delta t = 0.0025$, IE-Pre-Post-3 is unstable while IE-EIS-3
performs the best. For $\Delta t = 0.00125$, IE-Pre-Post-3 becomes stable and
is the most accurate solution.}
\label{fig:ie-x-component}
\end{figure}
\begin{figure}[th]
\begin{subfigure}{0.49\linewidth}
   \centering
   \includegraphics[width = 1\linewidth]{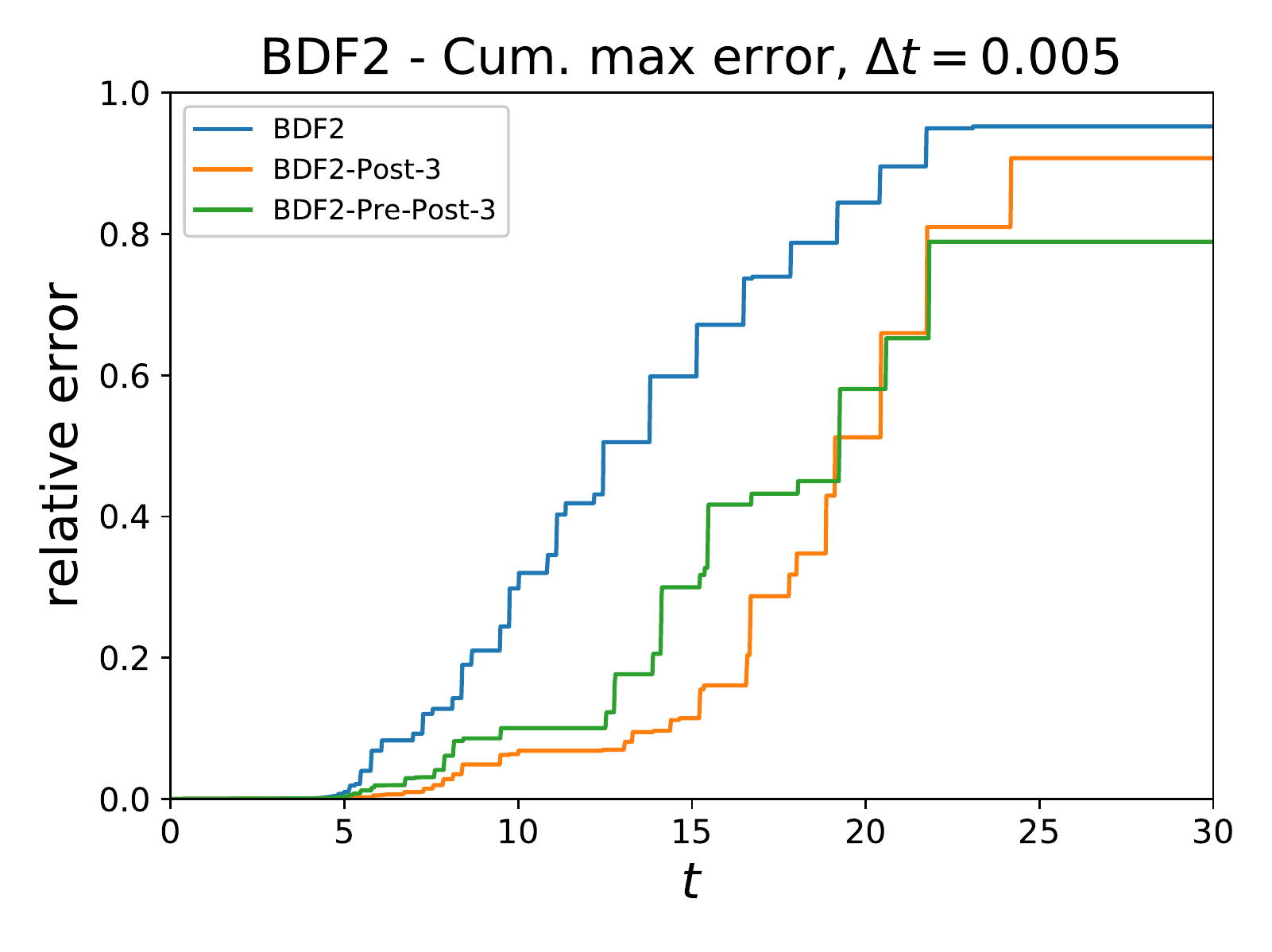}
\end{subfigure}
\begin{subfigure}{0.49\linewidth}
   \centering
   \includegraphics[width = 1\linewidth]{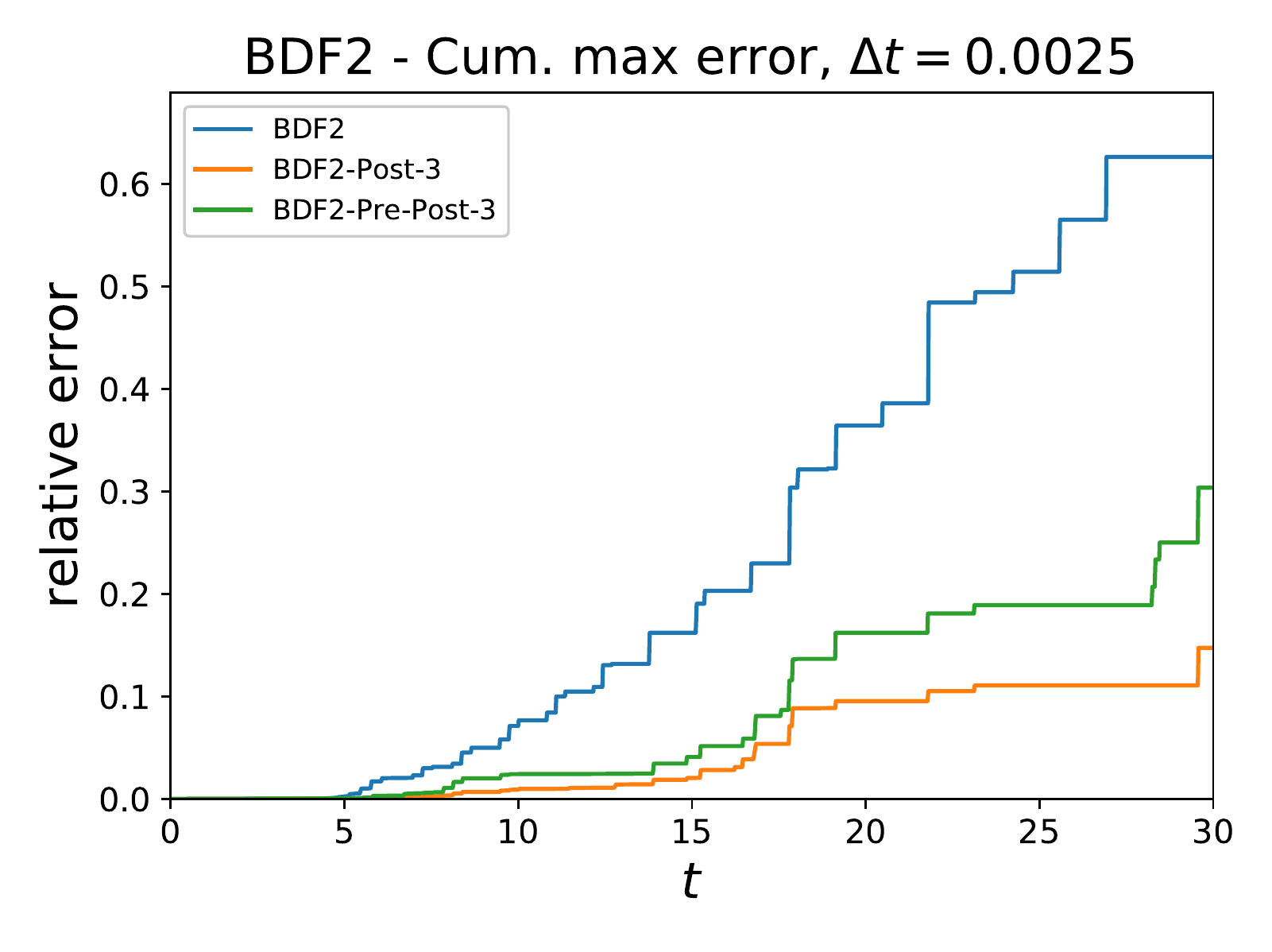}
\end{subfigure}
\caption{The third order methods are more accurate than their base method,
BDF2.}
\label{fig:bdf2-x-component}
\end{figure}
\begin{figure}[th]
\begin{subfigure}{0.49\linewidth}
   \centering
   \includegraphics[width = 1\linewidth]{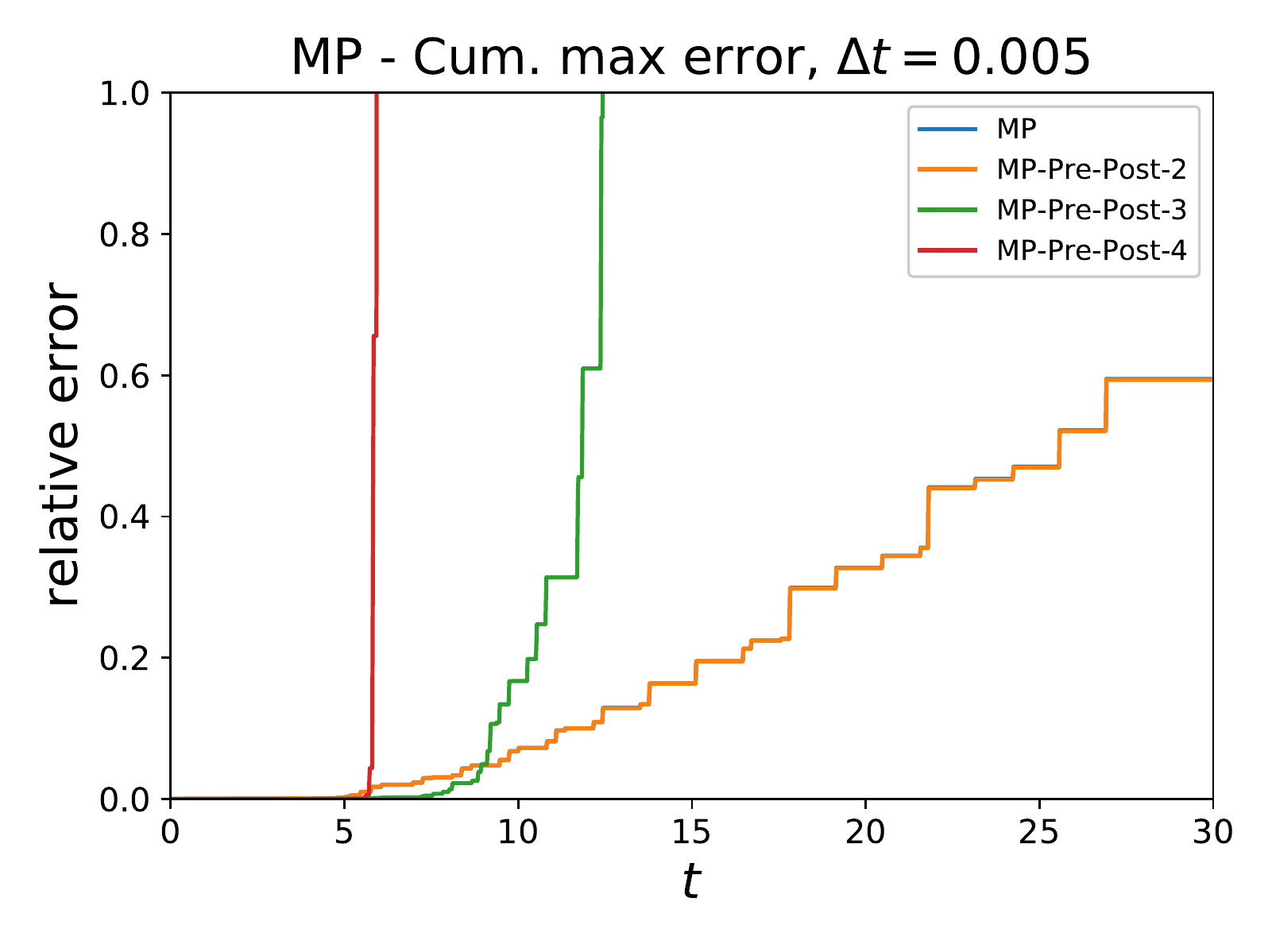}
\end{subfigure}
\begin{subfigure}{0.49\linewidth}
   \centering
   \includegraphics[width = 1\linewidth]{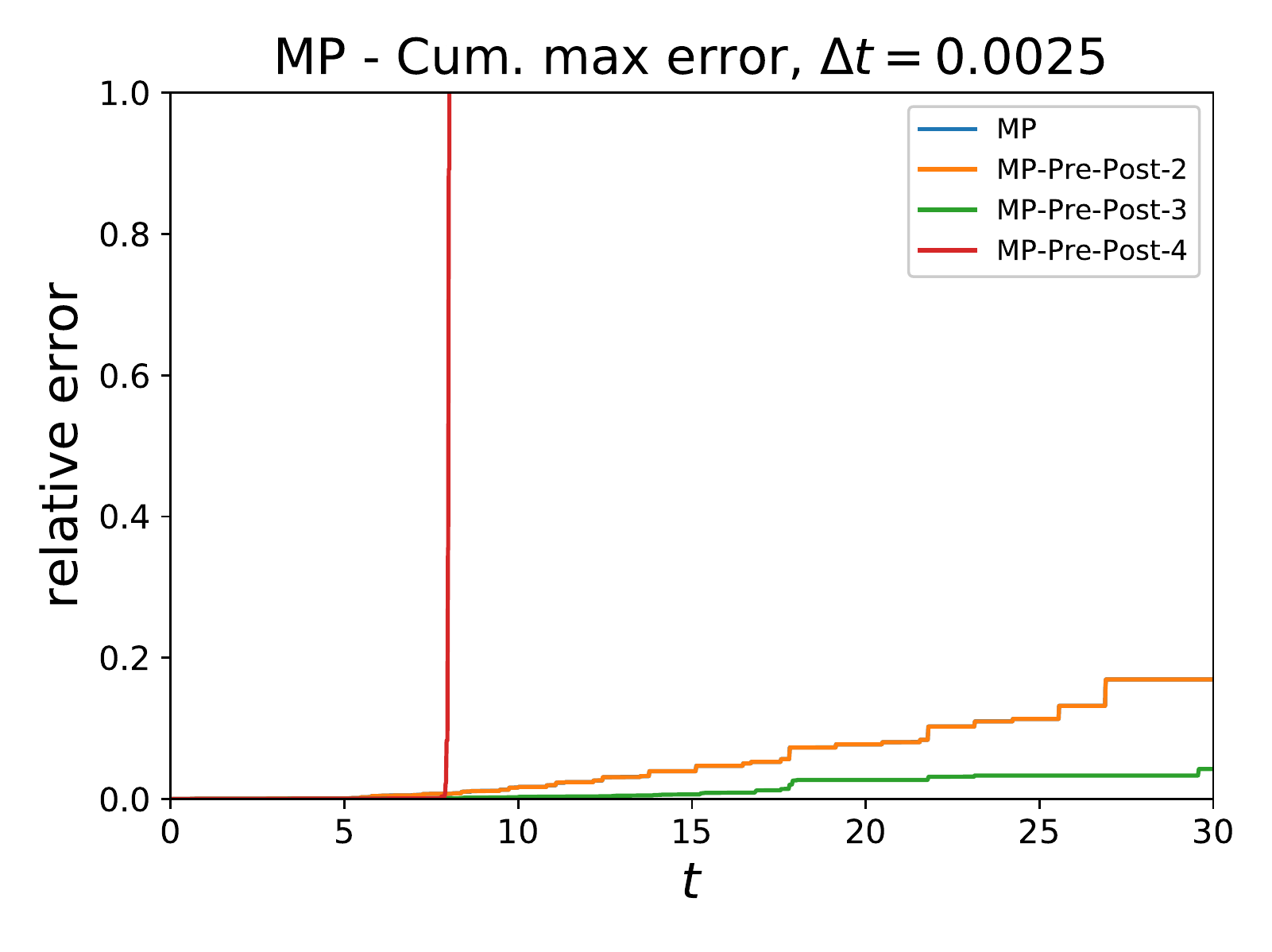}
\end{subfigure}
\par

\caption{For both $\Delta t$s, MP and MP-Pre-Post-2 give essentially the
same answer and MP-Pre-Post-4 is unstable. For $\Delta t = 0.0025$,
MP-Pre-Post-3 is the most accurate, but is unstable when $\Delta t$ is
doubled. IE-Filt($\frac{3-\sqrt{3}}{3}$) performs the best of the second order methods. IE-Pre-2 and IE-Filt(0) are essentially the same.}
\label{fig:mp-x-component}
\end{figure}

\begin{figure}[tbp]
\centering
\begin{tabular}{cM{20mm}M{20mm}M{20mm}M{20mm}M{20mm}}
           \toprule
            $t$ & Ref. & IE & IE-Pre-2 & IE-Filt($\frac{3-\sqrt{3}}{3}$) &IE-EIS-3 \\
            \midrule
            10 & \includegraphics[width = 1\linewidth]{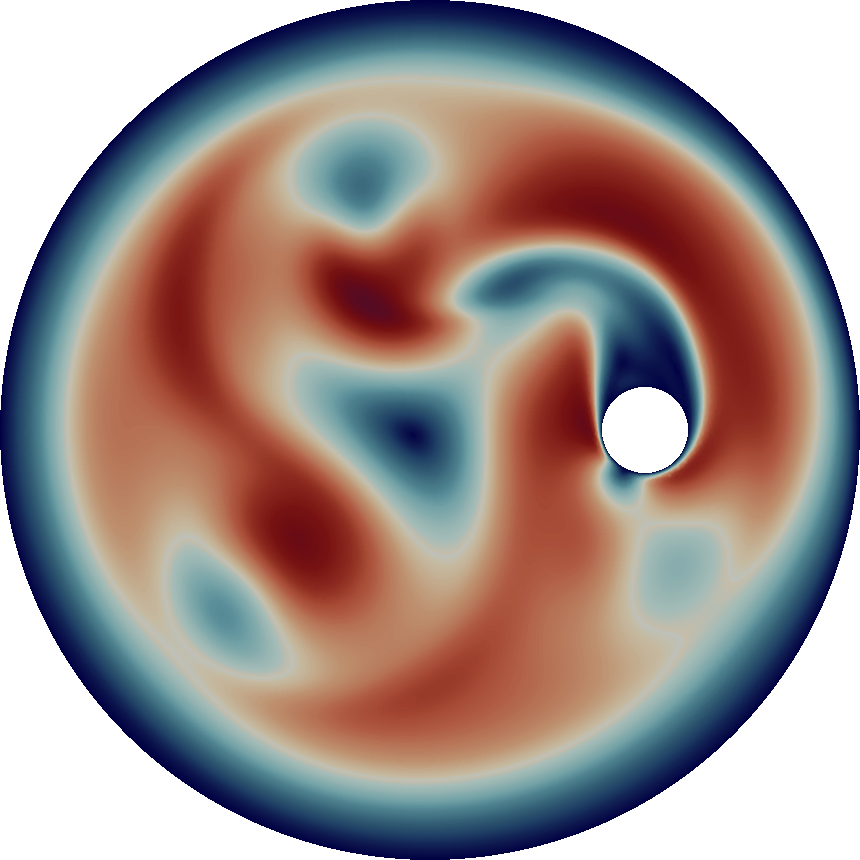} & \includegraphics[width = 1\linewidth]{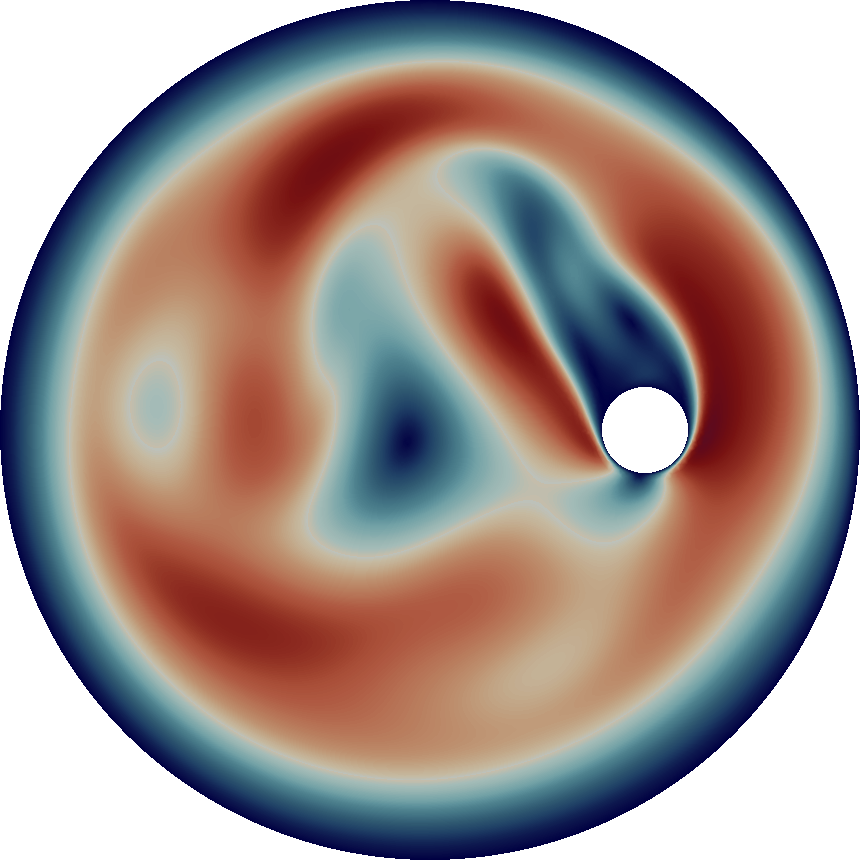} & \includegraphics[width = 1\linewidth]{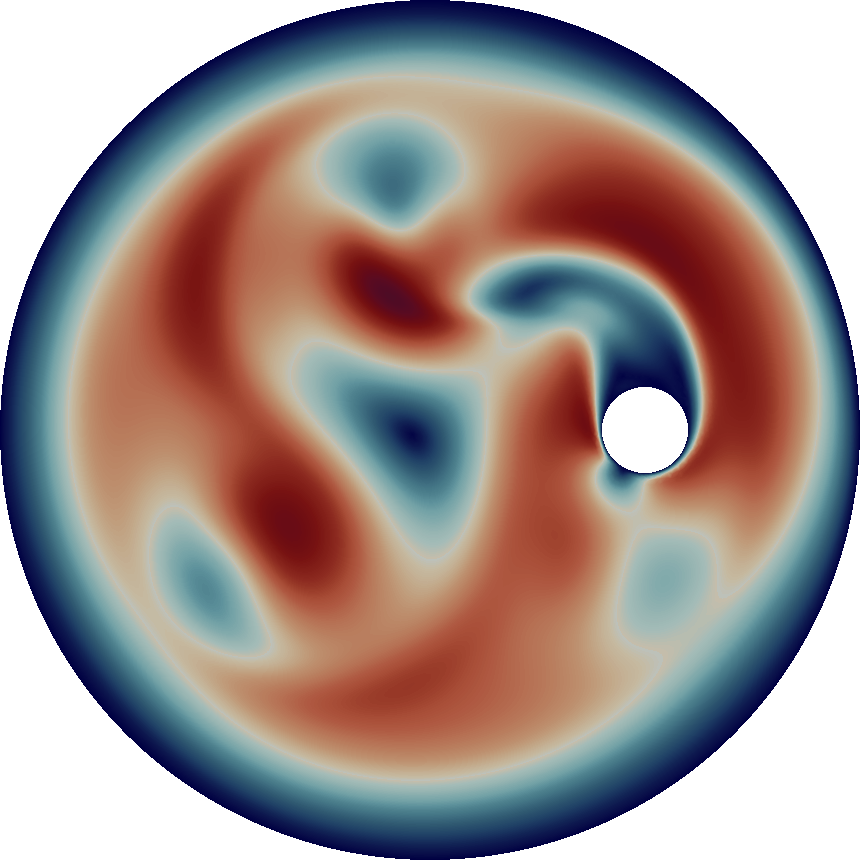} & \includegraphics[width = 1\linewidth]{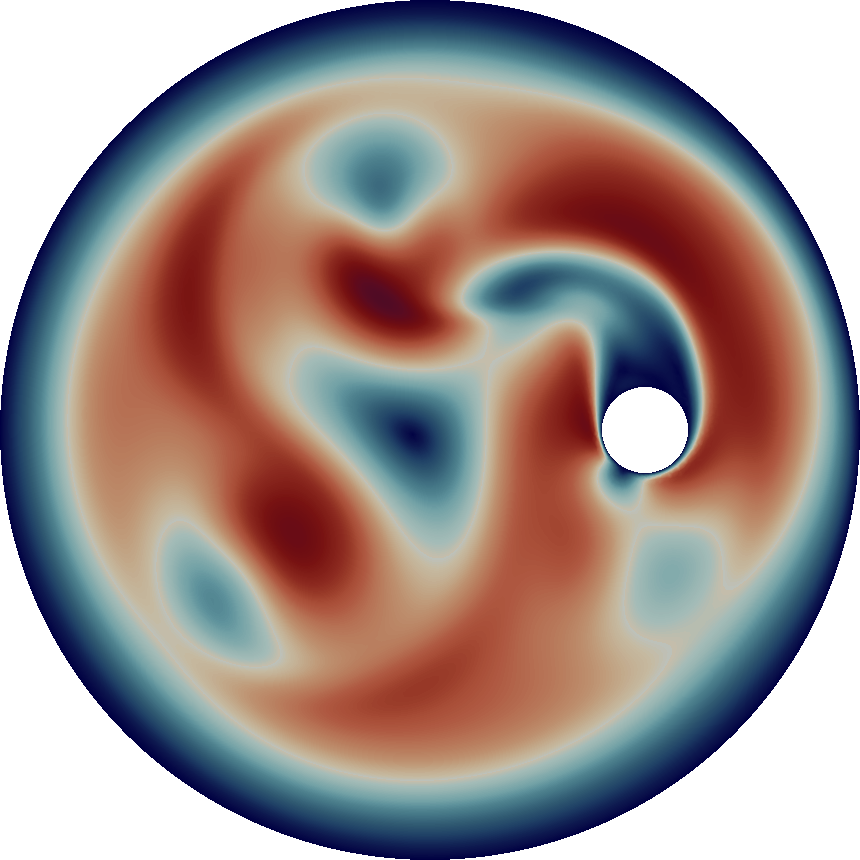} & \includegraphics[width = 1\linewidth]{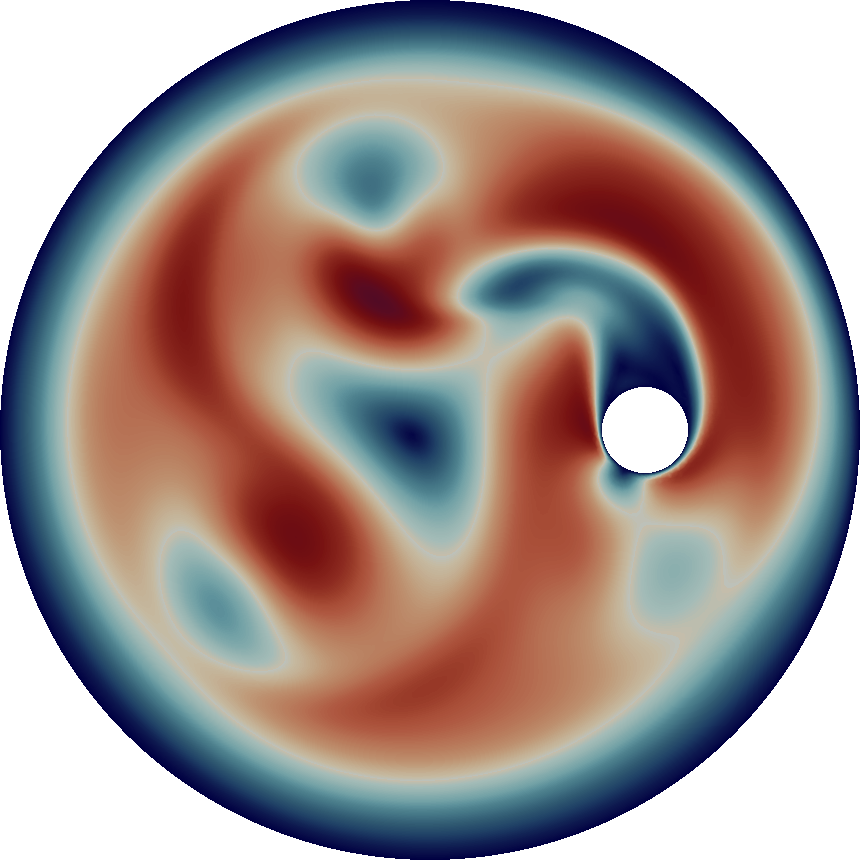}\\
            20 & \includegraphics[width = 1\linewidth]{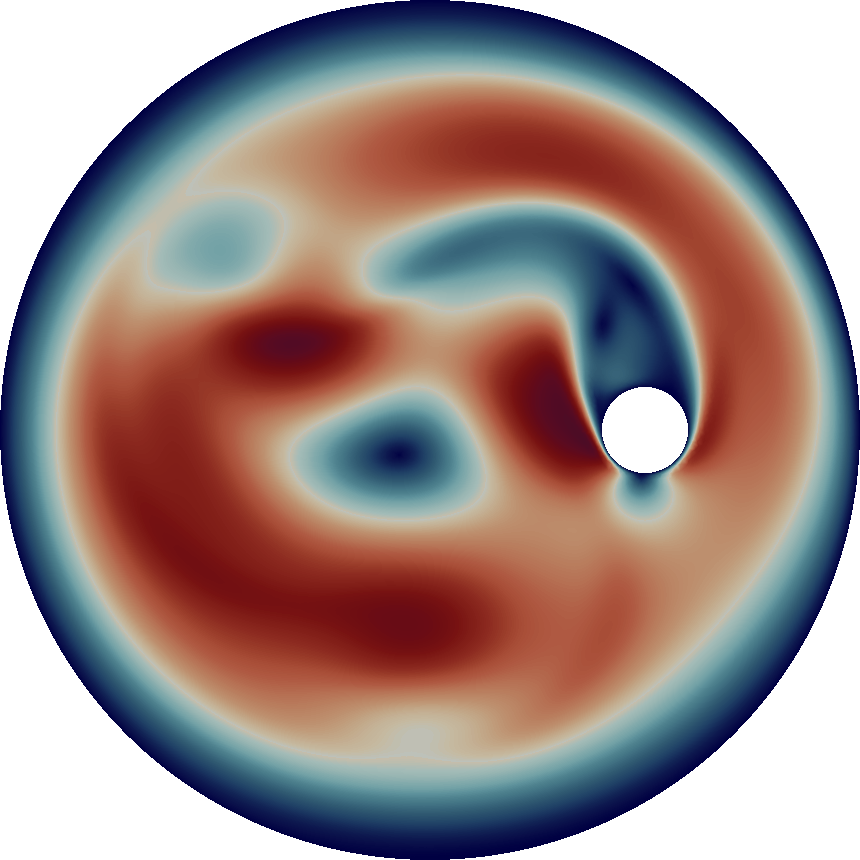} & \includegraphics[width = 1\linewidth]{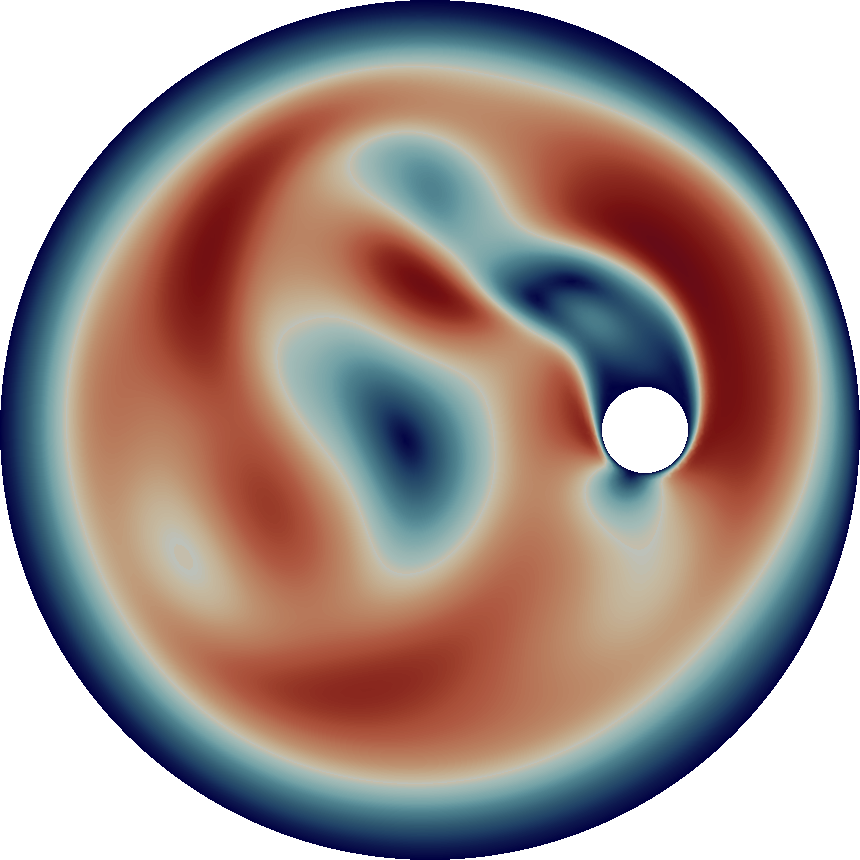} & \includegraphics[width = 1\linewidth]{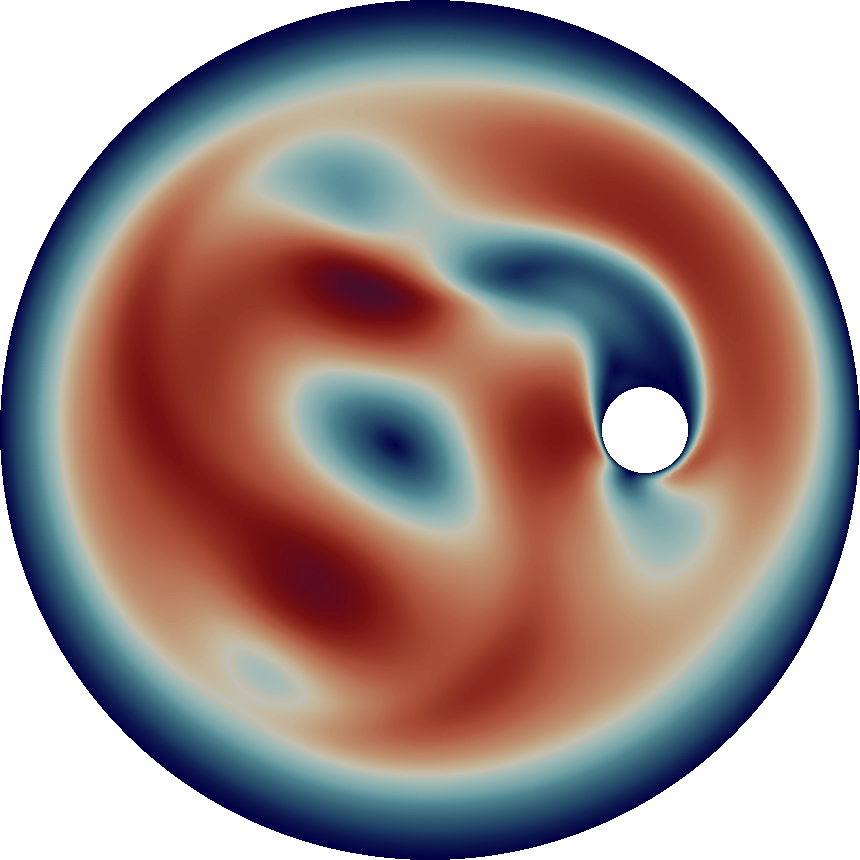} & \includegraphics[width = 1\linewidth]{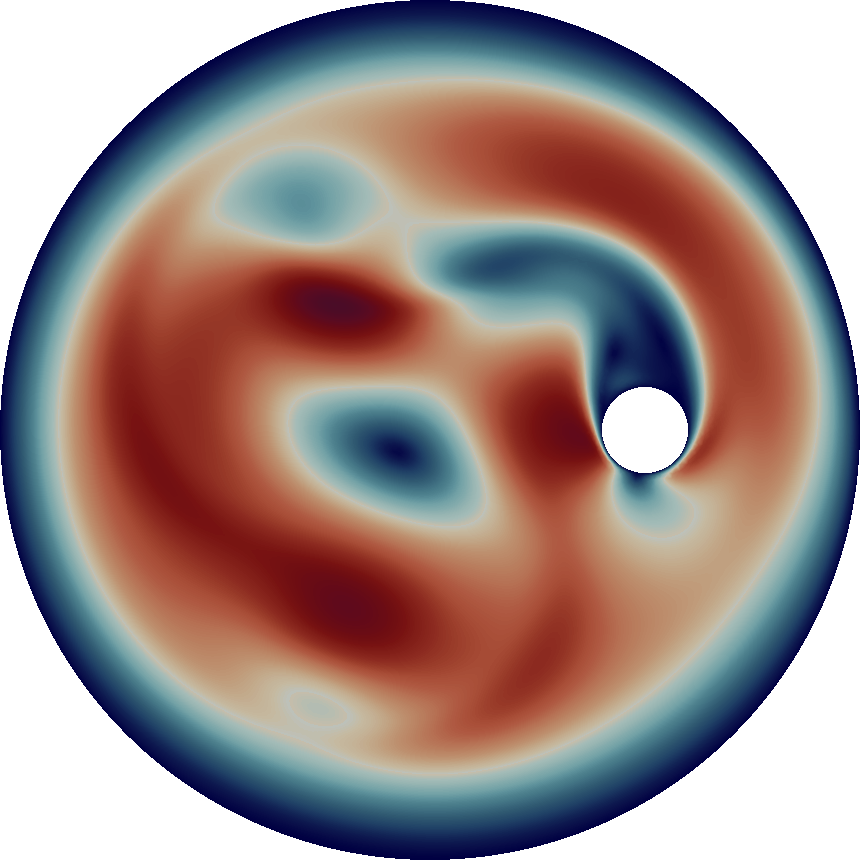} & \includegraphics[width = 1\linewidth]{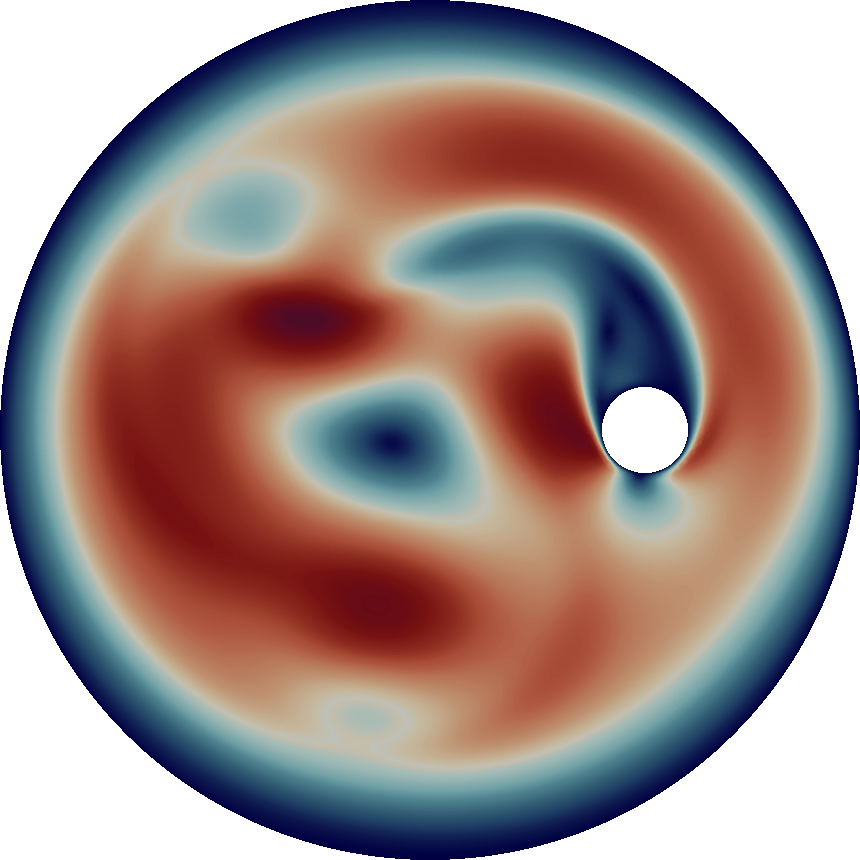} \\
            30 & \includegraphics[width = 1\linewidth]{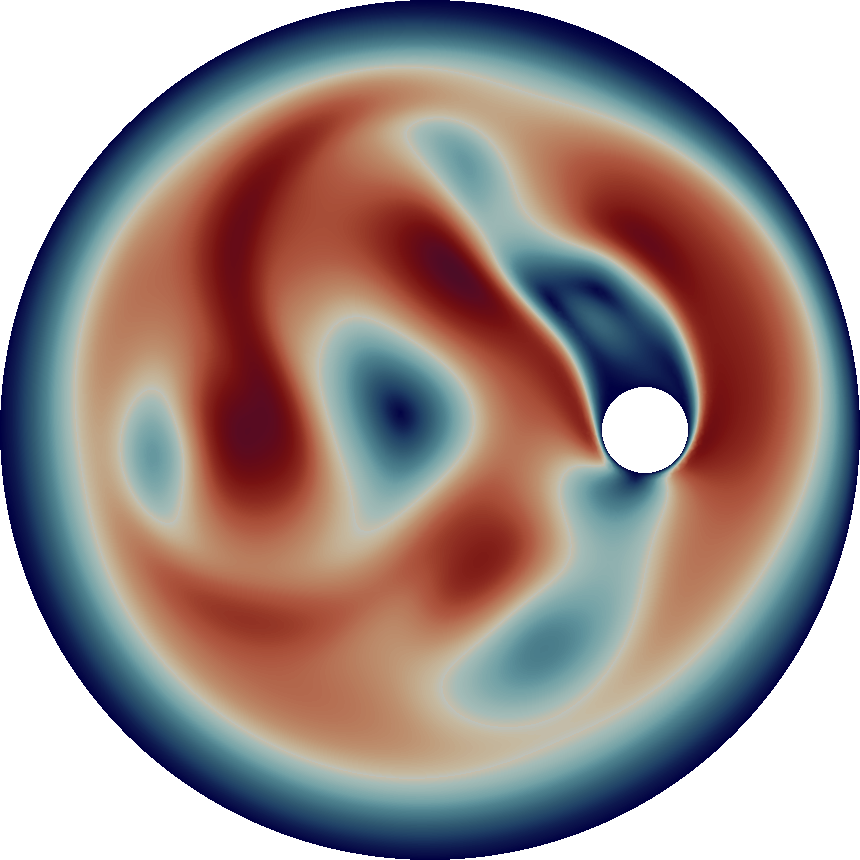} & \includegraphics[width = 1\linewidth]{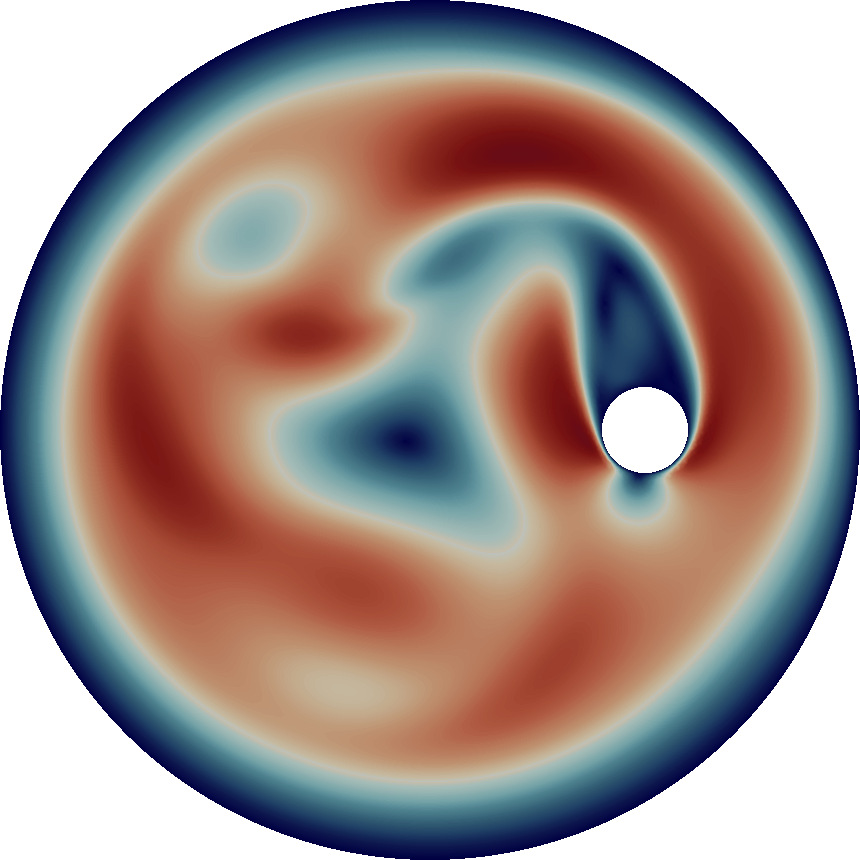} & \includegraphics[width = 1\linewidth]{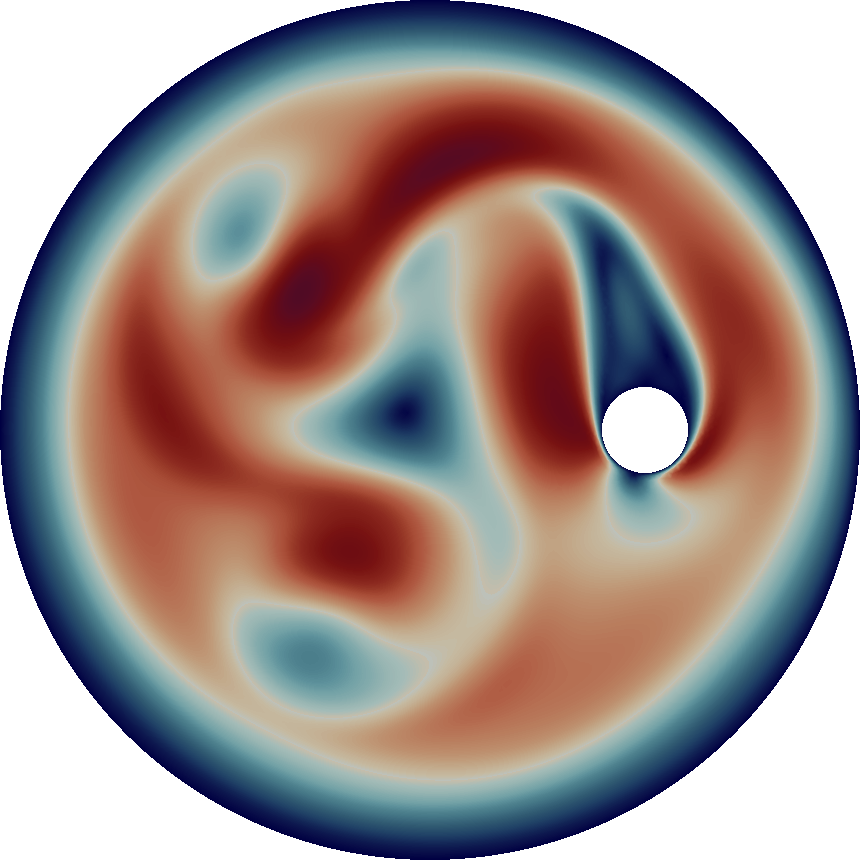} & \includegraphics[width = 1\linewidth]{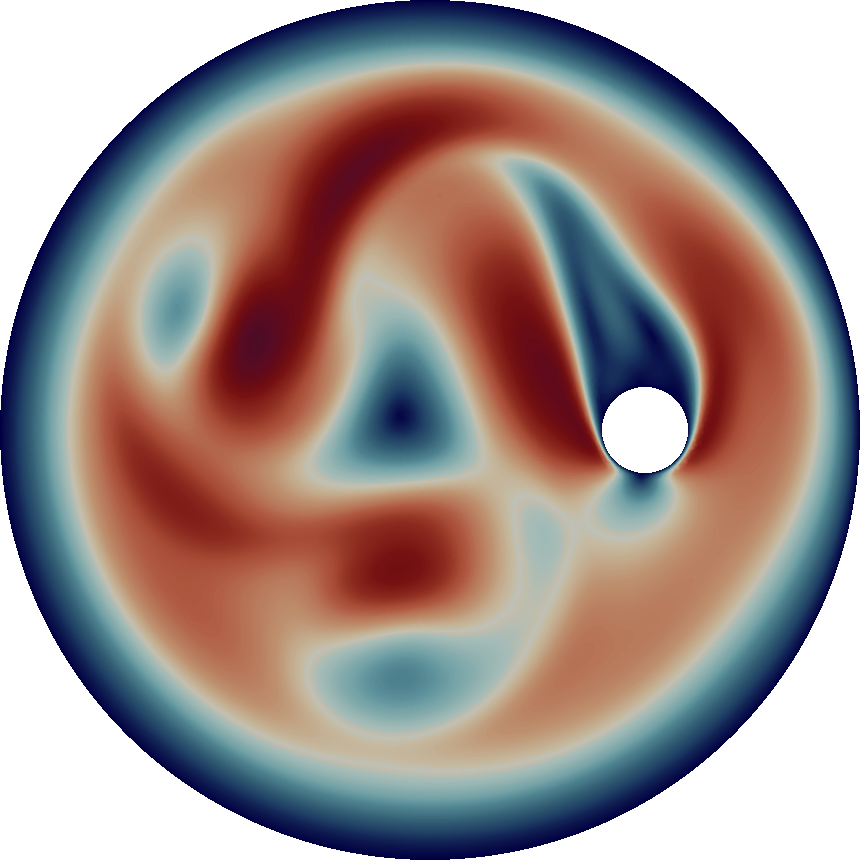} & \includegraphics[width = 1\linewidth]{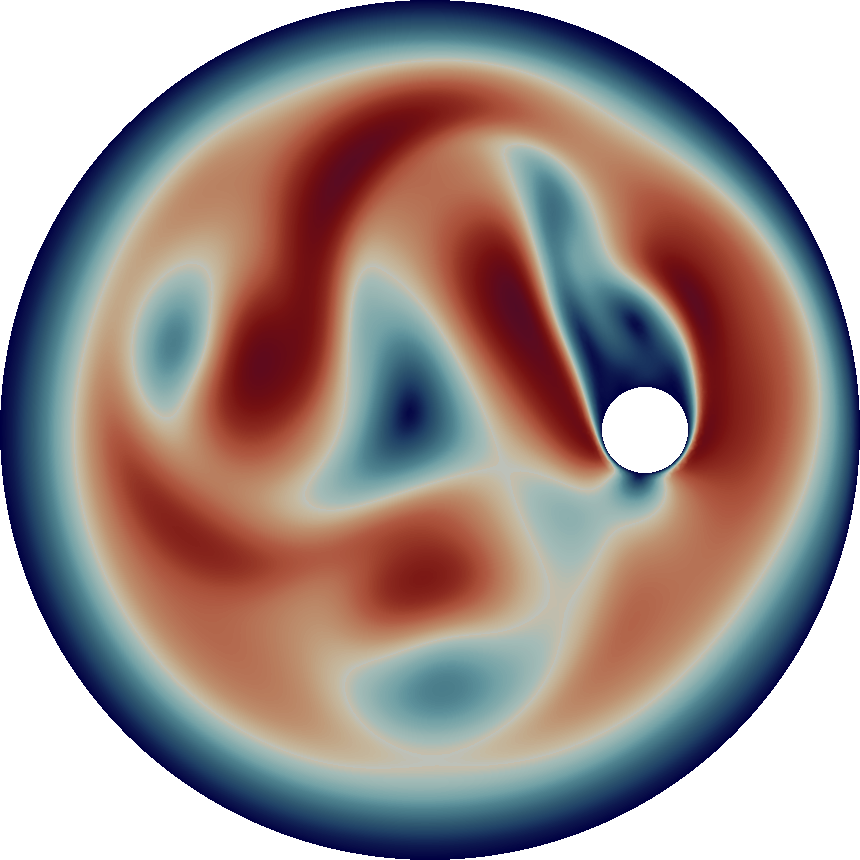}\\
            \bottomrule
        \end{tabular}
\caption{Snapshots of the IE{} based methods with $\Delta t = 0.0025$ (and $%
\Delta t = 0.005$ for IE-EIS-3). All methods showed improvement over IE.
IE-EIS-3{} gave the closest results to the reference solution, as observed,
by smaller phase error at $t = 30$ (for example, compare the detached vortex
shown by IE-EIS-3 versus the attached vortex in the second order methods.).
The IE-Pre-Post-3{} simulation failed before the first snapshot could be
generated. }
\label{fig:ie_offset_cylinder}
\end{figure}

\begin{figure}[tbp]
\centering
\begin{tabular}{cM{20mm}M{20mm}M{20mm}M{20mm}M{20mm}M{20mm}}
           \toprule
            $t$ & Ref. & BDF2 & BDF2-Post-3 & BDF2-Pre-Post-3 & MP-Pre-Post-2 & MP-Pre-Post-3  \\
            \midrule
            10 & \includegraphics[width = 1\linewidth]{figures/t-10.png} & \includegraphics[width = 1\linewidth]{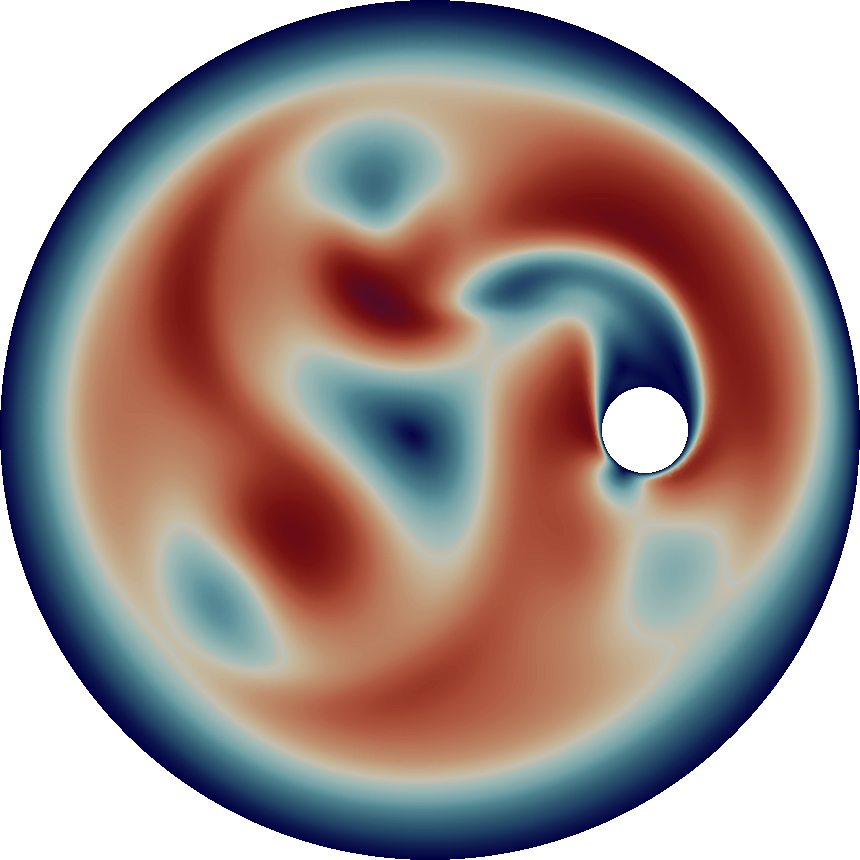} & \includegraphics[width = 1\linewidth]{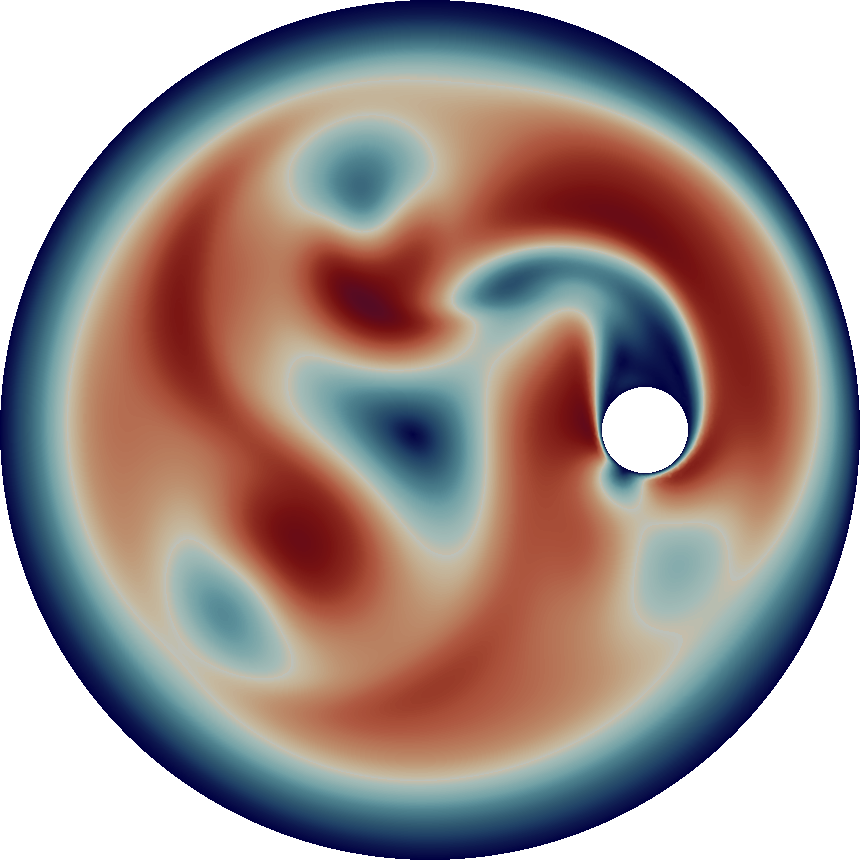} & \includegraphics[width = 1\linewidth]{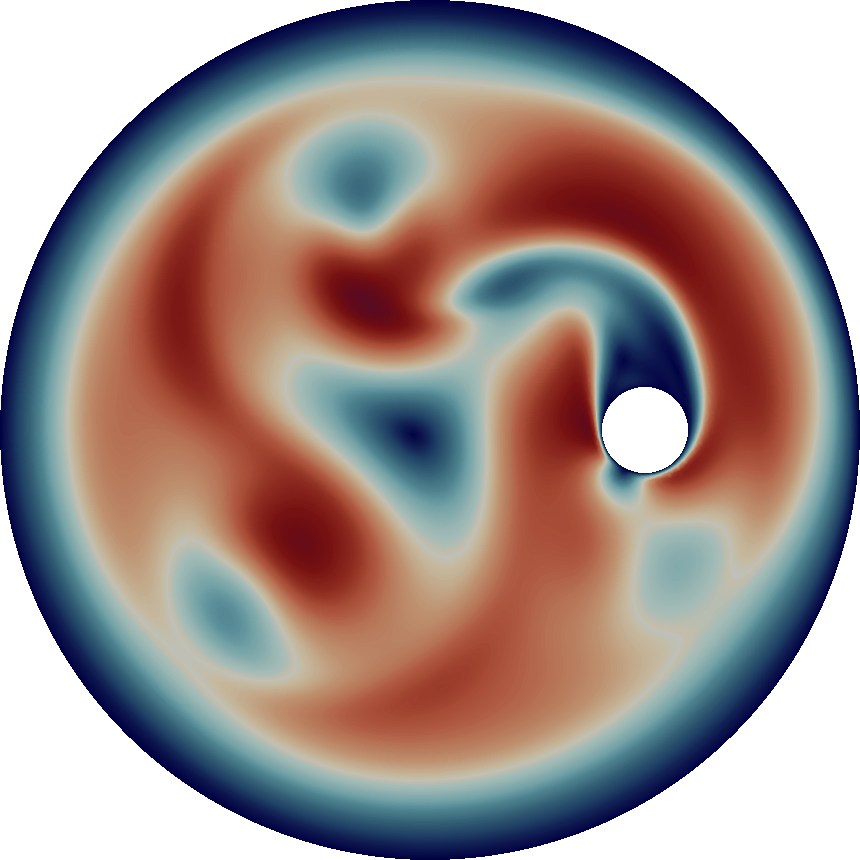} & \includegraphics[width = 1\linewidth]{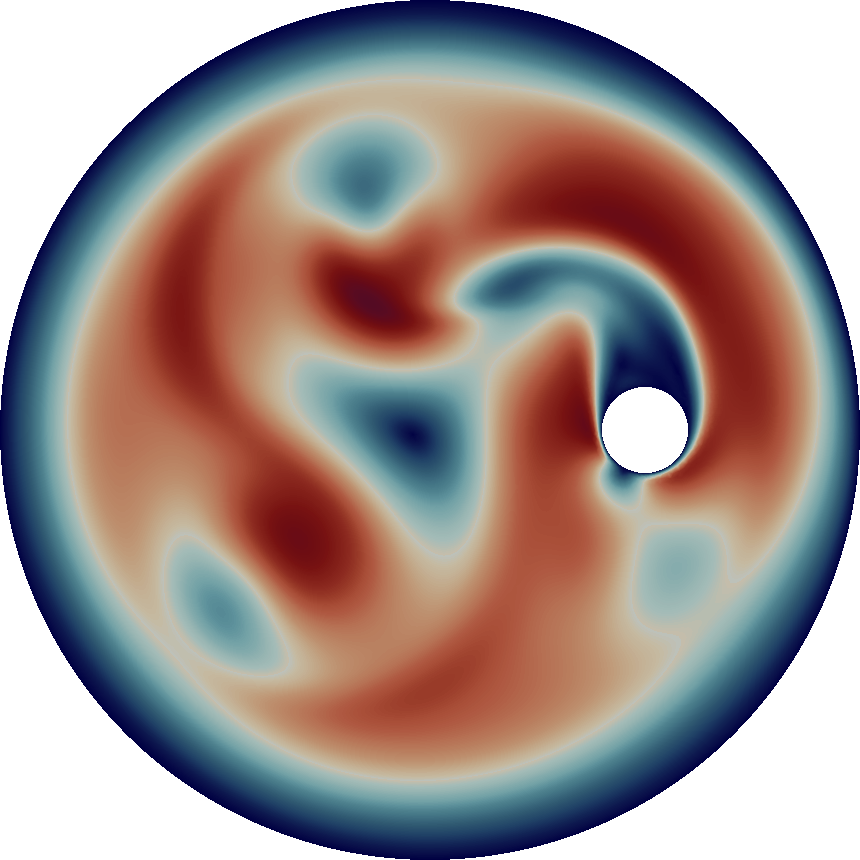} & \includegraphics[width = 1\linewidth]{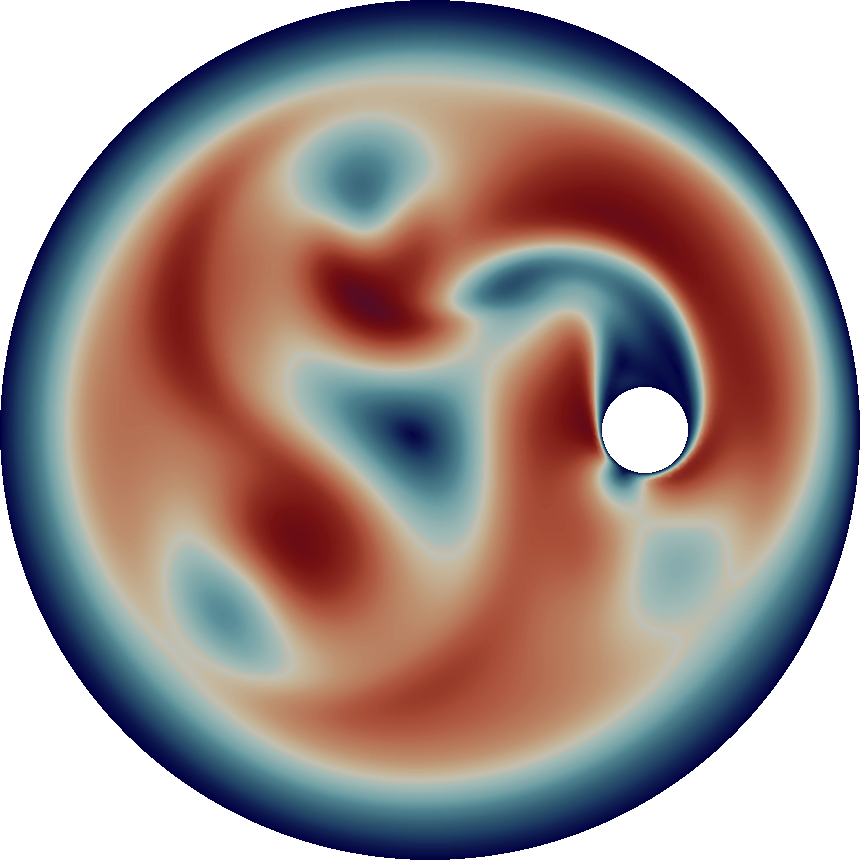} \\
            20 & \includegraphics[width = 1\linewidth]{figures/t-20.png} & \includegraphics[width = 1\linewidth]{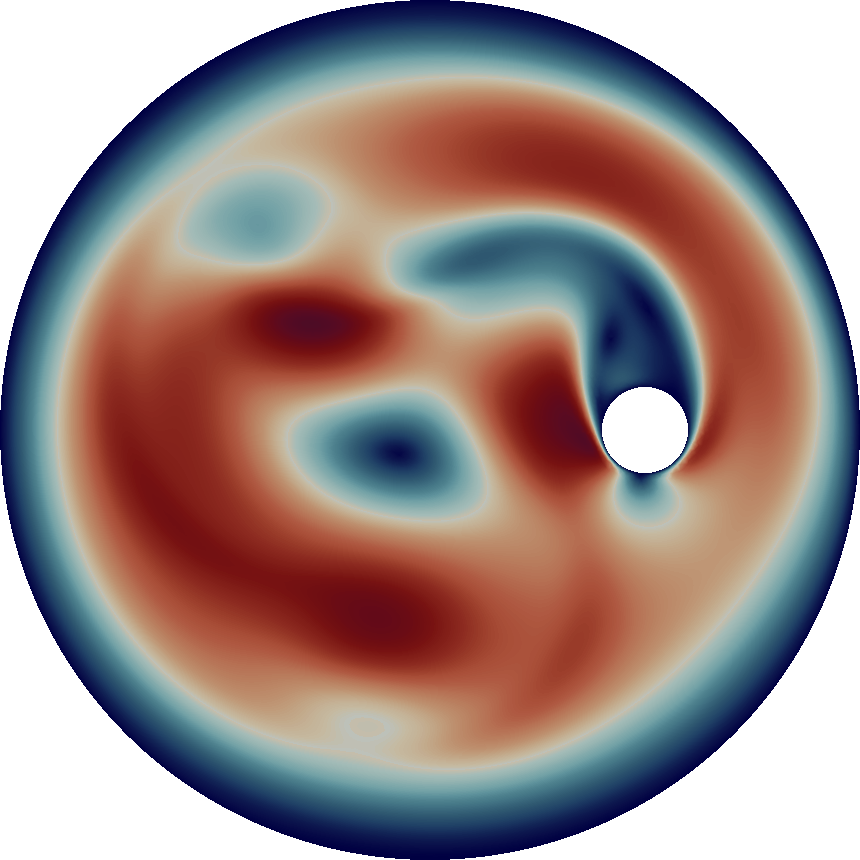} & \includegraphics[width = 1\linewidth]{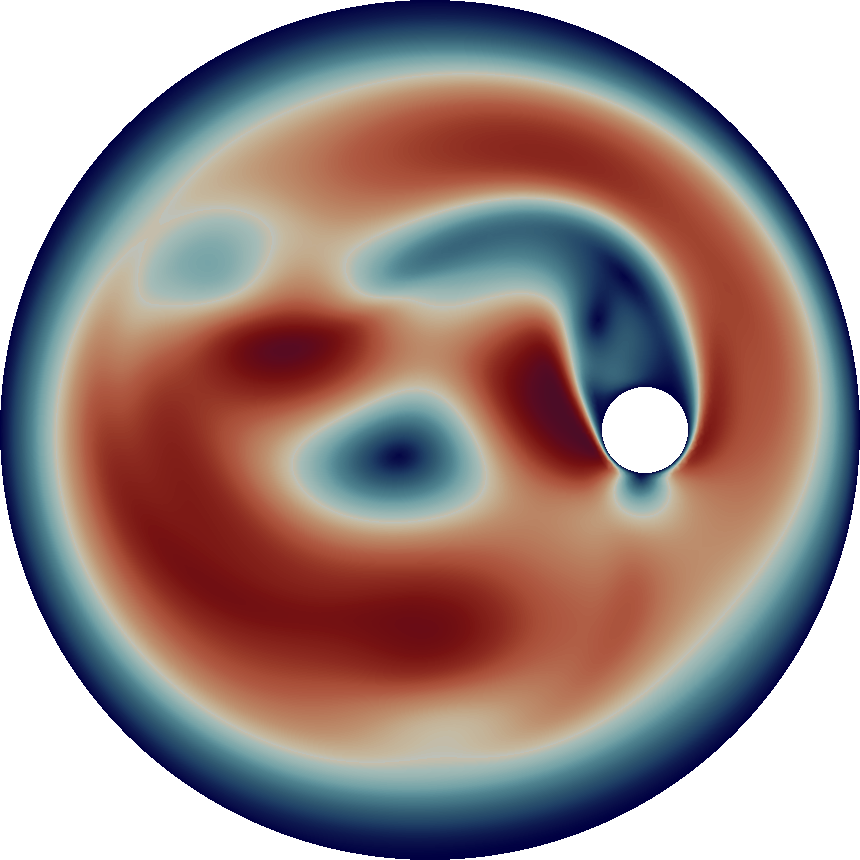} & \includegraphics[width = 1\linewidth]{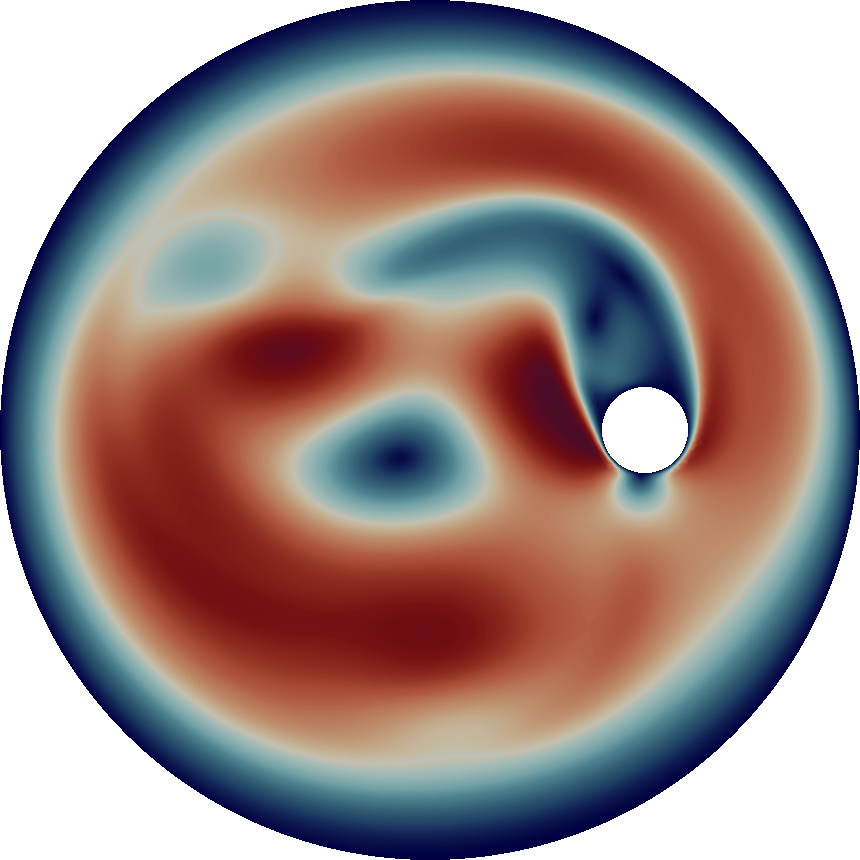} & \includegraphics[width = 1\linewidth]{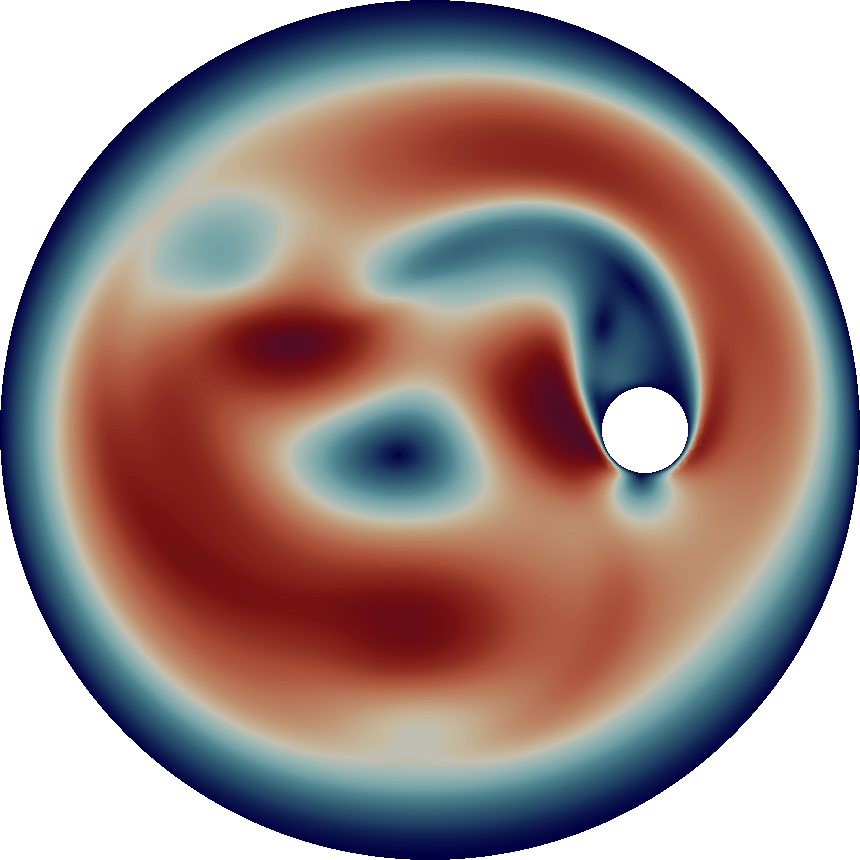} & \includegraphics[width = 1\linewidth]{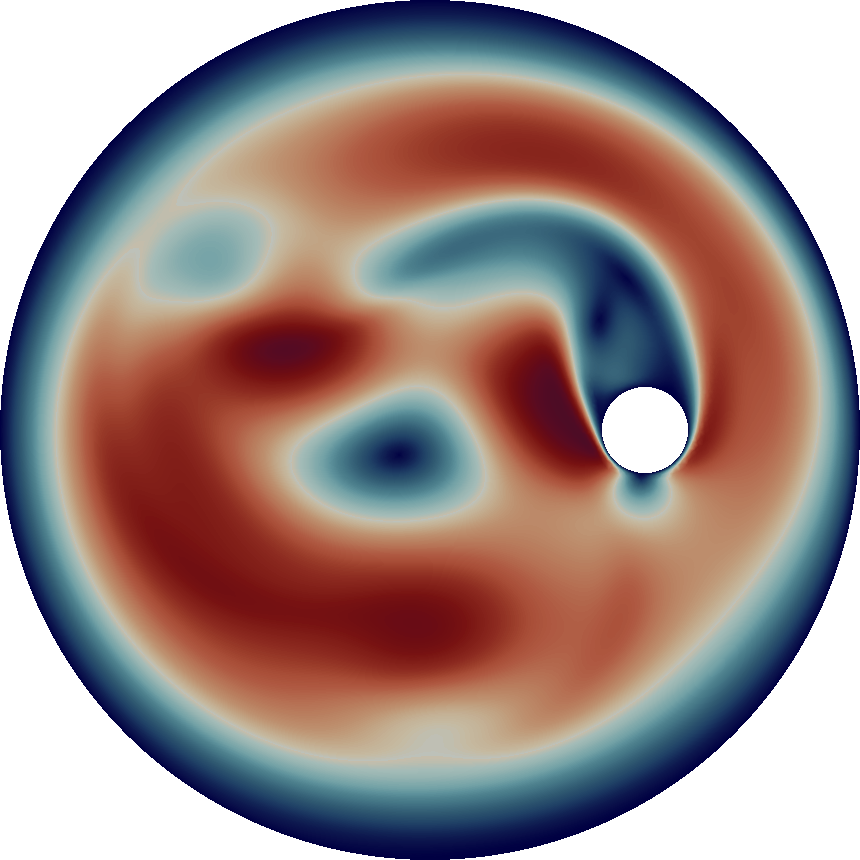} \\
            30 & \includegraphics[width = 1\linewidth]{figures/t-30.png} & \includegraphics[width = 1\linewidth]{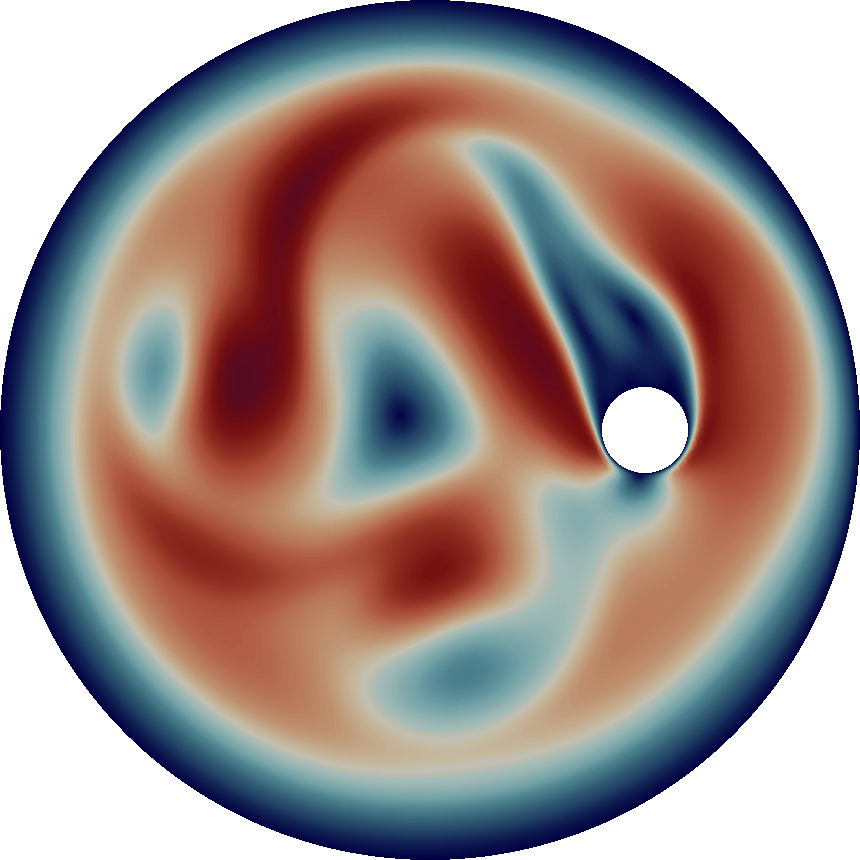} & \includegraphics[width = 1\linewidth]{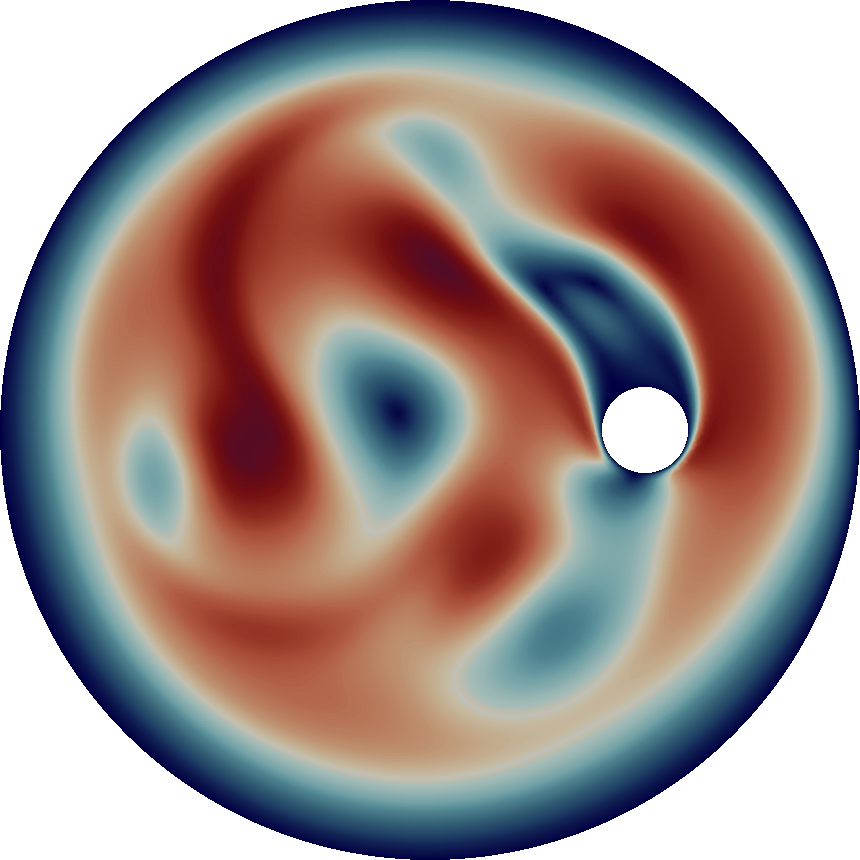} & \includegraphics[width = 1\linewidth]{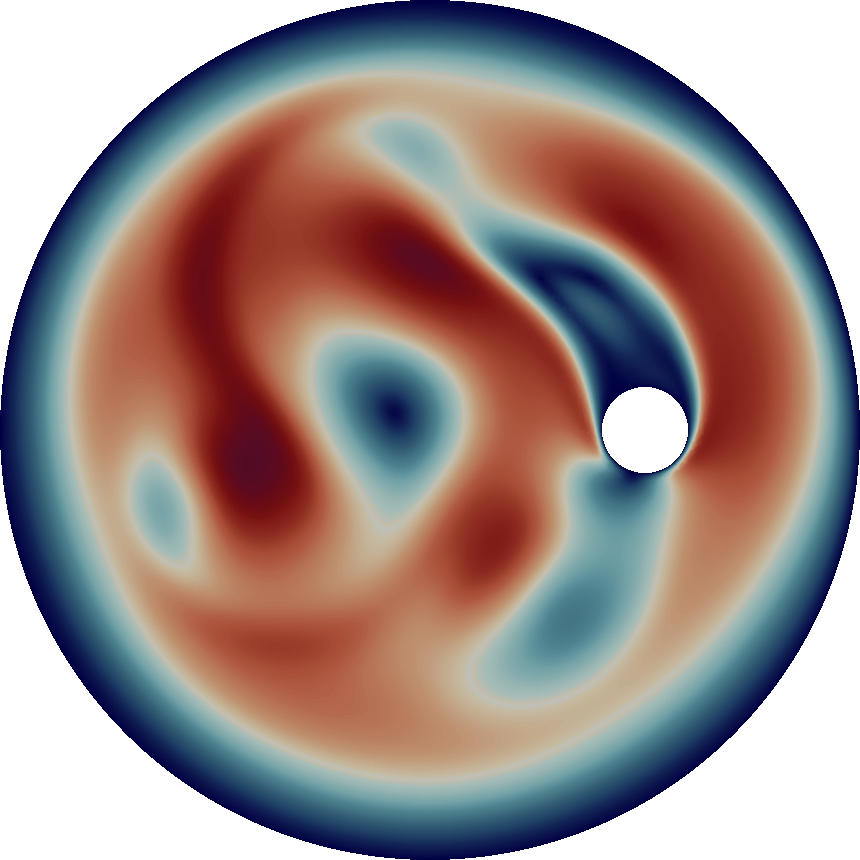} & \includegraphics[width = 1\linewidth]{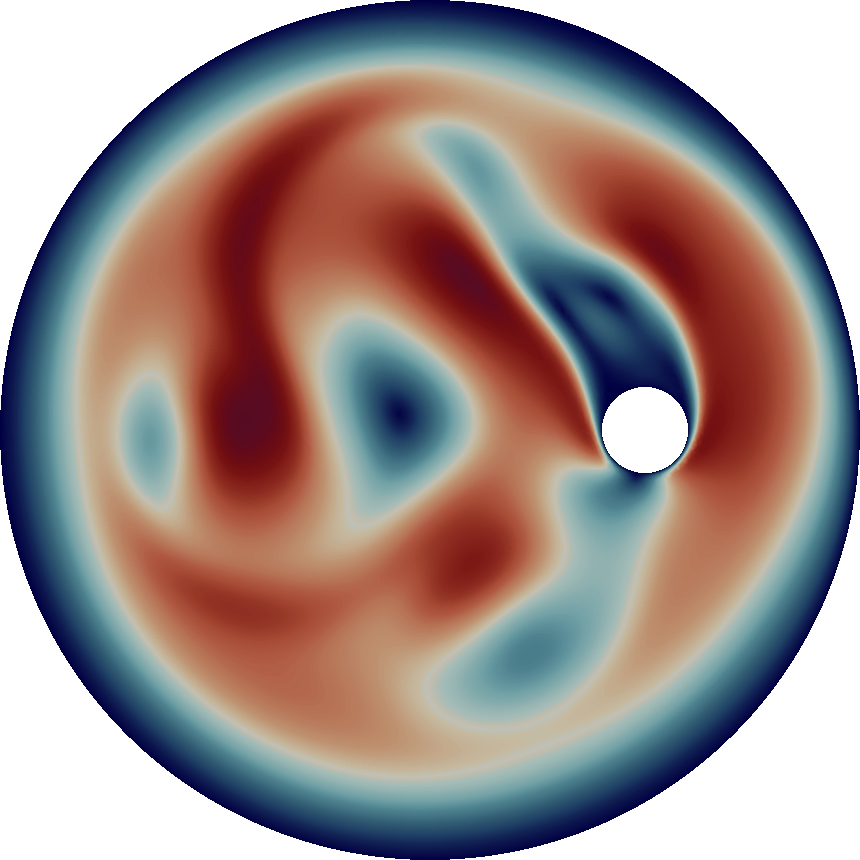} & \includegraphics[width = 1\linewidth]{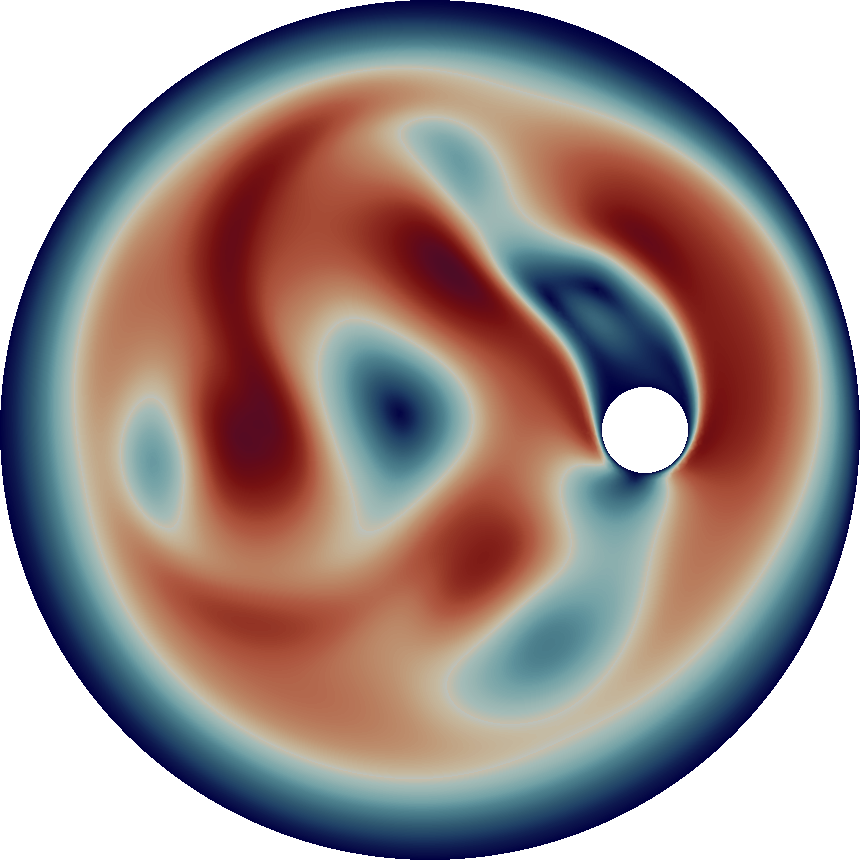}\\
            \bottomrule
        \end{tabular}
\caption{Snapshots of the BDF2{} and MP based methods with an effective $\Delta t =
0.0025$. Both of the third order methods show better agreement with the
reference solution than BDF2. For the MP based methods, the MP snapshots (not shown) were nearly identical to the MP-Pre-Post-2 snapshots. MP-Pre-Post-3 appears to most accurately capture the reference solution out of the methods in this table. The MP-Pre-Post-4 simulation failed before the first snapshot could be taken.}
\label{fig:bdf2_offset_cylinder}
\end{figure}

\section{Conclusions and future problems}

This work presented new time discretization methods with attractive properties and a new GLM framework for developing and 
analyzing time-filtered methods. Using the tools of GLM analysis
a number of interesting new methods have been derived and tested. 
These methods are easy to implement within existing codes, 
and can be used to improve the time accuracy of legacy codes.
The novel time-filtering methods have low cognitive complexity, 
and are in some cases optimized for other application-specific criteria.
Among these new methods, the error inhibiting method shows special
promise in tests even adjusting results for its added complexity.

In this work, we presented (mostly) a linear stability analysis and tested the methods
for a nonlinearity dominated application. The development of a general energy stability
theory of any new method herein would be a significant further advance. 
The new methods presented have an embedded structure, so that 
step and order adaptivity are natural  next steps. 
Our long term goal is to develop self-adaptive, variable step, variable order
methods that are of low cognitive complexity and easily implemented in
legacy codes based on these embedded families. 

\bigskip
\noindent{\bf Acknowledgements:}
The authors acknowledge support from AFOSR Grant No. FA9550-18-1-0383 (S.G.) and NSF Grant No. DMS1817542 (W.L.).

\appendix

\section{Order Conditions}
\label{app:OC} The order conditions are typically given in the form:

\begin{itemize}
\item[1.] Consistency conditions ($q=0$): 
\begin{equation*}
\tau_{0_1}= \Theta e_{k}- 1 \; \; \mbox{ and} \; \; \; \tau_{0_2} = \mathbf{%
\tilde{D}} e_k - e_{s+k-1}.
\end{equation*}

\item[2.] First order condition ($q=1$): 
\begin{equation*}
\tau_{1_1} = \mathbf{\tilde{b}} e_{s+k-1} + \Theta \ell - 1.
\end{equation*}

\item[3.] Second order condition ($q=2$): 
\begin{equation*}
\tau_{2_1} = \mathbf{\tilde{b}} \mathbf{c} + \frac{1}{2}\Theta \ell^2 - 
\frac{1}{2}.
\end{equation*}

\item[4.] Third order conditions ($q=3$): 
\begin{equation*}
\tau_{3_1} = \mathbf{\tilde{b}} \mathbf{c}^2 + \frac{1}{3}\Theta \ell^3 - 
\frac{1}{3} \; \; \; \mbox{and} \; \; \; \tau_{3_2} = \mathbf{\tilde{b}}^t 
\mathbf{\tilde{A}} \mathbf{c} +\frac{1}{2} \mathbf{\tilde{b}} \mathbf{\tilde{%
D}} \ell^2 + \frac{1}{6}\Theta \ell^3 - \frac{1}{6}.
\end{equation*}

\item[5.] Fourth order conditions ($q=4$): 
\begin{eqnarray*}
\tau _{4_{1}} &=&\mathbf{\tilde{b}}\mathbf{c}^{3}+\frac{1}{4}\Theta \ell
^{4}-\frac{1}{4},\\
\tau _{4_{2}} &=&\mathbf{\tilde{b}}\mathbf{\tilde{A}}c^{2}+\frac{1}{3}%
\mathbf{\tilde{b}}\mathbf{\tilde{D}}\ell ^{3}+\frac{1}{12}\Theta \ell ^{4}-%
\frac{1}{12},\\
\tau _{4_{3}}&=&\mathbf{\tilde{b}}\mathbf{\tilde{A}}%
\mathbf{\tilde{A}}\mathbf{c}+\frac{1}{2}\mathbf{\tilde{b}}\mathbf{\tilde{A}}%
\mathbf{\tilde{D}}\ell ^{2}+\frac{1}{6}\mathbf{\tilde{b}}\mathbf{\tilde{D}}%
\ell ^{3}+\frac{1}{24}\Theta \ell ^{4}-\frac{1}{24} \\
\tau _{4_{4}}&=&\mathbf{\tilde{b}}(\mathbf{c}\cdot \mathbf{%
\tilde{A}}\mathbf{c})+\frac{1}{2}\mathbf{\tilde{b}}(\mathbf{c}\cdot \mathbf{%
\tilde{D}}\ell ^{2})+\frac{1}{8}\Theta \ell ^{4}-\frac{1}{8}.
\end{eqnarray*}
\end{itemize}

\section{Energy stability Proof for the IE-Filt(d) methods}

The d-filtered family of methods for Implicit Euler methods IE-Filt(d) are given by
\begin{align*}
y^{(1)}& =du^{n-1}+(1-d)u^{n} \\
y^{(2)}& =y^{(1)}+\Delta t F(y^{(2)}) \\
u^{n+1}& =\frac{1}{3-2d}(2 y^{(2)}+2(1-d)u^{n}-u^{n-1})
\end{align*}%
We consider the energy at the  second stage and
rewrite $y^{(2)}$ in terms of $u^{n-1},u^{n}$ and $u^{n+1}$. 
\begin{equation*}
\langle y^{(2)},F(y^{(2)}\rangle \; =  \;
\langle y^{(2)},y^{(2)}-du^{n-1}+(1-d)u^{n}\rangle
\end{equation*}
where 
\[y^{(2)}=\frac{3-2d}{2}u^{n+1}+(d-1)u^{n}+\frac{1}{2}u^{n-1}.\]

If the operator F is dissipative then the inner product is negative, 
\[ \langle x,F(x) \rangle \; \; \; \leq  \; \; \; 0 \] 
and so we have
\[ \langle \frac{3-2d}{2}u^{n+1}+(d-1)u^{n}+\frac{1}{2}u^{n-1},\frac{3-2d}{2}
u^{n+1}+2(d-1)u^{n}+\frac{1-2d}{2}u^{n-1}\rangle \; \; \leq 0.\]
Expanding and re-arranging, we have
\begin{eqnarray} \label{energy1}
\frac{2d^{2}-7d+6}{4}\Vert u^{n+1}\Vert ^{2}-\frac{(2d-3)(d-1)}{2}%
\langle u^{n+1},u^{n}\rangle +\frac{2d^{2}-3d+2}{4}\Vert u^{n}\Vert ^{2} \nonumber  \\
 \leq \; \;  \frac{2d^{2}-7d+6}{4}\Vert u^{n}\Vert ^{2}-\frac{(2d-3)(d-1)}{2}
\langle u^{n},u^{n-1} \rangle \nonumber   \\
+ \;  \frac{2d^{2}-3d+2}{4} \Vert u^{n-1}\Vert ^{2} \; \; \; 
-\frac{(2d-3)(d-1)}{4}\Vert u^{n+1}-2u^{n}+u^{n-1}\Vert ^{2} . \; \; \;
\end{eqnarray}
\[ \mbox{Let} \; \; X^n = \left( \begin{array}{c}
 u^{n} \\
u^{n-1}\\
\end{array} \right) , \; \; 
\mbox{and so} \; \; \; 
X^{n+1} = \left( \begin{array}{c}
 u^{n+1} \\
u^{n} \\
\end{array} \right) ,\]
and define the G-norm, for any SPD matrix G:
\[ \Vert X\Vert _{G}^{2}=X^{t}GX. \]
In particular, we are interested in  the matrix 
\[ G=\frac{1}{4}
\begin{bmatrix}
2d^{2}-7d+6 & -(2d-3)(d-1) \\ 
-(2d-3)(d-1) & 2d^{2}-3d+2%
\end{bmatrix}\]
which is SPD for $d\leq \frac{3}{2}$.
Observe that eqn \eqref{energy1} becomes
\begin{equation*}
\Vert X^{n+1}\Vert _{G}^{2}\leq \Vert X^{n}\Vert _{G}^{2}-\frac{(2d-3)(d-1)}{%
4}\Vert u^{n+1}-2u^{n}+u^{n-1}\Vert ^{2}
\end{equation*}%
Clearly, then, the  d-filtered Implicit Euler method is energy stable for $0 \leq d \leq 1$. 

\section{Considerations for non-autonomous forces, time-dependent boundary conditions, and pressure recovery\label{sec:non-aut}}In this section, we list the abscissas and pressure recovery formulas for the methods developed herein.
While the methods were derived using autonomous theory for simplicity, the extension to non-autonomous ODEs is straightforward. The pre-processing steps can shift the abscissas of the data, so non-autonomous sources (such as $f(t)$ in \eqref{eq:NSE} or externally applied time dependent boundary conditions) must be adjusted accordingly.  

For the incompressible Navier-Stokes equations, pressure is not solved via an evolution equation, but the gradient of the pressure is a force in the evolution equation for velocity. Thus, the resulting pressure of the fully discrete methods are collocated at the same time as the non-autonomous forces. Some quantities of interest, such as lift and drag, require velocity and pressure approximations at the same time level, so pressure must be interpolated or extrapolated as necessary.
All the methods presented herein produce velocity approximations at time $t = t^{n+1}$ after post-processing. Thus, we explain how to derive a pressure approximation such that $p^{n+1} \approx p(t^{n+1})$.

Let $\tilde{p}^{n+1}$ denote the intermediate pressure from solving the coupled velocity-pressure system for $(u^{n+1},\tilde{p}^{n+1})$. Note that $\widetilde{p}^{n+1}$ is not necessarily an approximation at time level $t^{n+1}$.
Denote the abscissa for non-autonomous forces by $\tilde{t}$.

The formulas for IE, IE-Pre-2, IE-Pre-Post-3, IE-Filt(0), BDF2, BDF2-Post-3, and MP-Pre-Post-2/3/4   are
\begin{equation}\notag
\tilde{t} = t^n+ \Delta t,\quad p^{n+1} = \widetilde{p}^{n+1},
\end{equation}
the formulas for MP are
\begin{equation}\notag
\tilde{t} = t^n + \frac{1}{2}\Delta t,\quad p^{n+1} = \frac{3}{2}\widetilde{p}^{n+1} - \frac{1}{2}\widetilde{p}^{n}.
\end{equation}
the formulas for IE-Filt($d$) are
\begin{equation}\notag
\tilde{t} = t^n + (1-d)\Delta t,\quad  p^{n+1} = (1+d)\widetilde{p}^{n+1} - d\widetilde{p}^{n},
\end{equation}
and the formulas for BDF2-Pre-Post-3 are 
\begin{gather*}\notag
\tilde{t} = t^n + c \, \Delta t, \quad
p^{n+1} = a_3\widetilde{p}^{n+1} +a_2\tilde{p}^{n} + a_1 \widetilde{p}^{n-1}, \\
c = 3.930023404911324, \quad 
\begin{bmatrix}
a_1\\
a_2\\
a_3
\end{bmatrix}
=
\begin{bmatrix}
2.827506874208412\\
- 2.7249903435055\\
0.8974834692970881
\end{bmatrix}.
\end{gather*}
IE-EIS-3 contains two implicit Euler solves per timestep and yields pressure approximations at both $t^{n+2/3}$ and $t^{n+1}$. There is a pair $(\tilde{t}^{(2)}, \widetilde{p}^{(2)})$ associated with the intermediate solve in \eqref{eqn:eis3-solve-1}, and a pair $(\tilde{t}, \widetilde{p}^{n+1})$ associated with the final solve in \eqref{eqn:eis3-solve-2}. The formulas are simple;
\begin{equation}\notag
\tilde{t}^{(2)} = t^n + \frac{2}{3}\Delta t, \quad p^{n+2/3} = \widetilde{p}^{(2)}, \quad \tilde{t} = t^{n+1}, \quad p^{n+1} = \widetilde{p}^{n+1}.
\end{equation}

\end{document}